\documentclass[a4paper]{amsart}
\usepackage[foot]{amsaddr}

\usepackage[latin1]{inputenc}
\usepackage{amssymb,amsmath,amsthm}
\usepackage{amsfonts}
\usepackage{amsaddr}
\usepackage{enumitem}
\usepackage{booktabs}
\usepackage{ifthen}
\usepackage{tikz}
\usetikzlibrary{math,calc}
\usepackage{url}
\usepackage{hyphenat}
\usepackage{hyperref}
\usepackage{calc,etoolbox}

\theoremstyle{plain}
\newtheorem{theorem}{Theorem}[section]
\newtheorem{proposition}[theorem]{Proposition}
\newtheorem{lemma}[theorem]{Lemma}
\newtheorem{corollary}[theorem]{Corollary}

\theoremstyle{definition}
\newtheorem{definition}[theorem]{Definition}
\newtheorem{remark}[theorem]{Remark}

\newcommand{\IN}{\ensuremath{\mathbb{N}}}

\newcommand{\nset}[1]{\ensuremath{[{#1}]}}
\newcommand{\card}[1]{\ensuremath{\lvert{#1}\rvert}}
\newcommand{\gendefault}{}
\newcommand{\gen}[2][\gendefault]{\ensuremath{\langle{#2}\rangle_{#1}}}
\newcommand{\clonegen}[1]{\gen[]{#1}}
\newcommand{\vect}[1]{\ensuremath{\mathbf{#1}}}

\newcommand{\lhs}{\hspace{2em}&\hspace{-2em}}

\newcommand{\clIntVal}[3]{\ensuremath{#1_{\ifthenelse{\equal{#2}{}}{\mathord{*}}{#2}\ifthenelse{\equal{#3}{}}{\mathord{*}}{#3}}}}
\newcommand{\clIntEq}[1]{\ensuremath{#1_{=}}}
\newcommand{\clIntNeq}[1]{\ensuremath{#1_{\neq}}}
\newcommand{\clIntLeq}[1]{\ensuremath{#1_{\leq}}}
\newcommand{\clIntGeq}[1]{\ensuremath{#1_{\geq}}}
\newcommand{\clIntNeqOO}[1]{\ensuremath{#1_{{\neq},00}}}
\newcommand{\clIntNeqII}[1]{\ensuremath{#1_{{\neq},11}}}

\newcommand{\clAll}{\ensuremath{\mathsf{\Omega}}}
\newcommand{\clEioo}{\ensuremath{\clIntNeqII{\clAll}}}
\newcommand{\clEioi}{\ensuremath{\clIntGeq{\clAll}}}
\newcommand{\clEiio}{\ensuremath{\clIntLeq{\clAll}}}
\newcommand{\clEiii}{\ensuremath{\clIntNeqOO{\clAll}}}
\newcommand{\clEq}{\ensuremath{\clIntEq{\clAll}}}
\newcommand{\clNeq}{\ensuremath{\clIntNeq{\clAll}}}
\newcommand{\clOX}{\ensuremath{\clIntVal{\clAll}{0}{}}}
\newcommand{\clIX}{\ensuremath{\clIntVal{\clAll}{1}{}}}
\newcommand{\clXO}{\ensuremath{\clIntVal{\clAll}{}{0}}}
\newcommand{\clXI}{\ensuremath{\clIntVal{\clAll}{}{1}}}
\newcommand{\clOXC}{\ensuremath{\clOX \cup \clVak}}
\newcommand{\clIXC}{\ensuremath{\clIX \cup \clVak}}
\newcommand{\clXOC}{\ensuremath{\clXO \cup \clVak}}
\newcommand{\clXIC}{\ensuremath{\clXI \cup \clVak}}
\newcommand{\clOO}{\ensuremath{\clIntVal{\clAll}{0}{0}}}
\newcommand{\clII}{\ensuremath{\clIntVal{\clAll}{1}{1}}}
\newcommand{\clOI}{\ensuremath{\clIntVal{\clAll}{0}{1}}}
\newcommand{\clIO}{\ensuremath{\clIntVal{\clAll}{1}{0}}}
\newcommand{\clOOC}{\ensuremath{\clOO \cup \clVak}}
\newcommand{\clIIC}{\ensuremath{\clII \cup \clVak}}
\newcommand{\clOICO}{\ensuremath{\clOI \cup \clVako}}
\newcommand{\clOICI}{\ensuremath{\clOI \cup \clVaki}}
\newcommand{\clOIC}{\ensuremath{\clOI \cup \clVak}}
\newcommand{\clIOCO}{\ensuremath{\clIO \cup \clVako}}
\newcommand{\clIOCI}{\ensuremath{\clIO \cup \clVaki}}
\newcommand{\clIOC}{\ensuremath{\clIO \cup \clVak}}
\newcommand{\clS}{\ensuremath{\mathsf{S}}}
\newcommand{\clSc}{\ensuremath{\clIntVal{\clS}{0}{1}}}
\newcommand{\clScneg}{\ensuremath{\clIntVal{\clS}{1}{0}}}
\newcommand{\clSM}{\ensuremath{\mathsf{SM}}}
\newcommand{\clSMneg}{\ensuremath{\overline{\clSM}}}
\newcommand{\clSmin}{\ensuremath{\clS^{-}}}
\newcommand{\clSmaj}{\ensuremath{\clS^{+}}}
\newcommand{\clSminNeq}{\ensuremath{\clIntNeq{\clSmin}}}
\newcommand{\clSmajNeq}{\ensuremath{\clIntNeq{\clSmaj}}}
\newcommand{\clSminOX}{\ensuremath{\clIntVal{\clSmin}{0}{}}}
\newcommand{\clSminXO}{\ensuremath{\clIntVal{\clSmin}{}{0}}}
\newcommand{\clSminOO}{\ensuremath{\clIntVal{\clSmin}{0}{0}}}
\newcommand{\clSminOI}{\ensuremath{\clIntVal{\clSmin}{0}{1}}}
\newcommand{\clSminIO}{\ensuremath{\clIntVal{\clSmin}{1}{0}}}
\newcommand{\clSminOICO}{\ensuremath{\clSminOI \cup \clVako}}
\newcommand{\clSminIOCO}{\ensuremath{\clSminIO \cup \clVako}}
\newcommand{\clSmajIX}{\ensuremath{\clIntVal{\clSmaj}{1}{}}}
\newcommand{\clSmajXI}{\ensuremath{\clIntVal{\clSmaj}{}{1}}}
\newcommand{\clSmajOI}{\ensuremath{\clIntVal{\clSmaj}{0}{1}}}
\newcommand{\clSmajIO}{\ensuremath{\clIntVal{\clSmaj}{1}{0}}}
\newcommand{\clSmajII}{\ensuremath{\clIntVal{\clSmaj}{1}{1}}}
\newcommand{\clSmajOICI}{\ensuremath{\clSmajOI \cup \clVaki}}
\newcommand{\clSmajIOCI}{\ensuremath{\clSmajIO \cup \clVaki}}
\newcommand{\clM}{\ensuremath{\mathsf{M}}}
\newcommand{\clMo}{\ensuremath{\clIntVal{\clM}{0}{}}}
\newcommand{\clMi}{\ensuremath{\clIntVal{\clM}{}{1}}}
\newcommand{\clMc}{\ensuremath{\clIntVal{\clM}{0}{1}}}
\newcommand{\clMneg}{\ensuremath{\overline{\clM}}}
\newcommand{\clMoneg}{\ensuremath{\clIntVal{\clMneg}{1}{}}}
\newcommand{\clMineg}{\ensuremath{\clIntVal{\clMneg}{}{0}}}
\newcommand{\clMcneg}{\ensuremath{\clIntVal{\clMneg}{1}{0}}}
\newcommand{\clUk}[1]{\ensuremath{\mathsf{U}^{#1}}}
\newcommand{\clU}{\ensuremath{\clUk{2}}}
\newcommand{\clTcU}{\ensuremath{\clIntVal{\clU}{0}{1}}}
\newcommand{\clTcUk}[1]{\ensuremath{\clIntVal{\clUk{#1}}{0}{1}}}
\newcommand{\clTcUCO}{\ensuremath{\clTcU \cup \clVako}}
\newcommand{\clMU}{\ensuremath{\mathsf{M}\clU}}
\newcommand{\clMcU}{\ensuremath{\clIntVal{\clMU}{0}{1}}}
\newcommand{\clMUk}[1]{\ensuremath{\mathsf{M}\clUk{#1}}}
\newcommand{\clMcUk}[1]{\ensuremath{\clIntVal{\clMUk{#1}}{0}{1}}}
\newcommand{\clWk}[1]{\ensuremath{\mathsf{W}^{#1}}}
\newcommand{\clW}{\ensuremath{\clWk{2}}}
\newcommand{\clTcW}{\ensuremath{\clIntVal{\clW}{0}{1}}}
\newcommand{\clTcWk}[1]{\ensuremath{\clIntVal{\clWk{#1}}{0}{1}}}
\newcommand{\clTcWCI}{\ensuremath{\clTcW \cup \clVaki}}
\newcommand{\clMW}{\ensuremath{\mathsf{M}\clW}}
\newcommand{\clMcW}{\ensuremath{\clIntVal{\clMW}{0}{1}}}
\newcommand{\clMWk}[1]{\ensuremath{\mathsf{M}\clWk{#1}}}
\newcommand{\clMcWk}[1]{\ensuremath{\clIntVal{\clMWk{#1}}{0}{1}}}
\newcommand{\clUneg}{\ensuremath{\overline{\clU}}}
\newcommand{\clTcUneg}{\ensuremath{\clIntVal{\clUneg}{1}{0}}}
\newcommand{\clTcUnegCI}{\ensuremath{\clTcUneg \cup \clVaki}}
\newcommand{\clMUneg}{\ensuremath{\overline{\clMU}}}
\newcommand{\clMcUneg}{\ensuremath{\clIntVal{\clMUneg}{1}{0}}}
\newcommand{\clWneg}{\ensuremath{\overline{\clW}}}
\newcommand{\clTcWneg}{\ensuremath{\clIntVal{\clWneg}{1}{0}}}
\newcommand{\clTcWnegCO}{\ensuremath{\clTcWneg \cup \clVako}}
\newcommand{\clMWneg}{\ensuremath{\overline{\clMW}}}
\newcommand{\clMcWneg}{\ensuremath{\clIntVal{\clMWneg}{1}{0}}}
\newcommand{\clUOO}{\ensuremath{\clIntVal{\clU}{0}{0}}}
\newcommand{\clWII}{\ensuremath{\clIntVal{\clW}{1}{1}}}
\newcommand{\clUnegII}{\ensuremath{\clIntVal{\clUneg}{1}{1}}}
\newcommand{\clWnegOO}{\ensuremath{\clIntVal{\clWneg}{0}{0}}}
\newcommand{\clUWneg}{\ensuremath{\clU \cap \clWneg}}
\newcommand{\clWUneg}{\ensuremath{\clW \cap \clUneg}}
\newcommand{\clRefl}{\ensuremath{\mathsf{R}}}  
\newcommand{\clReflOO}{\ensuremath{\clIntVal{\clRefl}{0}{0}}}
\newcommand{\clReflII}{\ensuremath{\clIntVal{\clRefl}{1}{1}}}
\newcommand{\clReflOOC}{\ensuremath{\clReflOO \cup \clVak}}
\newcommand{\clReflIIC}{\ensuremath{\clReflII \cup \clVak}}
\newcommand{\clVak}{\ensuremath{\mathsf{C}}}
\newcommand{\clVaka}[1]{\ensuremath{\clVak_{#1}}}
\newcommand{\clVako}{\ensuremath{\clVaka{0}}}
\newcommand{\clVaki}{\ensuremath{\clVaka{1}}}
\newcommand{\clEmpty}{\ensuremath{\mathsf{\emptyset}}}

\newcommand{\clL}{\ensuremath{\mathsf{L}}}

\newcommand{\clLc}{\ensuremath{\clIntVal{\clL}{0}{1}}}

\newcommand{\clLambda}{\ensuremath{\mathsf{\Lambda}}}
\newcommand{\clV}{\ensuremath{\mathsf{V}}}
\newcommand{\clUinf}{\ensuremath{\clUk{\infty}}}
\newcommand{\clWinf}{\ensuremath{\clWk{\infty}}}

\newcommand{\clIstar}{\ensuremath{\mathsf{I}^{*}}}
\newcommand{\clIc}{\ensuremath{\mathsf{J}}}

\newcommand{\id}{\ensuremath{\mathrm{id}}}
\newcommand{\cf}[2]{\ensuremath{\mathrm{c}^{({#1})}_{#2}}}

\DeclareMathOperator{\var}{var}
\DeclareMathOperator{\pr}{pr}
\newcommand{\arity}[1]{\ensuremath{\mathrm{ar}({#1})}}
\newcommand{\Str}[1]{\ensuremath{\mathrm{Str}({#1})}}
\newcommand{\closys}[1]{\ensuremath{\mathcal{L}_{#1}}}
\newcommand{\clProj}[1]{\ensuremath{\mathsf{J}_{#1}}}

\begin{document}
\title{Majority-closed minions of Boolean functions}

\author{Erkko Lehtonen}

\address%
   {Centro de Matem\'atica e Aplica\c{c}\~oes \\
    Faculdade de Ci\^encias e Tecnologia \\
    Universidade Nova de Lisboa \\
    Quinta da Torre \\
    2829-516 Caparica \\
    Portugal}

\date{\today}

\begin{abstract}
The 93 minions of Boolean functions stable under left composition with the clone of self\hyp{}dual monotone functions are described.
As an easy consequence, all $(C_1,C_2)$\hyp{}stable classes of Boolean functions are determined for an arbitrary clone $C_1$ and for any clone $C_2$ containing the clone of self\hyp{}dual monotone functions.
\end{abstract}

\maketitle


\section{Introduction}

We consider functions of several arguments from a set $A$ to another set $B$, that is, mappings of the form $f \colon A^n \to B$ for some positive integer $n$.
Given such a function $f$, the functions that can be obtained from $f$ by manipulation of arguments -- permutation of arguments, introduction or deletion of fictitious arguments, and identification of arguments -- are called \emph{minors} of $f$.
Classes of functions closed under formation of minors are called \emph{minor\hyp{}closed classes} or \emph{minions.}

Minions have been investigated from different points of view in the past decades.
In universal algebra, minions arise naturally as sets of operations induced by the terms of height $1$ on an algebra.
As an analogue of the classical Galois connection $\operatorname{Pol}$--$\operatorname{Inv}$ that characterizes clones, minor\hyp{}closed classes were characterized by Pippenger~\cite{Pippenger} as the closed classes of the Galois connection induced by the preservation relation between functions and relation pairs.
Minors and minions have recently emerged and played an important role in the analysis of the complexity of constraint satisfaction problems (CSP), especially in a new variant known as promise CSP (see the survey article by Barto, Bul\'in, Krokhin, and Opr\v{s}al \cite{BarBulKroOpr}).

Minions may satisfy additional closure conditions.
A well\hyp{}known example of this idea are clones, that is, classes of operations that contain all projections and are closed under composition.
Aichinger and Mayr \cite{AicMay} defined \emph{clonoids} as minor\hyp{}closed classes that are also stable under left composition with the operations of an algebra $\mathbf{B} = (B,F)$ (and hence under compositions with the clone of $\mathbf{B}$) and used them as a tool in showing that there is no infinite ascending chain of subvarieties of a finitely generated variety with an edge term.
In full generality, we can consider classes of functions that are stable under right compositions with a clone $C_1$ on $A$ and under left compositions with a clone $C_2$ on $B$, in brief, \emph{$(C_1,C_2)$\hyp{}stable} classes.
Couceiro and Foldes~\cite{CouFol-2009} characterized $(C_1,C_2)$\hyp{}stable classes with a specialization of Pippenger's Galois connection between functions and relation pairs where the two relations are limited to invariants of the clones $C_1$ and $C_2$.

This inevitably leads us to the problem of describing the lattice of $(C_1,C_2)$\hyp{}stable classes -- both the structure of the lattice and the classes themselves.
Let us denote by $\closys{(C_1,C_2)}$ the lattice of $(C_1,C_2)$\hyp{}stable classes.
It is known that there are uncountably many minions, even when $\card{A} = \card{B} = 2$ (in which case there are only countably many clones); hence an explicit description of all minions may be unattainable.
On the other hand, for some pairs of clones $C_1$ and $C_2$, there may be only a countable or finite number of $(C_1,C_2)$\hyp{}stable classes, and it may be possible to describe all of them.
A possible starting point for a systematical study of lattices of $(C_1,C_2)$\hyp{}stable classes is suggested by a recent result due to Sparks~\cite{Sparks-2019} that provides us with the cardinality of the lattice of clonoids with a two\hyp{}element target algebra.
It is noteworthy in this case that whether the closure system is finite, countably infinite, or uncountable depends only on the clone of the target algebra and not on the set $A$ (as long as $\card{A} > 1$).

Recall that an $n$\hyp{}ary operation $f \in \mathcal{O}_B$ with $n \geq 3$ is called a \emph{near\hyp{}unanimity operation} if $f(x, \dots, x, y, x, \dots, x) = x$ for all $x, y \in B$, where the single occurrence of $y$ can occur in any of the $n$ argument positions.
A ternary near\hyp{}unanimity operation is called a \emph{majority operation.}
A ternary operation $f \in \mathcal{O}_B$ is called a \emph{Mal'cev operation} if $f(y, y, x) = f(x, y, y) = x$ for all $x, y \in B$.

\begin{theorem}[{Sparks \cite[Theorem~1.3]{Sparks-2019}}]
\label{thm:Sparks}
Let $A$ be a finite set with $\card{A} > 1$, and let $B := \{0,1\}$.
Denote by $\clProj{A}$ the clone of projections on $A$, and let $C$ be a clone on $B$.
Then the following statements hold.
\begin{enumerate}[label={\upshape{(\roman*)}}]
\item\label{thm:Sparks:finite}
$\closys{(\clProj{A},C)}$ is finite if and only if $C$ contains a near\hyp{}unanimity operation.
\item\label{thm:Sparks:countable}
$\closys{(\clProj{A},C)}$ is countably infinite if and only if $C$ contains a Mal'cev operation but no majority operation.
\item\label{thm:Sparks:uncountable}
$\closys{(\clProj{A},C)}$ has the cardinality of the continuum if and only if $C$ contains neither a near\hyp{}unanimity operation nor a Mal'cev operation.
\end{enumerate}
\end{theorem}

\begin{figure}
\begin{center}
\scalebox{0.31}{%
\tikzstyle{every node}=[circle, draw, fill=black, scale=1, font=\LARGE]
\begin{tikzpicture}[baseline, scale=1]
   \node [label = below:$\clIc$] (Ic) at (0,-1) {};
   \node [label = above:$\clIstar$] (Istar) at (0,0.5) {};
   \node (I0) at (4.5,0.5) {};
   \node (I1) at (-4.5,0.5) {};
   \node (I) at (0,2) {};
   \node (Omega1) at (0,5) {};
   \node [label = below:$\clLc$] (Lc) at (0,7.5) {};
   \node (LS) at (0,9) {};
   \node (L0) at (3,9) {};
   \node (L1) at (-3,9) {};
   \node [label = above:$\clL$] (L) at (0,10.5) {};
   \node [label = below:$\clSM\,\,$] (SM) at (0,13.5) {};
   \node [label = left:$\clSc$] (Sc) at (0,15) {};
   \node [label = above:$\clS$] (S) at (0,16.5) {};
   \node [label = below:$\clMc$] (Mc) at (0,23) {};
   \node [label = left:$\clMo\,\,$] (M0) at (2,24) {};
   \node [label = right:$\,\,\clMi$] (M1) at (-2,24) {};
   \node [label = above:$\clM\,\,$] (M) at (0,25) {};
   \node (Lamc) at (7.2,6.7) {};
   \node (Lam1) at (5,7.5) {};
   \node (Lam0) at (8.7,7.5) {};
   \node [label = below:$\clLambda$] (Lam) at (6.5,8.3) {};
   \node (McUi) at (7.2,11.5) {};
   \node (MUi) at (8.7,13) {};
   \node (TcUi) at (10.2,12) {};
   \node [label = right:$\clUinf$] (Ui) at (11.7,13.5) {};
   \node (McU3) at (7.2,16) {};
   \node (MU3) at (8.7,17.5) {};
   \node (TcU3) at (10.2,16.5) {};
   \node (U3) at (11.7,18) {};
   \node [label = left:$\clMcU\,$] (McU2) at (7.2,19) {};
   \node [label = left:$\clMU\,$] (MU2) at (8.7,20.5) {};
   \node [label = right:$\clTcU$] (TcU2) at (10.2,19.5) {};
   \node [label = right:$\clU$] (U2) at (11.7,21) {};
   \node (Vc) at (-7.2,6.7) {};
   \node (V0) at (-5,7.5) {};
   \node (V1) at (-8.7,7.5) {};
   \node [label = below:$\clV$] (V) at (-6.5,8.3) {};
   \node (McWi) at (-7.2,11.5) {};
   \node (MWi) at (-8.7,13) {};
   \node (TcWi) at (-10.2,12) {};
   \node [label = left:$\clWinf$] (Wi) at (-11.7,13.5) {};
   \node (McW3) at (-7.2,16) {};
   \node (MW3) at (-8.7,17.5) {};
   \node (TcW3) at (-10.2,16.5) {};
   \node (W3) at (-11.7,18) {};
   \node [label = right:$\,\,\clMcW$] (McW2) at (-7.2,19) {};
   \node [label = right:$\clMW$] (MW2) at (-8.7,20.5) {};
   \node [label = left:$\clTcW$] (TcW2) at (-10.2,19.5) {};
   \node [label = left:$\clW$] (W2) at (-11.7,21) {};
   \node [label = above:$\clOI$] (Tc) at (0,28) {};
   \node [label = right:$\clOX$] (T0) at (5,29.5) {};
   \node [label = left:$\clXI$] (T1) at (-5,29.5) {};
   \node [label = above:$\clAll$] (Omega) at (0,31) {};
   \draw [thick] (Ic) -- (Istar) to[out=135,in=-135] (Omega1);
   \draw [thick] (I) -- (Omega1);
   \draw [thick] (Omega1) to[out=135,in=-135] (L);
   \draw [thick] (Ic) -- (I0) -- (I);
   \draw [thick] (Ic) -- (I1) -- (I);
   \draw [thick] (Ic) to[out=128,in=-134] (Lc);
   \draw [thick] (Ic) to[out=58,in=-58] (SM);
   \draw [thick] (I0) -- (L0);
   \draw [thick] (I1) -- (L1);
   \draw [thick] (Istar) to[out=60,in=-60] (LS);
   \draw [thick] (Ic) -- (Lamc);
   \draw [thick] (I0) -- (Lam0);
   \draw [thick] (I1) -- (Lam1);
   \draw [thick] (I) -- (Lam);
   \draw [thick] (Ic) -- (Vc);
   \draw [thick] (I0) -- (V0);
   \draw [thick] (I1) -- (V1);
   \draw [thick] (I) -- (V);
   \draw [thick] (Lamc) -- (Lam0) -- (Lam);
   \draw [thick] (Lamc) -- (Lam1) -- (Lam);
   \draw [thick] (Lamc) -- (McUi);
   \draw [thick] (Lam0) -- (MUi);
   \draw [thick] (Lam1) -- (M1);
   \draw [thick] (Lam) -- (M);
   \draw [thick] (Vc) -- (V0) -- (V);
   \draw [thick] (Vc) -- (V1) -- (V);
   \draw [thick] (Vc) -- (McWi);
   \draw [thick] (V0) -- (M0);
   \draw [thick] (V1) -- (MWi);
   \draw [thick] (V) -- (M);
   \draw [thick] (McUi) -- (TcUi) -- (Ui);
   \draw [thick] (McUi) -- (MUi) -- (Ui);
   \draw [thick,loosely dashed] (McUi) -- (McU3);
   \draw [thick,loosely dashed] (MUi) -- (MU3);
   \draw [thick,loosely dashed] (TcUi) -- (TcU3);
   \draw [thick,loosely dashed] (Ui) -- (U3);
   \draw [thick] (McU3) -- (TcU3) -- (U3);
   \draw [thick] (McU3) -- (MU3) -- (U3);
   \draw [thick] (McU3) -- (McU2);
   \draw [thick] (MU3) -- (MU2);
   \draw [thick] (TcU3) -- (TcU2);
   \draw [thick] (U3) -- (U2);
   \draw [thick] (McU2) -- (TcU2) -- (U2);
   \draw [thick] (McU2) -- (MU2) -- (U2);
   \draw [thick] (McU2) -- (Mc);
   \draw [thick] (MU2) -- (M0);
   \draw [thick] (TcU2) to[out=120,in=-25] (Tc);
   \draw [thick] (U2) -- (T0);
   \draw [thick] (McWi) -- (TcWi) -- (Wi);
   \draw [thick] (McWi) -- (MWi) -- (Wi);
   \draw [thick,loosely dashed] (McWi) -- (McW3);
   \draw [thick,loosely dashed] (MWi) -- (MW3);
   \draw [thick,loosely dashed] (TcWi) -- (TcW3);
   \draw [thick,loosely dashed] (Wi) -- (W3);
   \draw [thick] (McW3) -- (TcW3) -- (W3);
   \draw [thick] (McW3) -- (MW3) -- (W3);
   \draw [thick] (McW3) -- (McW2);
   \draw [thick] (MW3) -- (MW2);
   \draw [thick] (TcW3) -- (TcW2);
   \draw [thick] (W3) -- (W2);
   \draw [thick] (McW2) -- (TcW2) -- (W2);
   \draw [thick] (McW2) -- (MW2) -- (W2);
   \draw [thick] (McW2) -- (Mc);
   \draw [thick] (MW2) -- (M1);
   \draw [thick] (TcW2) to[out=60,in=-155] (Tc);
   \draw [thick] (W2) -- (T1);
   \draw [thick] (SM) -- (McU2);
   \draw [thick] (SM) -- (McW2);
   \draw [thick] (Lc) -- (LS) -- (L);
   \draw [thick] (Lc) -- (L0) -- (L);
   \draw [thick] (Lc) -- (L1) -- (L);
   \draw [thick] (Lc) to[out=120,in=-120] (Sc);
   \draw [thick] (LS) to[out=60,in=-60] (S);
   \draw [thick] (L0) -- (T0);
   \draw [thick] (L1) -- (T1);
   \draw [thick] (L) to[out=125,in=-125] (Omega);
   \draw [thick] (SM) -- (Sc) -- (S);
   \draw [thick] (Sc) to[out=142,in=-134] (Tc);
   \draw [thick] (S) to[out=42,in=-42] (Omega);
   \draw [thick] (Mc) -- (M0) -- (M);
   \draw [thick] (Mc) -- (M1) -- (M);
   \draw [thick] (Mc) to[out=120,in=-120] (Tc);
   \draw [thick] (M0) -- (T0);
   \draw [thick] (M1) -- (T1);
   \draw [thick] (M) to[out=55,in=-55] (Omega);
   \draw [thick] (Tc) -- (T0) -- (Omega);
   \draw [thick] (Tc) -- (T1) -- (Omega);
\end{tikzpicture}
}
\end{center}
\caption{Post's lattice.}
\label{fig:Post}
\end{figure}
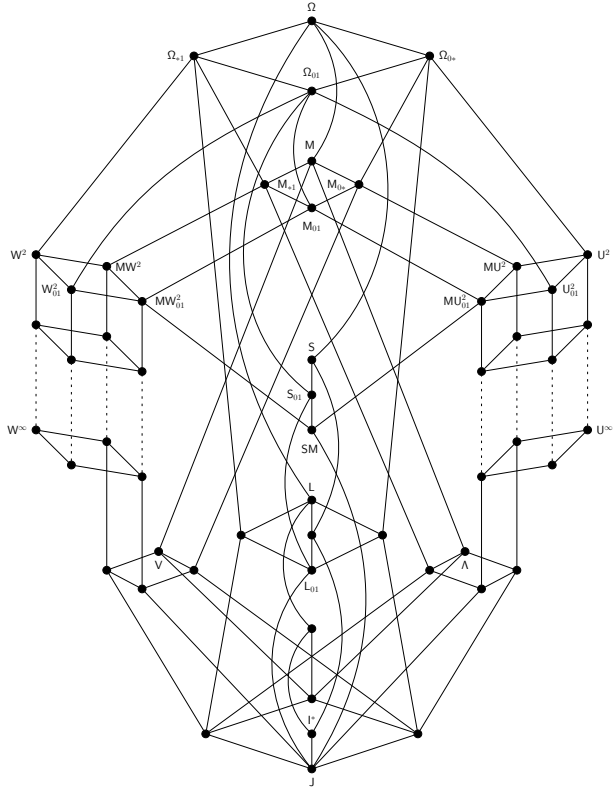

As a first attempt of describing lattices of $(C_1,C_2)$\hyp{}stable classes, we consider classes of Boolean functions (operations on $\{0,1\}$); this is the simplest nontrivial case, and Theorem~\ref{thm:Sparks} is applicable.
In this paper we focus on clones containing the majority operation $\mu$ on $\{0,1\}$ (the lattice of clones on $\{0,1\}$ is shown in Figure~\ref{fig:Post}).
Recall that $\mu$ generates the clone $\clSM$ of self\hyp{}dual monotone functions, and denote the clone of projections by $\clIc$.
Since $\mu$ is a near\hyp{}unanimity operation,
it follows from Sparks's result that there are only a finite number of $(\clIc,C)$\hyp{}stable classes of Boolean functions for any clone $C$ containing $\mu$.
Our goal is to refine this result and to explicitly describe the $(\clIc,C)$\hyp{}stable classes for every clone $C$ such that $\clSM \subseteq C$.

Regarding statement \ref{thm:Sparks:countable} of Theorem~\ref{thm:Sparks},
a clone $C$ on $\{0,1\}$ contains a Mal'cev operation but no majority operation if and only if $\clLc \subseteq C \subseteq \clL$, where $\clLc$ and $\clL$ denote the clone of idempotent linear functions and the clone of all linear functions, respectively.
This situation was completely described in \cite{CouLeh-Lcstability}; the case where $C = \clonegen{\mathord{+}}$ was settled earlier by Kreinecker~\cite[Theorem~3.12]{Kreinecker}.
As for statement \ref{thm:Sparks:uncountable}, a clone $C$ contains neither a near\hyp{}unanimity operation nor a Mal'cev operation if and only if $C$ is contained in one of the following clones:
$\clLambda = \clonegen{\mathord{\wedge}, 0, 1}$, $\clV = \clonegen{\mathord{\vee}, 0, 1}$, $\clIstar = \clonegen{\mathord{\neg}, 0}$, $\clWinf = \clonegen{\mathord{\rightarrow}}$, $\clUinf = \clonegen{\mathord{\nrightarrow}}$.

This paper is organized as follows.
In Section~\ref{sec:preliminaries}, we provide the basic notions and preliminary results that will be needed in the remaining sections.
Furthermore, we develop some tools that are applicable to the study of $(C_1,C_2)$\hyp{}stable classes also in a more general setting.
In Section~\ref{sec:Boolean}, we define properties of Boolean functions that are needed for presenting our results.
Section~\ref{sec:IcSM} is dedicated to our main result and its proof: a complete description of the $(\clIc,\clSM)$\hyp{}stable classes of Boolean functions.
The proof has two parts. Firstly, we show that the given classes are $(\clIc,\clSM)$\hyp{}stable; this is straightforward verification.
The more difficult part of the proof is to show that there are no further $(\clIc,\clSM)$\hyp{}stable classes.
From the description of the $(\clIc,\clSM)$\hyp{}stable classes, we can determine rather easily the $(C_1,C_2)$\hyp{}stable classes for clones $C_1$ and $C_2$, where $C_1$ is arbitrary and $\clSM \subseteq C_2$; this is done in Section~\ref{sec:C1C2}.
We conclude the paper with some comments on topics for further research in Section~\ref{sec:concluding}.


\section{Preliminaries}
\label{sec:preliminaries}

\subsection{General}

The set of nonnegative integers and the set of positive integers are denoted by $\IN$ and $\IN_{+}$, respectively.
For $n \in \IN$, let $\nset{n} := \{ \, i \in \IN \mid 1 \leq i \leq n \, \}$.

We denote tuples by bold letters and their components by the corresponding italic letters, e.g., $\vect{a} = (a_1, \dots, a_n)$.
Since an $n$\hyp{}tuple $\vect{a}$ is formally a mapping $\vect{a} \colon \nset{n} \to A$, we may compose $\vect{a}$ with a map $\sigma \colon \nset{m} \to \nset{n}$, and the resulting map $\vect{a} \circ \sigma \colon \nset{m} \to A$ is the $m$\hyp{}tuple $\vect{a} \circ \sigma = (a_{\sigma(1)}, \dots, a_{\sigma(m)})$; we write simply $\vect{a} \sigma$ for $\vect{a} \circ \sigma$.

We identify $n$\hyp{}tuples over $A$ with words of length $n$ over $A$.
For $a \in A$ and $n \in \IN$, $a^n$ stands for the word consisting of $n$ copies of $a$.

\subsection{Functions of several arguments, function class composition, minors, and stability}

\renewcommand{\gendefault}{(C_1,C_2)}

Let $A$ and $B$ be sets.
We consider \emph{functions of several arguments} from $A$ to $B$, that is, mappings $f \colon A^n \to B$ for some number $n \in \IN_{+}$ called the \emph{arity} of $f$.
Functions of several arguments from $A$ to $A$ are called \emph{operations} on $A$.
We denote the set of all functions of several arguments from $A$ to $B$ by $\mathcal{F}_{AB}$ and the set of all operations on $A$ by $\mathcal{O}_A$.
For any $C \subseteq \mathcal{F}_{AB}$ and $n \in \IN$, we denote by $C^{(n)}$ the set of all $n$\hyp{}ary members of $C$.
In particular, $\mathcal{F}^{(n)}_{AB}$ and $\mathcal{O}^{(n)}_A$ are the sets of all $n$\hyp{}ary functions of several arguments from $A$ to $B$ and the set of all $n$\hyp{}ary operations on $A$, respectively.

For $n \in \IN_{+}$ and $i \in \nset{n}$, the $i$\hyp{}th $n$\hyp{}ary \emph{projection} on $A$ is the operation $\pr^{(n)}_i \colon A^n \to A$, $(a_1, \dots, a_n) \mapsto a_i$.
The first unary projection $\pr^{(1)}_1$ is thus the identity map on $A$ and is also denoted by $\id_A$ or simply by $\id$ if the set $A$ is clear from the context.
We denote the set of all projections on $A$ by $\clProj{A}$.

Let $A$, $B$, and $C$ be nonempty sets, and let $f \in \mathcal{F}^{(n)}_{BC}$ and $g_1, \dots, g_n \in \mathcal{F}^{(m)}_{AB}$.
The \emph{composition} of $f$ with $g_1, \dots, g_n$ is the function $f(g_1, \dots, g_n) \in \mathcal{F}^{(m)}_{AC}$ defined by the rule $f(g_1, \dots, g_n)(\vect{a}) := f(g_1(\vect{a}), \dots, g_n(\vect{a}))$ for all $\vect{a} \in A^m$.
The notion of composition extends to function classes as follows.
Let $I \subseteq \mathcal{F}_{BC}$ and $J \subseteq \mathcal{F}_{AB}$.
The \emph{composition} of $I$ with $J$ is the class $IJ \subseteq \mathcal{F}_{AC}$ given by
\[
IJ := \{ \, f(g_1, \dots, g_n) \mid n, m \in \IN, \, f \in I^{(n)}, \, g_1, \dots, g_n \in J^{(m)} \, \}.
\]

Let $f \in \mathcal{F}^{(n)}_{AB}$, and let $\clProj{A}$ be the clone of projections on $A$.
The functions in $\{f\} \clProj{A}$ are called \emph{minors} of $f$.
Thus a function $g \in \mathcal{F}^{(m)}_{AB}$ is a minor of $f$ if and only if $g = f(\pr_{\sigma(1)}^{(m)}, \dots, \pr_{\sigma(n)}^{(m)})$ for some $\sigma \colon \nset{n} \to \nset{m}$.
We use the shorthand $f_\sigma$ for $f(\pr_{\sigma(1)}^{(m)}, \dots, \pr_{\sigma(n)}^{(m)})$.
Thus $f_\sigma(\vect{a}) = f(\vect{a} \sigma)$.
For a map $\sigma \colon \nset{n} \to \nset{m}$, we define the map $\underline{\sigma} \colon A^m \to A^n$ by $\underline{\sigma}(\vect{a}) := \vect{a} \sigma$.
Then we can also write $f_\sigma = f \circ \underline{\sigma}$.

As a further notational tool, we may specify a map $\sigma \colon \nset{n} \to \nset{m}$ by the word $\sigma(1) \sigma(2) \dots \sigma(n)$; thus we may write $f_{\sigma(1) \sigma(2) \dots \sigma(n)}$ for $f_\sigma$.
The arity of the minor is not explicit in this notation but will be clear from the context.

A set $C \subseteq \mathcal{O}_A$ is called a \emph{clone} on $A$ if $\clProj{A} \subseteq C$ and $C C \subseteq C$.
The set of all clones on $A$ constitutes a closure system. The smallest and the greatest clones on $A$ are the clone $\clProj{A}$ of projections and the clone $\mathcal{O}_A$ of all operations.
We denote by $\clonegen{F}$ the clone generated by $F$, i.e., the smallest clone on $A$ that contains $F$.

Let $F \subseteq \mathcal{F}_{AB}$, and let $C_1$ and $C_2$ be clones on $A$ and $B$, respectively.
We say that $F$ is \emph{stable under right composition with $C_1$} if $F C_1 \subseteq F$,
and we say that $F$ is \emph{stable under left composition with $C_2$} if $C_2 F \subseteq F$.
We say that $F$ is \emph{$(C_1,C_2)$\hyp{}stable} if it is stable under right composition with $C_1$ and stable under left composition with $C_2$.
The set of all $(C_1,C_2)$\hyp{}stable classes constitutes a closure system.
We denote by $\gen{F}$ the $(C_1,C_1)$\hyp{}stable class generated by $F$, i.e., the smallest $(C_1,C_2)$\hyp{}stable class containing $F$.

\begin{lemma}
\label{lem:stable-impl-stable}
Let $C_1$ and $C'_1$ be clones on $A$ and let $C_2$ and $C'_2$ by clones on $B$.
If $C_1 \subseteq C'_1$ and $C_2 \subseteq C'_2$, then every $(C'_1,C'_2)$\hyp{}stable class is $(C_1,C_2)$\hyp{}stable.
\end{lemma}

\begin{proof}
Let $F \subseteq \mathcal{F}_{AB}$, and assume that $F$ is $(C'_1,C'_2)$\hyp{}stable.
We have $F C_1 \subseteq F C'_1 \subseteq F$ and $C_2 F \subseteq C'_2 F \subseteq F$, so $F$ is $(C_1,C_2)$\hyp{}stable.
\end{proof}

\begin{lemma}
\label{lem:clones-stable}
Let $C$ and $K$ be clones on $A$.
Then the following statements hold:
\begin{enumerate}[label=\upshape{(\roman*)}]
\item\label{lem:clones-stable:R} $K C \subseteq K$ if and only if $C \subseteq K$.
\item\label{lem:clones-stable:L} $C K \subseteq K$ if and only if $C \subseteq K$.
\end{enumerate}
\end{lemma}

\begin{proof}
If $C \subseteq K$, then $K C \subseteq K K \subseteq K$ and $C K \subseteq K K \subseteq K$.
Assume now that $C \nsubseteq K$.
Then there exists an $f \in C$ such that $f \notin K$.
Since $\id \in K$, we have $f = \id(f) \in K C$ and $f = f(\id_1, \dots, \id_n) \in C K$.
Therefore $K C \nsubseteq K$ and $C K \nsubseteq K$.
\end{proof}

\begin{lemma}[{Couceiro, Foldes \cite[Associativity Lemma]{CouFol-2007,CouFol-2009}}]
\label{lem:associativity}
Let $A$, $B$, $C$, and $D$ be arbitrary nonempty sets, and let $I \subseteq \mathcal{F}_{CD}$, $J \subseteq \mathcal{F}_{BC}$, and $K \subseteq \mathcal{F}_{AB}$.
The following statements hold.
\begin{enumerate}[label=\upshape{(\roman*)}]
\item $(IJ)K \subseteq I(JK)$.
\item If $J$ is stable under right composition with the clone of projections on $B$ (i.e., minor\hyp{}closed), then $(IJ)K = I(JK)$.
\end{enumerate}
\end{lemma}

\begin{lemma}
\label{lem:F-closure}
Let $F \subseteq \mathcal{F}_{AB}$, and let $C_1$ and $C_2$ be clones on $A$ and $B$, respectively.
Then $\gen{F} = C_2 (F C_1)$.
\end{lemma}

\begin{proof}
Since the clones $C_1$ and $C_2$ contain all projections, we have that $F \subseteq F C_1 \subseteq C_2 (F C_1)$.
Applying the Associativity Lemma (Lemma~\ref{lem:associativity}), we see that $C_2 (F C_1)$ is $(C_1,C_2)$\hyp{}stable (note that clones are minor\hyp{}closed and closed under composition):
\begin{gather*}
(C_2 (F C_1)) C_1 \subseteq C_2 ((F C_1) C_1) \subseteq C_2( F (C_1 C_1)) \subseteq C_2 (F C_1), \\
C_2 ( C_2 (F C_1)) = (C_2 C_2)(F C_1) \subseteq C_2 (F C_1).
\end{gather*}
Therefore $C_2 (F C_1)$ is a $(C_1,C_2)$\hyp{}stable class containing $F$; hence the inclusion $\gen{F} \subseteq C_2 (F C_1)$ holds.
The converse inclusion follows immediately from the $(C_1,C_2)$\hyp{}stability of $\gen{F}$:
\[
C_2 (F C_1) \subseteq C_2 (\gen{F} C_1) \subseteq C_2 \gen{F} \subseteq \gen{F}.
\qedhere
\]
\end{proof}

The following two helpful lemmata were proved in \cite{CouLeh-Lcstability}.
Here we employ the binary composition operation $\ast$ defined as follows:
if $f \in \mathcal{O}_A^{(n)}$ and $g \in \mathcal{O}_A^{(m)}$, then $f \ast g \in \mathcal{O}_A^{(m+n-1)}$ is defined by
\[
(f \ast g)(a_1, \dots, a_{m+n-1}) :=
f(g(a_1, \dots, a_m), a_{m+1}, \dots, a_{m+n-1}),
\]
for all $a_1, \dots, a_{m+n-1} \in A$.

\begin{lemma}
\label{lem:right-stab-gen}
Let $F \subseteq \mathcal{O}_A$.
Let $C$ be a clone on $A$, and let $G$ be a generating set of $C$.
Then the following conditions are equivalent.
\begin{enumerate}[label=\upshape{(\roman*)}]
\item\label{FCsubF} $F C \subseteq F$
\item\label{minorC} $F$ is minor\hyp{}closed and $f \ast g \in F$ whenever $f \in F$ and $g \in C$.
\item\label{minorG} $F$ is minor\hyp{}closed and $f \ast g \in F$ whenever $f \in F$ and $g \in G$.
\end{enumerate}
\end{lemma}

\begin{lemma}
\label{lem:left-stab-gen}
Let $F \subseteq \mathcal{O}_A$.
Let $C$ be a clone on $A$, and let $G$ be a generating set of $C$.
Then the following conditions are equivalent.
\begin{enumerate}[label=\upshape{(\roman*)}]
\item\label{CFsubF} $C F \subseteq F$
\item\label{gf-C} $g(f_1, \dots, f_n) \in F$ whenever $g \in C^{(n)}$ and $f_1, \dots, f_n \in F^{(m)}$ for some $n, m \in \IN$.
\item\label{gf-G} $g(f_1, \dots, f_n) \in F$ whenever $g \in G^{(n)}$ and $f_1, \dots, f_n \in F^{(m)}$ for some $n, m \in \IN$.
\end{enumerate}
\end{lemma}


\subsection{Terms and algebras, stratified terms}

Let $X = \{x_1, x_2, \dots\}$ be a countably infinite set of \emph{variables,} and for each $n \in \IN$, let $X_n := \{x_1, \dots, x_n\}$.
Let $F$ be a set of \emph{function symbols,} and assume that $F$ is disjoint from $X$.
Assign to each $f \in F$ a natural number $\arity{f}$, called the \emph{arity} of $f$.
The map $\tau \colon F \to \IN$, $f \mapsto \arity{f}$, is called an (\emph{algebraic similarity}) \emph{type} (also known as a \emph{signature}).
\emph{Terms} of type $\tau$ over $X$ are defined as usual:
every variable $x_i \in X$ is a term of type $\tau$, and
if $f \in F$ is an $n$\hyp{}ary function symbol and $t_1, \dots, t_n$ are terms of type $\tau$, then $f(t_1, \dots, t_n)$ is a term of type $\tau$.
Terms are thus certain well\hyp{}formed words over the alphabet comprising the variables $X$, the function symbols $F$, and some punctuation (parentheses and comma).
Any subword of a term $t$ that is itself a term is called a \emph{subterm} of $t$.
If a term $t$ is of the form $f(t_1, \dots, t_n)$, we call it \emph{functional;}
$f$ is called the \emph{leading function symbol} of $t$, and the terms $t_1, \dots, t_n$ are called the \emph{immediate subterms} of $t$.
We call terms of type $\tau$ over $X$ also \emph{$F$\hyp{}terms}.
We denote the set of all $F$\hyp{}terms by $T(F)$.

The \emph{height} of a term $t$, denoted by $h(t)$, is defined recursively as follows:
variables have height $0$ ($h(x_i) := 0$ for every $x_i \in X$), and if
$t = f(t_1, \dots, t_n)$, then $h(t) := \max(h(t_1), \dots, h(t_n)) + 1$.

The set of variables occurring in a term $t$ is denoted by $\var(t)$.
If $t$ is a term with $\var(t) \subseteq X_n$ and $t_1, \dots, t_n$ are terms, then we define $t[t_1, \dots, t_n]$ as the term obtained by substitution of $t_i$ for $x_i$ in $t$, i.e., by simultaneously replacing each occurrence of $x_i$ in $t$ by the term $t_i$, for each $i \in \nset{n}$.

Let $G$ and $F$ be sets of function symbols, not necessarily disjoint.
A term $t \in T(G \cup F)$ is \emph{$(G,F)$\hyp{}stratified} if the following conditions hold:
\begin{enumerate}[label=(\roman*)]
\item for every subterm $t'$ of $t$ with $h(t') = 1$, the leading functional symbol of $t'$ belongs to $F$;
\item for every subterm $t'$ of $t$ with $h(t') \geq 2$, the leading functional symbol of $t'$ belongs to $G$ and no immediate subterm of $t'$ is a variable.
\end{enumerate}
In other words, $(G,F)$\hyp{}stratified terms are exactly the terms of the form $t[t_1,\dots,t_n]$ for some $t \in T(G)$ with $\var(t) \subseteq X_n$, $n \in \IN_{+}$, and $t_1, \dots, t_n \in T(F)$ such that $h(t_i) = 1$ for all $i \in \nset{n}$.
We denote by $\Str{G,F}$ the set of all $(G,F)$\hyp{}stratified terms in $T(G \cup F)$.

Given a type $\tau \colon F \to \IN$, an \emph{algebra} of type $\tau$ is a pair $\mathbf{A} = (A, F^\mathbf{A})$, where $A$ is a nonempty set called the \emph{universe} and $F^\mathbf{A} = (f^\mathbf{A})_{f \in F}$ is an indexed family of operations on $A$, where for each $f \in F$, the operation $f^\mathbf{A}$ has arity $\arity{f}$.
For an $n$\hyp{}ary term $t$ of type $\tau$ and an algebra $\mathbf{A} = (A,F)$ of type $\tau$, the \emph{term operation} induced by $t$ on $\mathbf{A}$, denoted by $t^\mathbf{A}$, is the $n$\hyp{}ary operation on $A$ that is defined recursively as follows.
If $t = x_i \in X_n$, then $t^\mathbf{A} := \pr^{(n)}_i$.
If $t = f(t_1, \dots, t_n)$, where $f$ is an $n$\hyp{}ary operation symbol and $t_1, \dots, t_n$ are terms, then $t^\mathbf{A} := f^\mathbf{A}(t_1^\mathbf{A}, \dots, t_n^\mathbf{A})$.
We extend the notation to sets of terms: if $T \subseteq T(F)$, then $T^\mathbf{A} := \{ \, t^\mathbf{A} \mid t \in T \, \}$.

Henceforth, we will regard functions as function symbols, and we give standard interpretation to terms of the corresponding type.
More precisely, we will assume that $F$ is a set of operations on $A$, and we consider the algebraic similarity type $\tau \colon F \to \IN$, $f \mapsto \arity{f}$, and the algebra $\mathbf{A} = (A,F^\mathbf{A})$, where $f^\mathbf{A} = f$ for each $f \in F$.
We will then consider term operations induced by $F$\hyp{}terms on $\mathbf{A}$.

It is well known that the set of term operations of an algebra is a clone, i.e., if $F \subseteq \mathcal{O}_A$ and $\mathbf{A} = (A,F)$, then $\clonegen{F} = T(F)^\mathbf{A}$.
In fact, every clone arises in this way.
We now present an analogous description of the term operations induced by stratified terms, which will be called \emph{stratified term operations.}

\begin{lemma}
\label{lem:gen-str}
Let $G$ and $F$ be sets of operations on $A$, let $C = \clonegen{G}$, and let $\clProj{A}$ be the clone of projections on $A$.
Then $\gen[(\clProj{A},C)]{F} = \Str{G,F}^\mathbf{A}$.
\end{lemma}

\begin{proof}
Assume first that $f \in \gen[(\clProj{A},C)]{F}$.
By Lemma~\ref{lem:F-closure}, $f \in C (F \clProj{A})$, so there are functions $g \in C$ and $h_1, \dots, h_n \in F \clProj{A}$ such that $f = g(h_1, \dots, h_n)$.
For each $i \in \nset{n}$, the function $h_i$ is a minor of some function from $F$, so there exists a terms $t_i \in T(F)$ of height $1$ representing $h_i$.
Since $g$ is a member of the clone generated by $G$, there exists a term $t \in T(G)$ representing $g$.
Then $t[t_1,\dots,t_n]$ is a $(G,F)$\hyp{}stratified term representing $g(h_1, \dots, h_n) = f$, so $f \in \Str{G,F}^\mathbf{A}$.

Conversely, assume that $f \in \Str{G,F}^\mathbf{A}$.
Then there exists a $(G,F)$\hyp{}stratified term $t[t_1,\dots,t_n]$ representing $f$; here $t \in T(G)$ and $t_1, \dots, t_n \in T(F)$ are terms of height $1$.
Then $t^\mathbf{A} \in \clonegen{G} = C$ and $t_1^\mathbf{A}, \dots, t_n^\mathbf{A} \in F \clProj{A}$, so $f = t^\mathbf{A}(t_1^\mathbf{A}, \dots, t_n^\mathbf{A}) \in C (F \clProj{A}) = \gen[(\clProj{A},C)]{F}$ by Lemma~\ref{lem:F-closure}.
\end{proof}

Stratified term operations may fail to constitute a clone.
The following two lemmata describe useful situations where this nevertheless happens.

\begin{lemma}
\label{lem:str-gen}
Let $G$ and $F$ be sets of operations on $A$, let $C = \clonegen{G}$ and $C' = \clonegen{G \cup F}$.
Then $\gen[(\clProj{A},C)]{F} = C'$ if and only if for every term $t \in T(G \cup F)$ there exists a $(G,F)$\hyp{}stratified term $t' \in \Str{G,F}$ such that $t^\mathbf{A} = (t')^\mathbf{A}$.
\end{lemma}

\begin{proof}
Assume first that $\gen[(\clProj{A},C)]{F} = C'$,
We have $C' = T(G \cup F)^\mathbf{A}$ and $\gen[(\clProj{A},C)]{F} = \Str{G,F}^\mathbf{A}$ by Lemma~\ref{lem:gen-str}, so clearly for every $t \in T(G \cup F)$ there exists $t' \in \Str{G,F}$ such that $t^\mathbf{A} = (t')^\mathbf{A}$.

Assume now that for every $t \in T(G \cup F)$ there exists a $t' \in \Str{G,F}$ such that $t^\mathbf{A} = (t')^\mathbf{A}$.
Then $T(G \cup F)^\mathbf{A} \subseteq \Str{G,F}^\mathbf{A}$.
Since $\Str{G,F} \subseteq T(G \cup F)$, the converse inclusion $\Str{G,F}^\mathbf{A} \subseteq T(G \cup F)^\mathbf{A}$ clearly holds.
Then $\gen[(\clProj{A},C)]{F} = \Str{G,F}^\mathbf{A} = T(G \cup F)^\mathbf{A} = C'$.
\end{proof}

\begin{lemma}
\label{lem:str-terms}
Let $G$ and $F$ be sets of operations on $A$.
If
\begin{enumerate}[label=\upshape{(\roman*)}]
\item $\id_A \in F$, and
\item for every $\varphi \in F$ \textup{(}$n$\hyp{}ary\textup{)}, $\alpha \in G \cup F$ \textup{(}$m$\hyp{}ary\textup{)}, and $i \in \nset{n}$, there exists a term $t \in \Str{G,F}$ equivalent to
\[
\varphi(x_1, \dots, x_{i-1}, \alpha(x_i, \dots, x_{i+m-1}), x_{i+m}, \dots, x_{m+n-1})
\]
such that no variable is repeated in any subterm of $t$ of height $1$,
\end{enumerate}
then for every term $t \in T(G \cup F)$ there exists a term $t' \in \Str{G,F}$ such that $t^\mathbf{A} = (t')^\mathbf{A}$;
consequently $\gen[(\clProj{A},C)]{F} = \clonegen{G \cup F}$.
\end{lemma}

\begin{proof}
We prove the claim by induction on the number of function symbols in a term $t \in T(G \cup F)$.
If $t$ has no function symbol, then $t = x_i \in X$, and we can choose $t' := \id_A(x_i)$.
If $t$ has one function symbol, then $t = f(x_{i_1}, \dots, x_{i_n})$ for some $f \in G \cup F$.
If $f \in F$, then $t \in \Str{G,F}$ and we are done.
If $f \in G$, then $t$ is equivalent to the term $f(\id_A(x_{i_1}), \dots, \id_A(x_{i_n}))$, which is $(G,F)$\hyp{}stratified.

Assume that the claim is true for every term with at most $k$ function symbols ($k \geq 1$).
Consider now the case when $t$ has $k + 1$ function symbols.
Then $t = f(t_1, \dots, t_n)$ for some $f \in G \cup F$.
By the induction hypothesis, for every $i \in \nset{n}$, there exists a term $t'_i \in \Str{G,F}$ such that $t_i^\mathbf{A} = (t'_i)^\mathbf{A}$.
If $f \in G$, then $t' := f(t'_1, \dots, t'_n) \in \Str{G,F}$ and $t^\mathbf{A} = (t')^\mathbf{A}$.

Assume now that $f \in F$.
Since $t$ contains at least two function symbols, at least one of the immediate subterms of $t$ is functional, say $t_i = \alpha(u_1, \dots, u_m)$ for some $\alpha \in G \cup F$.
Let
\[
s := f(x_1, \dots, x_{i-1}, \alpha(x_i, \dots, x_{i+m-1}), x_{i+m}, \dots, x_{m+n-1}).
\]
By our assumptions, there exists a term $s' \in \Str{G,F}$ such that $s^\mathbf{A} = (s')^\mathbf{A}$ and no variable is repeated in any subterm of $s'$ of height $1$.
We clearly have $t = s[t_1, \dots, t_{i-1}, u_1, \dots, u_m, t_{i+1}, \dots, t_n]$, and therefore $t$ is equivalent to the term $r := s'[t_1, \dots, t_{i-1}, u_1, \dots, u_m, t_{i+1}, \dots, t_n]$.
The term $r$ is not necessarily $(G,F)$\hyp{}stratified, but we obtain an equivalent $(G,F)$\hyp{}stratified term as follows.
Let $q$ be a subterm of $s'$ of height $1$, so $q = \varphi(x_{i_1}, \dots, x_{i_k})$ for some $\varphi \in F$ and the variables $x_{i_1}, \dots, x_{i_k}$ are distinct.
Then $q[t'_{i_1}, \dots, t'_{i_k}]$, where
\[
t'_j :=
\begin{cases}
t_j, & \text{if $1 \leq j \leq i - 1$,} \\
u_{j-i+1}, & \text{if $i \leq j \leq i + m - 1$,} \\
t_{j-m+1}, & \text{if $i + m \leq j \leq m + n - 1$,}
\end{cases}
\]
is a subterm of $r$ and it has fewer function symbols than $f$, so by the induction hypothesis there is an equivalent term $q' \in \Str{G,F}$.
By replacing $q[t'_{i_1}, \dots, t'_{i_k}]$ by $q'$, for each subterm $q$ of $s'$ of height $1$, we obtain the desired $(G,F)$\hyp{}stratified term equivalent to $t$.

It now follows from Lemma~\ref{lem:str-gen} that $\gen[(\clProj{A},C)]{F} = \clonegen{G \cup F}$.
\end{proof}

\begin{remark}
\label{rem:str-terms}
When verifying the conditions of Lemma~\ref{lem:str-terms}, it is not necessary to consider the cases when $\varphi = \id_A$ or $\alpha = \id_A$, because the equivalent term $t' \in \Str{G,F}$ always exists in these cases, as specified below:
\begin{gather*}
\id_A(\alpha(x_1, \dots, x_m)) \equiv \alpha(\id_A(x_1), \dots, \id_A(x_m)), \\
\varphi(x_1, \dots, x_{i-1}, \id_A(x_i), x_{i+1}, \dots, x_n) \equiv \varphi(\id_A(x_1), \dots, \id_A(x_n)), \\
\id_A(\id_A(x_1)) \equiv \id_A(x_1).
\end{gather*}

Moreover, we may be able to further reduce the number of cases to consider by making use of possible symmetries of $\varphi$.
More precisely, assume that the identity $\varphi(x_1, \dots, x_n) \approx \varphi(x_{\sigma(1)}, \dots, x_{\sigma(n)})$ is satisfied by $\mathbf{A}$ for some permutation $\sigma$ of $\nset{n}$.
If
\[
\varphi(x_1, \dots, x_{i-1}, \alpha(x_i, \dots, x_{i+m-1}), x_{i+m}, \dots, x_{m+n-1})
\]
is equivalent to $t \in \Str{G,F}$, then, for $j := \sigma^{-1}(i)$, the term
\[
\varphi(x_1, \dots, x_{j - 1}, \alpha(x_j, \dots, x_{j + m - 1}), x_{j + m}, \dots, x_{m+n-1})
\]
is equivalent to
\[
\varphi(x_{\zeta(1)}, \dots, x_{\zeta(i-1)}, \alpha(x_{\zeta(i)}, \dots, x_{\zeta(i+m-1)}), x_{\zeta(i+m)}, \dots, x_{\zeta(m+n-1)}),
\]
where $\zeta$ is the permutation of $\nset{m+n-1}$ given by the rule
\[
\zeta(k) =
\begin{cases}
\sigma(k) & \text{if $k \leq j - 1$ and $\sigma(k) \leq i - 1$,} \\
\sigma(k) + m - 1 & \text{if $k \leq j - 1$ and $\sigma(k) \geq i + 1$,} \\
k + i - j & \text{if $j \leq k \leq j + m - 1$,} \\
\sigma(k - m + 1) & \text{if $k \geq j + m$ and $\sigma(k) \leq i - 1$,} \\
\sigma(k - m + 1) + m - 1 & \text{if $k \geq j + m$ and $\sigma(k) \geq i + 1$,}
\end{cases}
\]
and this term is equivalent to $t[x_{\zeta(1)}, x_{\zeta(2)}, \dots, x_{\zeta(m+n-1)}] \in \Str{G,F}$.
\end{remark}


\section{Properties and classes of Boolean functions}
\label{sec:Boolean}

Operations on $\{0,1\}$ are called \emph{Boolean functions.}
We will often encounter the following well\hyp{}known Boolean functions, defined by the operation tables in Figure~\ref{fig:BFs}:
the constant functions $0$ and $1$,
identity $\id$,
negation $\neg$,
conjunction $\mathord{\wedge}$,
disjunction $\mathord{\vee}$,
addition $\mathord{+}$,
equivalence $\mathord{\leftrightarrow}$,
nonimplication $\mathord{\nrightarrow}$,
majority $\mu$,
and triple sum $\mathord{\oplus_3}$.

\begin{figure}
\newcommand{\FIGunary}{$\begin{array}[t]{c|cccc}
x_1 & 0 & 1 & \id & \neg \\
\hline
0 & 0 & 1 & 0 & 1 \\
1 & 0 & 1 & 1 & 0
\end{array}$}
\newcommand{\FIGbinary}{$\begin{array}[t]{cc|ccccc}
x_1 & x_2 & \wedge & \vee & + & \leftrightarrow & \nrightarrow \\
\hline
0 & 0 & 0 & 0 & 0 & 1 & 0 \\
0 & 1 & 0 & 1 & 1 & 0 & 0 \\
1 & 0 & 0 & 1 & 1 & 0 & 1 \\
1 & 1 & 1 & 1 & 0 & 1 & 0
\end{array}$}
\newcommand{\FIGternary}{$\begin{array}[t]{ccc|cc}
x_1 & x_2 & x_3 & \mu & \mathord{\oplus_3} \\
\hline
0 & 0 & 0 & 0 & 0 \\
0 & 0 & 1 & 0 & 1 \\
0 & 1 & 0 & 0 & 1 \\
0 & 1 & 1 & 1 & 0 \\
1 & 0 & 0 & 0 & 1 \\
1 & 0 & 1 & 1 & 0 \\
1 & 1 & 0 & 1 & 0 \\
1 & 1 & 1 & 1 & 1
\end{array}$}
\newlength{\FIGternaryheight}
\settototalheight{\FIGternaryheight}{\FIGternary}
\newlength{\FIGbinarywidth}
\settowidth{\FIGbinarywidth}{\FIGbinary}
\begin{minipage}[t][\FIGternaryheight][t]{\FIGbinarywidth}
\FIGunary
\vfill
\FIGbinary
\end{minipage}
\qquad\qquad
\FIGternary

\caption{Some well\hyp{}known Boolean functions.}
\label{fig:BFs}
\end{figure}

For $a \in \{0,1\}^n$, let $\overline{a} :=  1 - a$.
For $\vect{a} = (a_1, \dots, a_n) \in \{0,1\}^n$, let $\overline{\vect{a}} := (\overline{a_1}, \dots, \overline{a_n})$.
We regard the set $\{0,1\}$ totally ordered by the natural order $0 < 1$, which induces the direct product order on $\{0,1\}^n$.
The poset $(\{0,1\}^n, {\leq})$ constitutes a Boolean lattice, i.e., a complemented distributive lattice with least and greatest elements $\vect{0} = (0, \dots 0)$ and $\vect{1} = (1, \dots, 1)$, and with the map $\vect{a} \mapsto \overline{\vect{a}}$ being the complementation.

Denote by $\clAll$ the class of all Boolean functions.
For any $C \subseteq \clAll$ and
any $a, b \in \{0,1\}$, let
\begin{align*}
\clIntVal{C}{a}{} & := \{ \, f \in C \mid f(\vect{0}) = a \, \}, &
\clIntVal{C}{}{b} & := \{ \, f \in C \mid f(\vect{1}) = b \, \}, \\
\clIntVal{C}{a}{b} & := \clIntVal{C}{a}{} \cap \clIntVal{C}{}{b}, \\
\clIntEq{C} & := \{ \, f \in C \mid f(\vect{0}) = f(\vect{1}) \, \}, &
\clIntNeq{C} & := \{ \, f \in C \mid f(\vect{0}) \neq f(\vect{1}) \, \}, \\
\clIntLeq{C} & := \{ \, f \in C \mid f(\vect{0}) \leq f(\vect{1}) \, \}, &
\clIntGeq{C} & := \{ \, f \in C \mid f(\vect{0}) \geq f(\vect{1}) \, \}, \\
\clIntNeqOO{C} & := \clIntNeq{C} \cup \clIntVal{C}{0}{0}, &
\clIntNeqII{C} & := \clIntNeq{C} \cup \clIntVal{C}{1}{1}.
\end{align*}

Let $f \colon \{0,1\}^n \to \{0,1\}$.
The \emph{negation} $\overline{f}$, the \emph{inner negation} $f^\mathrm{n}$, and the \emph{dual} $f^\mathrm{d}$ of $f$ are the $n$\hyp{}ary Boolean functions given by the rules
$\overline{f}(\vect{a}) := \overline{f(\vect{a})}$, $f^\mathrm{n}(\vect{a}) := f(\overline{\vect{a}})$, and $f^\mathrm{d}(\vect{a}) := \overline{f(\overline{\vect{a}})}$, for all $\vect{a} \in \{0,1\}^n$.
A function $f \colon \{0,1\}^n \to \{0,1\}$ is \emph{reflexive} if $f = f^\mathrm{n}$, i.e., $f(\vect{a}) = f(\overline{\vect{a}})$ for all $\vect{a} \in \{0,1\}^n$, and $f$ is \emph{self\hyp{}dual} if $f = f^\mathrm{d}$, i.e., $f(\vect{a}) = \overline{f(\overline{\vect{a}})}$ for all $\vect{a} \in \{0,1\}^n$.
We denote by $\clRefl$ and $\clS$ the classes of all reflexive and all self\hyp{}dual functions, respectively.
For any $C \subseteq \clAll$, let $\overline{C} := \{ \, \overline{f} \mid f \in C \, \}$, $C^\mathrm{n} := \{ \, f^\mathrm{n} \mid f \in C \, \}$, and $C^\mathrm{d} := \{ \, f^\mathrm{d} \mid f \in C \, \}$.

Let $f, g \colon \{0,1\}^n \to \{0,1\}$.
We say that $f$ is a \emph{minorant} of $g$ if $f(\vect{a}) \leq g(\vect{a})$ for all $\vect{a} \in \{0,1\}^n$,
and we say that $f$ is a \emph{majorant} of $g$ if $f(\vect{a}) \geq g(\vect{a})$ for all $\vect{a} \in \{0,1\}^n$.
We denote by $\clSmin$ and $\clSmaj$ the classes of all minorants of self\hyp{}dual functions and all majorants of self\hyp{}dual functions, respectively.
Note that $\clS = \clSmin \cap \clSmaj$ and that $f \in \clSmin$ if and only if $f(\vect{a}) \wedge f(\overline{\vect{a}}) = 0$ for all $\vect{a} \in \{0,1\}^n$, and $f \in \clSmaj$ if and only if $f(\vect{a}) \vee f(\overline{\vect{a}}) = 1$ for all $\vect{a} \in \{0,1\}^n$.

A function $f \in \{0,1\}^n \to \{0,1\}$ is \emph{monotone} if $f(\vect{a}) \leq f(\vect{b})$ whenever $\vect{a} \leq \vect{b}$.
We denote by $\clM$ the class of all monotone functions.

A function $f \in \{0,1\}^n \to \{0,1\}$ is \emph{$1$\hyp{}separating of rank $k$} if for all $\vect{a}_1, \dots, \vect{a}_k \in f^{-1}(1)$ it holds that $\vect{a}_1 \wedge \dots \wedge \vect{a}_k \neq \vect{0}$,
and $f$ is \emph{$0$\hyp{}separating of rank $k$} if for all $\vect{a}_1, \dots, \vect{a}_k \in f^{-1}(0)$ it holds that $\vect{a}_1 \vee \dots \vee \vect{a}_k \neq \vect{1}$.
We denote by $\clUk{k}$ and $\clWk{k}$ the classes of all $1$\hyp{}separating functions of rank $k$ and all $0$\hyp{}separating functions of rank $k$.
We introduce the shorthands $\clSM$ for $\clS \cap \clM$, $\clMUk{k}$ for $\clM \cap \clUk{k}$, and $\clMWk{k}$ for $\clM \cap \clWk{k}$.

We denote by $\clVak$ the class of all constant functions, and we introduce the shorthands $\clVako := \clIntVal{\clVak}{0}{0}$ and $\clVaki := \clIntVal{\clVak}{1}{1}$.
We denote by $\clIc$ the class of all projections on $\{0,1\}$.

The clones of Boolean functions were described by Post~\cite{Post}.
In this paper, we will need only a handful of them, namely the clones $\clAll$, $\clOX$, $\clXI$, $\clOI$, $\clM$, $\clMo$, $\clMi$, $\clMc$, $\clS$, $\clSc$, $\clSM$, $\clU$, $\clTcU$, $\clMU$, $\clMcU$, $\clW$, $\clTcW$, $\clMW$, $\clMcW$, and $\clIc$ that were defined above.
The lattice of clones of Boolean functions, \emph{Post's lattice,} is presented in Figure~\ref{fig:Post}, and the above\hyp{}mentioned clones are indicated in the Hasse diagram.


\section{\texorpdfstring{$(\clIc,\clSM)$\hyp{}stable}{(J,SM)-stable} classes}
\label{sec:IcSM}

\renewcommand{\gendefault}{(\clIc,\clSM)}

\begin{theorem}
\label{thm:SM-stable}
There are precisely 93 $(\clIc,\clSM)$\hyp{}stable classes of Boolean functions, and they are the following:
\[
\begin{array}{l}
\clAll, \quad
\clEiio, \quad \clEioi, \quad \clEiii, \quad \clEioo, \\
\clOXC, \quad \clXOC, \quad \clOOC, \quad \clOIC, \quad \clOICO, \quad \clOICI, \\
\clIXC, \quad \clXIC, \quad \clIIC, \quad \clIOC, \quad \clIOCI, \quad \clIOCO, \\
\clNeq, \quad \clEq, \quad \clOX, \quad \clIX, \quad \clXO, \quad \clXI, \quad \clOO, \quad \clII, \quad \clOI, \quad \clIO, \\
\clM, \quad \clMo, \quad \clMi, \quad \clMc, \quad
\clMneg, \quad \clMoneg, \quad \clMineg, \quad \clMcneg, \\
\clSmin, \quad \clSminNeq, \quad \clSminOX, \quad \clSminXO, \quad \clSminOICO, \quad \clSminIOCO, \quad \clSminOI, \quad \clSminIO, \quad \clSminOO, \\
\clSmaj, \quad \clSmajNeq, \quad \clSmajIX, \quad \clSmajXI, \quad \clSmajIOCI, \quad \clSmajOICI, \quad \clSmajIO, \quad \clSmajOI, \quad \clSmajII, \\
\clS, \quad \clSc, \quad \clScneg, \quad \clSM, \quad \clSMneg, \\
\clU, \quad \clTcUCO, \quad \clTcU, \quad \clMU, \quad \clMcU, \quad \clUOO, \\
\clUneg, \quad \clTcUnegCI, \quad \clTcUneg, \quad \clMUneg, \quad \clMcUneg, \quad \clUnegII, \\
\clW, \quad \clTcWCI, \quad \clTcW, \quad \clMW, \quad \clMcW, \quad \clWII, \\
\clWneg, \quad \clTcWnegCO, \quad \clTcWneg, \quad \clMWneg, \quad \clMcWneg, \quad \clWnegOO, \quad
\clUWneg, \quad \clWUneg, \\
\clRefl, \quad \clReflOOC, \quad \clReflIIC, \quad \clReflOO, \quad \clReflII, \quad
\clVak, \quad \clVako, \quad \clVaki, \quad \clEmpty.
\end{array}
\]
\end{theorem}

The lattice of $(\clIc,\clSM)$\hyp{}stable classes is shown in Figure~\ref{fig:SM-stable}.
The Hasse diagram is drawn in such a way that the reflection along the central vertical axis corresponds to the automorphism $C \mapsto \overline{C}$.

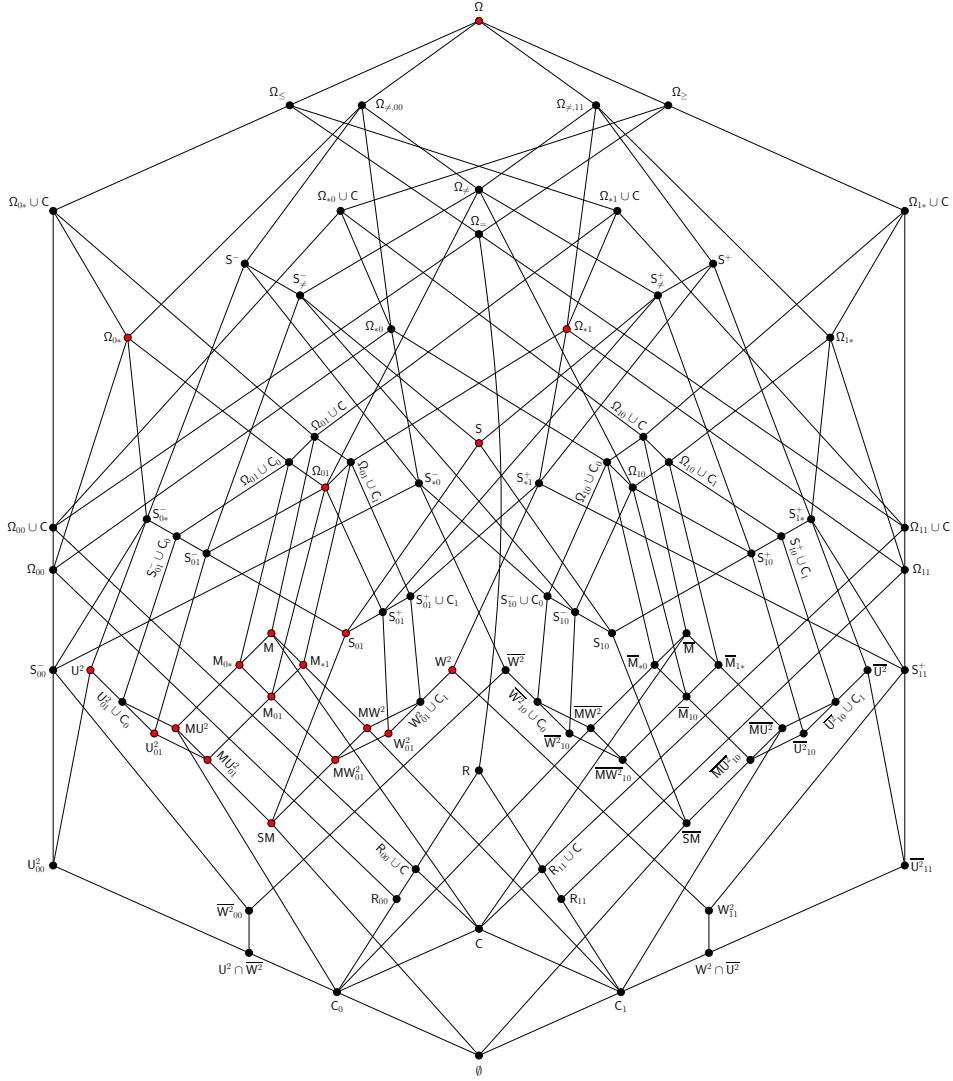
\begin{figure}
\begin{center}
\scalebox{0.28}{
\tikzstyle{every node}=[circle, draw, fill=black, scale=1, font=\huge]
\tikzstyle{clone}=[fill=red]
\begin{tikzpicture}[baseline, scale=1]
\coordinate (bottom) at (0,0);
\coordinate (sw) at ($(bottom)+(-20,9)$);
\coordinate (se) at ($(bottom)+(20,9)$);
\coordinate (top) at ($(bottom)+(0,49)$);
\coordinate (nw) at ($(top)+(-20,-9)$);
\coordinate (ne) at ($(top)+(20,-9)$);
\node (empty) at (bottom) {};
\draw ($(empty)+(270:0.7)$) node[draw=none,fill=none]{$\clEmpty$};
\node (D0C0) at ($(bottom)!1/3!(sw)$) {};
\draw ($(D0C0)+(270:0.7)$) node[draw=none,fill=none]{$\clVako$};
\node (D0C1) at ($(bottom)!1/3!(se)$) {};
\draw ($(D0C1)+(270:0.7)$) node[draw=none,fill=none]{$\clVaki$};
\node (D0) at ($(bottom)+(0,6)$) {};
\draw ($(D0)+(270:0.7)$) node[draw=none,fill=none]{$\clVak$};
\node (U2W2neg) at ($(bottom)!0.54!(sw)$) {};
\draw ($(U2W2neg)+(240:0.8)$) node[draw=none,fill=none]{$\clUWneg$};
\node (W2U2neg) at ($(bottom)!0.54!(se)$) {};
\draw ($(W2U2neg)+(300:0.8)$) node[draw=none,fill=none]{$\clWUneg$};
\node (U2E0) at (sw) {};
\draw ($(U2E0)+(180:0.8)$) node[draw=none,fill=none]{$\clUOO$};
\node (U2negE1) at (se) {};
\draw ($(U2negE1)+(0:0.8)$) node[draw=none,fill=none]{$\clUnegII$};
\node (W2C1) at ($(W2U2neg)+(0,2)$) {};
\draw ($(W2C1)+(0:0.9)$) node[draw=none,fill=none]{$\clWII$};
\node (W2negC0) at ($(U2W2neg)+(0,2)$) {};
\draw ($(W2negC0)+(180:0.9)$) node[draw=none,fill=none]{$\clWnegOO$};
\node[clone] (SM) at ($(bottom) + (-9.75,11)$) {};
\draw ($(SM)+(250:0.7)$) node[draw=none,fill=none]{$\clSM$};
\node (SMneg) at ($(bottom) + (9.75,11)$) {};
\draw ($(SMneg)+(290:0.7)$) node[draw=none,fill=none]{$\clSMneg$};
\node[clone] (Mc) at ($(SM)+(0,6)$) {};
\draw ($(Mc)+(275:0.8)$) node[draw=none,fill=none]{$\clMc$};
\node[clone] (M0) at ($(Mc)+(-1.5,1.5)$) {};
\draw ($(M0)+(170:0.8)$) node[draw=none,fill=none]{$\clMo$};
\node[clone] (M1) at ($(Mc)+(1.5,1.5)$) {};
\draw ($(M1)+(10:0.8)$) node[draw=none,fill=none]{$\clMi$};
\node[clone] (M) at ($(Mc)+(0,3)$) {};
\draw ($(M)+(260:0.7)$) node[draw=none,fill=none]{$\clM$};
\node[clone] (MU2) at ($(M0)+(-3,-3)$) {};
\draw ($(MU2)+(0:1)$) node[draw=none,fill=none]{$\clMU$};
\node[clone] (McU2) at ($(Mc)+(-3,-3)$) {};
\draw ($(McU2)+(350:1)$) node[draw=none,fill=none,rotate=315]{$\clMcU$};
\node[clone] (TcU2) at ($(McU2)+(-2.5,1.25)$) {};
\draw ($(TcU2)+(270:0.7)$) node[draw=none,fill=none]{$\clTcU$};
\node (TcU2D0) at ($(MU2)+(-2.5,1.25)$) {};
\draw ($(TcU2D0)+(225:0.55)$) node[draw=none,fill=none,rotate=315]{$\clTcUCO$};
\node[clone] (U2) at ($(TcU2D0)+(-1.5,1.5)$) {};
\draw ($(U2)+(180:0.6)$) node[draw=none,fill=none]{$\clU$};
\node[clone] (MW2) at ($(M1)+(3,-3)$) {};
\draw ($(MW2)+(75:0.8)$) node[draw=none,fill=none]{$\clMW$};
\node[clone] (McW2) at ($(Mc)+(3,-3)$) {};
\draw ($(McW2)+(315:0.9)$) node[draw=none,fill=none]{$\clMcW$};
\node[clone] (TcW2) at ($(McW2)+(2.5,1.25)$) {};
\draw ($(TcW2)+(325:0.9)$) node[draw=none,fill=none]{$\clTcW$};
\node (TcW2D0) at ($(MW2)+(2.5,1.25)$) {};
\draw ($(TcW2D0)+(315:0.55)$) node[draw=none,fill=none,rotate=45]{$\clTcWCI$};
\node[clone] (W2) at ($(TcW2D0)+(1.5,1.5)$) {};
\draw ($(W2)+(135:0.6)$) node[draw=none,fill=none]{$\clW$};
\node (Mcneg) at ($(SMneg)+(0,6)$) {};
\draw ($(Mcneg)+(275:0.8)$) node[draw=none,fill=none]{$\clMcneg$};
\node (M0neg) at ($(Mcneg)+(1.5,1.5)$) {};
\draw ($(M0neg)+(10:0.8)$) node[draw=none,fill=none]{$\clMoneg$};
\node (M1neg) at ($(Mcneg)+(-1.5,1.5)$) {};
\draw ($(M1neg)+(170:0.8)$) node[draw=none,fill=none]{$\clMineg$};
\node (Mneg) at ($(Mcneg)+(0,3)$) {};
\draw ($(Mneg)+(280:0.7)$) node[draw=none,fill=none]{$\clMneg$};
\node (MU2neg) at ($(M0neg)+(3,-3)$) {};
\draw ($(MU2neg)+(180:1)$) node[draw=none,fill=none]{$\clMUneg$};
\node (McU2neg) at ($(Mcneg)+(3,-3)$) {};
\draw ($(McU2neg)+(190:1.2)$) node[draw=none,fill=none,rotate=45]{$\clMcUneg$};
\node (TcU2neg) at ($(McU2neg)+(2.5,1.25)$) {};
\draw ($(TcU2neg)+(275:0.7)$) node[draw=none,fill=none]{$\clTcUneg$};
\node (TcU2negD0) at ($(MU2neg)+(2.5,1.25)$) {};
\draw ($(TcU2negD0)+(315:0.55)$) node[draw=none,fill=none,rotate=45]{$\clTcUnegCI$};
\node (U2neg) at ($(TcU2negD0)+(1.5,1.5)$) {};
\draw ($(U2neg)+(0:0.6)$) node[draw=none,fill=none]{$\clUneg$};
\node (MW2neg) at ($(M1neg)+(-3,-3)$) {};
\draw ($(MW2neg)+(100:0.8)$) node[draw=none,fill=none]{$\clMWneg$};
\node (McW2neg) at ($(Mcneg)+(-3,-3)$) {};
\draw ($(McW2neg)+(237:0.8)$) node[draw=none,fill=none]{$\clMcWneg$};
\node (TcW2neg) at ($(McW2neg)+(-2.5,1.25)$) {};
\draw ($(TcW2neg)+(215:0.7)$) node[draw=none,fill=none]{$\clTcWneg$};
\node (TcW2negD0) at ($(MW2neg)+(-2.5,1.25)$) {};
\draw ($(TcW2negD0)+(225:0.55)$) node[draw=none,fill=none,rotate=315]{$\clTcWnegCO$};
\node (W2neg) at ($(TcW2negD0)+(-1.5,1.5)$) {};
\draw ($(W2neg)+(45:0.7)$) node[draw=none,fill=none]{$\clWneg$};
\node[clone] (All) at (top) {};
\draw ($(All)+(90:0.7)$) node[draw=none,fill=none]{$\clAll$};
\node (Eiio) at ($(top)!4/9!(nw)$) {};
\draw ($(Eiio)+(135:0.8)$) node[draw=none,fill=none]{$\clEiio$};
\node (Eiii) at ($(top)+(-5.5,-4)$) {};
\draw ($(Eiii)+(0:1.3)$) node[draw=none,fill=none]{$\clEiii$};
\node (Eioo) at ($(top)+(5.5,-4)$) {};
\draw ($(Eioo)+(180:1.2)$) node[draw=none,fill=none]{$\clEioo$};
\node (Eioi) at ($(top)!4/9!(ne)$) {};
\draw ($(Eioi)+(45:0.8)$) node[draw=none,fill=none]{$\clEioi$};
\node (P1) at ($(top)+(0,-8)$) {};
\draw ($(P1)+(175:0.8)$) node[draw=none,fill=none]{$\clNeq$};
\node (P0) at ($(top)+(0,-10.1)$) {};
\draw ($(P0)+(90:0.6)$) node[draw=none,fill=none]{$\clEq$};
\node (C0D0) at (nw) {};
\draw ($(C0D0)+(160:1.2)$) node[draw=none,fill=none]{$\clOXC$};
\node[clone] (C0) at ($(C0D0)+(3.5,-6)$) {};
\draw ($(C0)+(180:0.7)$) node[draw=none,fill=none]{$\clOX$};
\node (C1D0) at (ne) {};
\draw ($(C1D0)+(20:1.2)$) node[draw=none,fill=none]{$\clIXC$};
\node (C1) at ($(C1D0)+(-3.5,-6)$) {};
\draw ($(C1)+(0:0.7)$) node[draw=none,fill=none]{$\clIX$};
\node (E1D0) at ($(top)+(6.5,-9)$) {};
\draw ($(E1D0)+(80:0.7)$) node[draw=none,fill=none]{$\clXIC$};
\node[clone] (E1) at ($(E1D0)+1.4*(-1.7,-4)$) {};
\draw ($(E1)+(10:0.8)$) node[draw=none,fill=none]{$\clXI$};
\node (E0D0) at ($(top)+(-6.5,-9)$) {};
\draw ($(E0D0)+(100:0.7)$) node[draw=none,fill=none]{$\clXOC$};
\node (E0) at ($(E0D0)+1.4*(1.7,-4)$) {};
\draw ($(E0)+(170:0.8)$) node[draw=none,fill=none]{$\clXO$};
\node[clone] (C0E1) at ($(M)+2.3*(1.1,3)$) {};
\draw ($(C0E1)+(104:0.8)$) node[draw=none,fill=none]{$\clOI$};
\node (C0E1D000) at ($(C0E1)+(-1.2-0.5,1.2)$) {};
\draw ($(C0E1D000)+(197:1.4)$) node[draw=none,fill=none,rotate=33]{$\clOICO$};
\node (C0E1D011) at ($(C0E1)+(1.2,1.2)$) {};
\draw ($(C0E1D011)+(313:1.3)$) node[draw=none,fill=none,rotate=294]{$\clOICI$};
\node (C0E1D0) at ($(C0E1)+(0-0.5,2*1.2)$) {};
\draw ($(C0E1D0)+(55:1.2)$) node[draw=none,fill=none,rotate=36]{$\clOIC$};
\node (C1E0) at ($(Mneg)+2.3*(-1.1,3)$) {};
\draw ($(C1E0)+(75:0.8)$) node[draw=none,fill=none]{$\clIO$};
\node (C1E0D011) at ($(C1E0)+(1.2+0.5,1.2)$) {};
\draw ($(C1E0D011)+(343:1.5)$) node[draw=none,fill=none,rotate=327]{$\clIOCI$};
\node (C1E0D000) at ($(C1E0)+(-1.2,1.2)$) {};
\draw ($(C1E0D000)+(226:1.25)$) node[draw=none,fill=none,rotate=66]{$\clIOCO$};
\node (C1E0D0) at ($(C1E0)+(0+0.5,2*1.2)$) {};
\draw ($(C1E0D0)+(125:1.2)$) node[draw=none,fill=none,rotate=323]{$\clIOC$};
\node[clone] (S) at ($(bottom)+(0,29)$) {};
\draw ($(S)+(90:0.7)$) node[draw=none,fill=none]{$\clS$};
\node[clone] (Sc) at ($(M) + (3.5,0)$) {};
\draw ($(Sc)+(315:0.7)$) node[draw=none,fill=none]{$\clSc$};
\node (Scneg) at ($(Mneg) + (-3.5,0)$) {};
\draw ($(Scneg)+(225:0.7)$) node[draw=none,fill=none]{$\clScneg$};
\node (Smin) at ($(top)+(-11,-11.5)$) {};
\draw ($(Smin)+(160:0.6)$) node[draw=none,fill=none]{$\clSmin$};
\node (Smaj) at ($(top)+(11,-11.5)$) {};
\draw ($(Smaj)+(20:0.6)$) node[draw=none,fill=none]{$\clSmaj$};
\node (SminP1) at ($(Smin)+(330:3)$) {};
\draw ($(SminP1)+(90:0.7)$) node[draw=none,fill=none]{$\clSminNeq$};
\node (SmajP1) at ($(Smaj)+(210:3)$) {};
\draw ($(SmajP1)+(90:0.7)$) node[draw=none,fill=none]{$\clSmajNeq$};
\node (SminC0) at ($(Sc)+(150:10.8)$) {};
\draw ($(SminC0)+(20:0.7)$) node[draw=none,fill=none]{$\clSminOX$};
\node (SmajC1) at ($(Scneg)+(30:10.8)$) {};
\draw ($(SmajC1)+(160:0.7)$) node[draw=none,fill=none]{$\clSmajIX$};
\node (SmajE1) at ($(C1E0D000)+(-3.2,-1)$) {};
\draw ($(SmajE1)+(160:0.7)$) node[draw=none,fill=none]{$\clSmajXI$};
\node (SminE0) at ($(C0E1D011)+(3.2,-1)$) {};
\draw ($(SminE0)+(20:0.7)$) node[draw=none,fill=none]{$\clSminXO$};
\node (SminC0E1) at ($(Sc)!0.70!(SminC0)$) {};
\draw ($(SminC0E1)+(200:0.7)$) node[draw=none,fill=none]{$\clSminOI$};
\node (SminC0E1D000) at ($(Sc)!0.85!(SminC0)$) {};
\draw ($(SminC0E1D000)+(230:1.3)$) node[draw=none,fill=none,rotate=70]{$\clSminOICO$};
\node (SmajC1E0D011) at ($(Scneg)!0.85!(SmajC1)$) {};
\draw ($(SmajC1E0D011)+(310:1.4)$) node[draw=none,fill=none,rotate=290]{$\clSmajIOCI$};
\node (SmajC1E0) at ($(Scneg)!0.70!(SmajC1)$) {};
\draw ($(SmajC1E0)+(340:0.7)$) node[draw=none,fill=none]{$\clSmajIO$};
\node (SmajC0E1) at ($(Sc)+(30:2)$) {};
\draw ($(SmajC0E1)+(340:0.7)$) node[draw=none,fill=none]{$\clSmajOI$};
\node (SmajC0E1D011) at ($(Sc)+(30:3.5)$) {};
\draw ($(SmajC0E1D011)+(350:1.3)$) node[draw=none,fill=none]{$\clSmajOICI$};
\node (SminC1E0D000) at ($(Scneg)+(150:3.5)$) {};
\draw ($(SminC1E0D000)+(190:1.2)$) node[draw=none,fill=none]{$\clSminIOCO$};
\node (SminC1E0) at ($(Scneg)+(150:2)$) {};
\draw ($(SminC1E0)+(200:0.7)$) node[draw=none,fill=none]{$\clSminIO$};
\node (SminC0E0) at ($(sw)+(0,9.25)$) {};
\draw ($(SminC0E0)+(180:0.7)$) node[draw=none,fill=none]{$\clSminOO$};
\node (SmajC1E1) at ($(se)+(0,9.25)$) {};
\draw ($(SmajC1E1)+(0:0.7)$) node[draw=none,fill=none]{$\clSmajII$};
\node (C0E0) at ($(sw)+(0,14)$) {};
\draw ($(C0E0)+(180:0.8)$) node[draw=none,fill=none]{$\clOO$};
\node (C0E0D0) at ($(C0E0)+(0,2)$) {};
\draw ($(C0E0D0)+(180:1.2)$) node[draw=none,fill=none]{$\clOOC$};
\node (C1E1) at ($(se)+(0,14)$) {};
\draw ($(C1E1)+(0:0.8)$) node[draw=none,fill=none]{$\clII$};
\node (C1E1D0) at ($(C1E1)+(0,2)$) {};
\draw ($(C1E1D0)+(0:1.2)$) node[draw=none,fill=none]{$\clIIC$};
\node (X1P0) at ($(bottom)+(0,13.5)$) {};
\draw ($(X1P0)+(180:0.6)$) node[draw=none,fill=none]{$\clRefl$};
\node (X1C0E0D0) at (intersection of D0--C0E0D0 and D0C0--X1P0) {};
\draw ($(X1C0E0D0)+(158:1.2)$) node[draw=none,fill=none,rotate=317]{$\clReflOOC$};
\node (X1C1E1D0) at (intersection of D0--C1E1D0 and D0C1--X1P0) {};
\draw ($(X1C1E1D0)+(22:1.2)$) node[draw=none,fill=none,rotate=43]{$\clReflIIC$};
\node (X1C0E0) at ($(D0C0)!0.42!(X1P0)$) {};
\draw ($(X1C0E0)+(180:0.8)$) node[draw=none,fill=none]{$\clReflOO$};
\node (X1C1E1) at ($(D0C1)!0.42!(X1P0)$) {};
\draw ($(X1C1E1)+(0:0.8)$) node[draw=none,fill=none]{$\clReflII$};
\foreach \u/\v in {
   empty/D0C0, empty/D0C1, D0C0/D0, D0C1/D0, empty/SM, empty/SMneg, SM/Sc, SMneg/Scneg, Sc/S, Scneg/S,
   D0/X1C0E0D0, D0/X1C1E1D0,
   D0C0/MU2, SM/McU2, SM/McW2,
   D0C0/U2W2neg, U2W2neg/U2E0, U2W2neg/W2negC0,
   X1C0E0/X1C0E0D0, X1C0E0D0/X1P0, X1C0E0/C0E0, X1C1E1/X1C1E1D0, X1C1E1D0/X1P0, X1C1E1/C1E1, 
   D0C0/X1C0E0, D0C1/X1C1E1,
   McU2/Mc, McU2/MU2, McU2/TcU2, MU2/M0, MU2/TcU2D0, TcU2/TcU2D0, TcU2D0/U2, McW2/Mc, McW2/MW2, McW2/TcW2, MW2/M1, MW2/TcW2D0, TcW2/TcW2D0, TcW2D0/W2, Mc/M0, Mc/M1, M0/M, M1/M, U2E0/U2, W2C1/W2,
   D0C1/MU2neg, SMneg/McU2neg, SMneg/McW2neg,
   D0C1/W2U2neg, W2U2neg/U2negE1, W2U2neg/W2C1,
   McU2neg/Mcneg, McU2neg/MU2neg, McU2neg/TcU2neg, MU2neg/M0neg, MU2neg/TcU2negD0, TcU2neg/TcU2negD0, TcU2negD0/U2neg, McW2neg/Mcneg, McW2neg/MW2neg, McW2neg/TcW2neg, MW2neg/M1neg, MW2neg/TcW2negD0, TcW2neg/TcW2negD0, TcW2negD0/W2neg, Mcneg/M0neg, Mcneg/M1neg, M0neg/Mneg, M1neg/Mneg, U2negE1/U2neg, W2negC0/W2neg,
   Mc/C0E1, U2/SminC0, W2/SmajE1, U2E0/SminC0E0, W2negC0/SminC0E0, SminC0E0/C0E0,
   Mcneg/C1E0, U2neg/SmajC1, W2neg/SminE0, U2negE1/SmajC1E1, W2C1/SmajC1E1, SmajC1E1/C1E1,
   M0/C0E1D000, M1/C0E1D011, M/C0E1D0, M0neg/C1E0D011, M1neg/C1E0D000, Mneg/C1E0D0,
   X1C0E0D0/C0E0D0, X1C1E1D0/C1E1D0,
   C0E0/C0, C0E0/E0, C0E1D000/C0, C0E1D011/E1, C0E1/P1, C1E0D011/C1, C1E0D000/E0, C1E0/P1, C1E1/C1, C1E1/E1,
   C0E0/C0E0D0, C1E1/C1E1D0, C0E1/C0E1D000, C0E1/C0E1D011, C0E1D000/C0E1D0, C0E1D011/C0E1D0, C1E0/C1E0D000, C1E0/C1E0D011, C1E0D000/C1E0D0, C1E0D011/C1E0D0,
   C0/C0D0, C1/C1D0, E0/E0D0, E1/E1D0,
   C0E0D0/C0D0, C0E0D0/E0D0, C1E1D0/C1D0, C1E1D0/E1D0, C0E1D0/C0D0, C0E1D0/E1D0, C1E0D0/C1D0, C1E0D0/E0D0,
   C0E0D0/P0, C1E1D0/P0,
   C0D0/Eiio, C1D0/Eioi, E0D0/Eioi, E1D0/Eiio,
   C0/Eiii, E0/Eiii, E1/Eioo, C1/Eioo,
   P0/Eiio, P0/Eioi, P1/Eiii, P1/Eioo,
   Eiio/All, Eiii/All, Eioo/All, Eioi/All,
   S/SminP1, S/SmajP1, SminP1/Smin, SmajP1/Smaj, SminP1/P1, SmajP1/P1,
   Sc/SminC0E1, Sc/SmajC0E1, Scneg/SmajC1E0, Scneg/SminC1E0,
   TcU2/SminC0E1, TcW2/SmajC0E1, TcU2neg/SmajC1E0, TcW2neg/SminC1E0,
   TcU2D0/SminC0E1D000, TcW2D0/SmajC0E1D011, TcU2negD0/SmajC1E0D011, TcW2negD0/SminC1E0D000,
   SminC0E0/SminC0, SminC0E0/SminE0, SmajC1E1/SmajC1, SmajC1E1/SmajE1,
   SminC0E1/SminC0E1D000, SmajC0E1/SmajC0E1D011, SminC1E0/SminC1E0D000, SmajC1E0/SmajC1E0D011,
   SminC0E1/C0E1, SminC0E1D000/C0E1D000, SminC0E1D000/SminC0, SmajC0E1/C0E1, SmajC0E1D011/C0E1D011, SmajC0E1D011/SmajE1,
   SmajC1E0/C1E0, SmajC1E0D011/C1E0D011, SmajC1E0D011/SmajC1, SminC1E0/C1E0, SminC1E0D000/C1E0D000, SminC1E0D000/SminE0,
   SminC0/C0, SminE0/E0, SmajE1/E1, SmajC1/C1,
   SminC0/Smin, SminE0/Smin, SmajE1/Smaj, SmajC1/Smaj,
   SminC0E1/SminP1, SminC1E0/SminP1, SmajC0E1/SmajP1, SmajC1E0/SmajP1,
   Smin/Eiii, Smaj/Eioo,
   D0/M, D0/Mneg,
   D0C0/MW2neg, D0C1/MW2%
}
{
   \draw [thick] (\u) -- (\v);
}
\draw [thick] (X1P0) to[out=82,in=278] (P0);
\end{tikzpicture}
}
\end{center}
\caption{$(\clIc,\clSM)$\hyp{}stable classes. Clones are highlighted in red.}
\label{fig:SM-stable}
\end{figure}

The remainder of this section is devoted to the proof of Theorem~\ref{thm:SM-stable}.
The proof comprises two main parts.
Firstly, we show that each class listed in the theorem is $(\clIc,\clSM)$\hyp{}stable; this is rather straightforward verification.
Secondly, we prove that there are no further $(\clIc,\clSM)$\hyp{}stable classes. This is the more complicated part of the proof.
The idea is to show that any set of Boolean functions generates one of the $(\clIc,\clSM)$\hyp{}stable classes listed in Theorem~\ref{thm:SM-stable}, namely the one suggested by Figure~\ref{fig:SM-stable}.
We develop some tools that allow us to easily determine what a given set of functions generates.

Without further ado, we now show that the classes listed in Theorem~\ref{thm:SM-stable} are $(\clIc,\clSM)$\hyp{}stable.
With the help of the following two lemmata, we can reduce the number of classes we need to consider.

\begin{lemma}
\label{lem:IcSM-stable-negation}
Let $C, F \subseteq \clAll$, and assume that 
$\gen{F} = C$.
Then also $\overline{C}$, $C^\mathrm{n}$, and $C^\mathrm{d}$ are $(\clIc,\clSM)$\hyp{}stable and
$\gen{\overline{F}} = \overline{C}$, $\gen{F^\mathrm{n}} = C^\mathrm{n}$, $\gen{F^\mathrm{d}} = C^\mathrm{d}$.
\end{lemma}

\begin{proof}
In order to show that $\overline{C}$ is $(\clIc,\clSM)$\hyp{}stable, take three arbitrary members of $\overline{C}$; they are of the form $\overline{f}, \overline{g}, \overline{h}$ for some $f, g, h \in C$.
Then $\overline{f}_\sigma = (\overline{\id} \circ f) \circ \underline{\sigma} = \overline{\id} \circ (f \circ \underline{\sigma}) = \overline{f_\sigma} \in \overline{C}$ because $f_\sigma \in C$.
Moreover, $\mu(\overline{f}, \overline{g}, \overline{h}) = \overline{\mu(f,g,h)} \in \overline{C}$ because $\mu(f,g,h) \in C$.
We conclude that $\overline{C}$ is $(\clIc,\clSM)$\hyp{}stable.

In order to show that $\overline{C}$ is generated by $\overline{F}$, let $\overline{f} \in \overline{C}$ be $n$\hyp{}ary.
Then $f \in C$, so there is a function $g \in SM^{(m)}$, $h_1, \dots, h_m \in F$ and $n$\hyp{}ary minors $(h_i)_{\sigma_i}$ ($1 \leq i \leq m$) such that $f = g((h_1)_{\sigma_1}, \dots, (h_m)_{\sigma_m})$.
But then
\[
\begin{split}
\overline{f}
&
= \overline{g((h_1)_{\sigma_1}, \dots, (h_m)_{\sigma_m})}
= \overline{g}((h_1)_{\sigma_1}, \dots, (h_m)_{\sigma_m})
\\ &
= g(\overline{(h_1)_{\sigma_1}}, \dots, \overline{(h_m)_{\sigma_m}})
= g((\overline{h_1})_{\sigma_1}, \dots, (\overline{h_m})_{\sigma_m})
\in \gen{\overline{F}},
\end{split}
\]
where the third equality holds because $g$ is self\hyp{}dual.
Therefore $\overline{C} \subseteq \gen{\overline{F}}$.
Since $\overline{F} \subseteq \overline{C}$ and $\overline{C}$ is $(\clIc,\clSM)$\hyp{}stable, we have $\gen{\overline{F}} \subseteq \gen{\overline{C}} = \overline{C}$.

In order to show that $C^\mathrm{n}$ is $(\clIc,\clSM)$\hyp{}stable, take three arbitrary members of $C^\mathrm{n}$; they are of the form $f^\mathrm{n}, g^\mathrm{n}, h^\mathrm{n}$ for some $f, g, h \in C$.
Then
$(f^\mathrm{n})_\sigma(a_1, \dots, a_m) = f(\overline{a_{\sigma(1)}}, \dots, \overline{a_{\sigma(n)}}) = (f_\sigma)^\mathrm{n}(a_1, \dots, a_m)$, i.e., $(f^\mathrm{n})_\sigma = (f_\sigma)^\mathrm{n} \in C^\mathrm{n}$ because $f_\sigma \in C$.
Moreover, $\mu(f^\mathrm{n}, g^\mathrm{n}, h^\mathrm{n}) = (\mu(f, g, h))^\mathrm{n} \in C^\mathrm{n}$ because $\mu(f, g, h) \in C$.
We conclude that $C^\mathrm{n}$ is $(\clIc,\clSM)$\hyp{}stable.

In order to show that $C^\mathrm{n}$ is generated by $F^\mathrm{n}$, let $f^\mathrm{n} \in C^\mathrm{n}$ be $n$\hyp{}ary.
Then $f \in C$, so there is a function $g \in \clSM$, $h_1, \dots, h_m \in F$, and $n$\hyp{}ary minors $(h_i)_{\sigma_i}$ ($1 \leq i \leq m$) such that $f = g((h_1)_{\sigma_1}, \dots, (h_m)_{\sigma_m})$.
But then
\[
\begin{split}
f^\mathrm{n}
&
= (g((h_1)_{\sigma_1}, \dots, (h_m)_{\sigma_m}))^\mathrm{n}
= g(((h_1)_{\sigma_1})^\mathrm{n}, \dots, ((h_m)_{\sigma_m})^\mathrm{n})
\\ &
= g((h_1^\mathrm{n})_{\sigma_1}, \dots, (h_m^\mathrm{n})_{\sigma_m})
\in \gen{F^\mathrm{n}},
\end{split}
\]
so $C^\mathrm{n} \subseteq \gen{F^\mathrm{n}}$.
Since $F^\mathrm{n} \subseteq C^\mathrm{n}$ and $C^\mathrm{n}$ is $(\clIc,\clSM)$\hyp{}stable, we have $\gen{F^\mathrm{n}} \subseteq \gen{C^\mathrm{n}} = C^\mathrm{n}$.

The statements about $C^\mathrm{d}$ and $F^\mathrm{d}$ follow from the above, because $f^\mathrm{d} = \overline{f^\mathrm{n}} = (\overline{f})^\mathrm{n}$ for any $f \in \clAll$.
\end{proof}

\begin{lemma}
\label{lem:IcSM-stable-constants}
Let $C$ be a $(\clIc,\clSM)$\hyp{}stable class.
\begin{enumerate}[label=\upshape{(\roman*)}]
\item\label{lem:IcSM-stable-constants:0}
If $\{\mathord{\wedge}\} C \subseteq C$ then $C \cup \clVako$ is $(\clIc,\clSM)$\hyp{}stable.
\item\label{lem:IcSM-stable-constants:1}
If $\{\mathord{\vee}\} C \subseteq C$ then $C \cup \clVaki$ is $(\clIc,\clSM)$\hyp{}stable.
\item\label{lem:IcSM-stable-constants:01}
If $\{\mathord{\wedge}, \mathord{\vee}\} C \subseteq C$ then $C \cup \clVak$ is $(\clIc,\clSM)$\hyp{}stable.
\end{enumerate}
\end{lemma}

\begin{proof}
\ref{lem:IcSM-stable-constants:0}
Assume that $C$ is $(\clIc,\clSM)$\hyp{}stable and $\{\mathord{\wedge}\} C \subseteq C$.
Then $C \cup \clVako$ is clearly minor\hyp{}closed.
Since $\clSM = \clonegen{\mu}$, by Lemma~\ref{lem:left-stab-gen}, it remains to show that $C \cup \clVako$ is stable under left composition with $\{\mu\}$.
Let $f, g, h \in C \cup \clVako$, all $n$\hyp{}ary.
If $f$, $g$, and $h$ are all in $C$, then $\mu(f, g, h) \in C$ because $C$ is stable under left composition with $\clSM$.
If two of $f$, $g$, and $h$ are in $C$ and the third is in $\clVako$ (that is, it is a constant $0$ function), say $f, g \in C$ and $h = 0$, then $\mu(f,g,h) = \mu(f,g,0) = f \wedge g \in C$.
If (at least) two of $f$, $g$, and $h$ are in $\clVako$, say $g = 0$ and $h = 0$, then $\mu(f,g,h) = \mu(f,0,0) = 0 \in \clVako$.
We conclude that $C \cup \clVako$ is $(\clIc,\clSM)$\hyp{}stable.

\ref{lem:IcSM-stable-constants:1}
The proof is similar to that of \ref{lem:IcSM-stable-constants:0} and makes use of the fact that $\mu(f, g, 1) = f \vee g$.

\ref{lem:IcSM-stable-constants:01}
The proof is similar to the previous parts.
The only new case to consider is when $f, g, h \in C \cup \clVak$ and there are two different constant functions among $f$, $g$, and $h$, say $g = 0$ and $h = 1$.
Then $\mu(f,g,h) = \mu(f,0,1) = f \in C \cup \clVak$.
We conclude that $C \cup \clVak$ is $(\clIc,\clSM)$\hyp{}stable.
\end{proof}

Recall that an operation $f \in \mathcal{O}_A^{(n)}$ \emph{preserves} a relation $\rho \subseteq A^m$ if for all $(a_{i1}, \dots, a_{im}) \in \rho$ ($i \in \nset{n}$), we have $(f(a_{11}, \dots, a_{n1}), \dots, f(a_{1m}, \dots, a_{nm})) \in \rho$.

\begin{lemma}
\label{lem:mu-preserves}
The fuction $\mu$ preserves every binary relation on $\{0,1\}$.
\end{lemma}

\begin{proof}
Every function preserves the empty relation, so
assume $\emptyset \neq \rho \subseteq \{0,1\}^2$, and let $(a_1, b_1), (a_2, b_2), (a_3, b_3) \in S$.
Then there exist $p, q, r, s \in \nset{3}$ such that $p \neq q$, $r \neq s$, $a_p = a_q$, and $b_r = b_s$.
We have $\mu(a_1, a_2, a_3) = a_p = a_q$ and $\mu(b_1, b_2, b_3) = b_r = b_s$.
Since $\{p, q\} \cap \{r, s\} \neq \emptyset$, there is an $i \in \{p, q\} \cap \{r, s\}$, and we have
$(\mu(a_1, a_2, a_3), \mu(b_1, b_2, b_3)) = (a_i, b_i) \in \rho$.
\end{proof}

\begin{proposition}
\label{prop:SM-stable:sufficiency}
The classes listed in Theorem~\ref{thm:SM-stable} are $(\clIc,\clSM)$\hyp{}stable.
\end{proposition}

\begin{proof}
Since intersections of $(\clIc,\clSM)$\hyp{}stable classes are $(\clIc,\clSM)$\hyp{}stable,
it is enough to prove the claim for the meet\hyp{}irreducible classes.
One can read off from Figure~\ref{fig:SM-stable} that the meet\hyp{}irreducible classes are
$\clAll$, $\clEiio$, $\clEioi$, $\clEiii$, $\clEioo$, $\clOXC$, $\clIXC$, $\clXOC$, $\clXIC$,
$\clSmin$, $\clSmaj$, $\clM$, $\clMneg$, $\clU$, $\clW$, $\clUneg$, $\clWneg$, $\clRefl$.
We can further simplify the task with the help of Lemma~\ref{lem:IcSM-stable-negation}, which asserts that if $C$ is $(\clIc,\clSM)$\hyp{}stable, then so are $\overline{C}$, $C^\mathrm{n}$, and $C^\mathrm{d}$.
Thus the only classes we need to consider are the following eight:
$\clAll$, $\clEiio$, $\clEiii$, $\clOXC$, $\clSmin$, $\clM$, $\clU$, $\clRefl$.

The classes $\clAll$, $\clM$, and $\clU$ are clones containing $\clSM$, so they are obviously $(\clIc,\clSM)$\hyp{}stable.
The same holds for the class $\clOX$, which is also closed under left composition with $\{\mathord{\wedge}, \mathord{\vee}\}$, so it follows by Lemma~\ref{lem:IcSM-stable-constants} that the class $\clOXC$ is $(\clIc,\clSM)$\hyp{}stable.

The classes $\clEiio$ and $\clEiii$ are clearly minor\hyp{}closed, because for any $f \in \clAll$ and any $\sigma \colon \nset{m} \to \nset{n}$, it holds that $f_\sigma(\vect{0}) = f(\vect{0} \sigma) = f(\vect{0})$ and $f_\sigma(\vect{1}) = f(\vect{1} \sigma) = f(\vect{1})$.
It is easy to see that they are also closed under left composition with $\{\mu\}$, because $\mu$ preserves every subset of $\{0,1\}^2$ by Lemma~\ref{lem:mu-preserves}.

Consider now the class $\clSmin$.
Let $f \in \clSmin$ and $\sigma \colon \nset{m} \to \nset{n}$.
Let $\vect{a} \in \{0,1\}^m$.
We have $f_\sigma(\vect{a}) = f(\vect{a} \sigma)$ and $f_\sigma(\overline{\vect{a}}) = f(\overline{\vect{a}} \sigma)$.
Since $f \in \clSmin$, we have $f(\vect{a} \sigma) \wedge f(\overline{\vect{a} \sigma}) = 0$, and since $\overline{\vect{a} \sigma} = \overline{\vect{a}} \sigma$, it follows that $f_\sigma(\vect{a}) \wedge f_\sigma(\overline{\vect{a}}) = 0$, so $f_\sigma \in \clSmin$.
Let now $f, g, h \in \clSmin$, all $n$\hyp{}ary, and let $\varphi := \mu(f, g, h)$.
We need to show that $\varphi(\vect{a}) \wedge \varphi(\overline{\vect{a}}) = 0$ for every $\vect{a} \in \{0,1\}^n$.
Let $\vect{a} \in \{0,1\}^n$.
If $\varphi(\vect{a}) = 0$, then $\varphi(\vect{a}) \wedge \varphi(\overline{\vect{a}}) = 0$.
Assume that $\varphi(\vect{a}) = 1$.
Then at least two of $f(\vect{a})$, $g(\vect{a})$, and $h(\vect{a})$ are equal to $1$.
Since $f, g, h \in \clSmin$, we have $f(\vect{a}) \wedge f(\overline{\vect{a}}) = 0$, $g(\vect{a}) \wedge f(\overline{\vect{a}}) = 0$, $h(\vect{a}) \wedge h(\overline{\vect{a}}) = 0$, so at least two of $f(\overline{\vect{a}})$, $g(\overline{\vect{a}})$, and $h(\overline{\vect{a}})$ are equal to $0$.
Consequently $\varphi(\overline{\vect{a}}) = 0$, so $\varphi(\vect{a}) \wedge \varphi(\overline{\vect{a}}) = 1 \wedge 0 = 0$ also in this case.

Finally, consider the class $\clRefl$.
Let $f \in \clRefl$ and $\sigma \colon \nset{m} \to \nset{n}$.
We have $f_\sigma(\vect{a}) = f(\vect{a} \sigma) = f(\overline{\vect{a} \sigma}) = f(\overline{\vect{a}} \sigma) = f_\sigma(\overline{\vect{a}})$ for every $\vect{a} \in \{0,1\}^m$, so $f_\sigma \in \clRefl$.
Let now $f, g, h \in \clRefl$, all $n$\hyp{}ary, and let $\varphi := \mu(f, g, h)$.
We have $\varphi(\vect{a}) = \mu(f(\vect{a}), g(\vect{a}), h(\vect{a})) = \mu(f(\overline{\vect{a}}), g(\overline{\vect{a}}), h(\overline{\vect{a}})) = \varphi(\overline{\vect{a}})$ for every $\vect{a} \in \{0,1\}^n$, so $\varphi \in \clRefl$.
\end{proof}

It remains to show that there are no further $(\clIc,\clSM)$\hyp{}stable classes than the ones listed in Theorem~\ref{thm:SM-stable}.
To this end, we show that each class $K$ is generated by any subset $F$ of $K$ that is not included in any proper subclass of $K$, i.e., for each proper subclass $C$ of $K$, the set $F$ contains an element of $K \setminus C$.
Since there are only a finite number of classes, every proper subclass of $K$ is included in a lower cover of $K$, and therefore it is sufficient to consider only lower covers of $K$.
This will be established in the propositions that comprise the remainder of this section.

Among our main tools are stratified terms and Lemma~\ref{lem:str-terms}, which provides a fairly simple way to deal with classes that are clones.
Lemma~\ref{lem:str-terms} is a powerful, generic tool, but, unfortunately, sometimes it may be a bit tricky to build all the stratified terms that are required for the application of the lemma.
We need to develop another tool, tailored for $(\clIc,\clSM)$\hyp{}stability, that is easier to use in such cases and is applicable also for classes that are not clones.

For any $S \subseteq \{0,1\}^n$, let
$\overline{S} := \{ \overline{\vect{a}} \mid \vect{a} \in S \}$,
${\uparrow} S := \{ \vect{u} \in \{0,1\}^n \mid \exists \vect{a} \in S \colon \vect{a} \leq \vect{u} \}$,
${\downarrow} S := \{ \vect{u} \in \{0,1\}^n \mid \exists \vect{a} \in S \colon \vect{u} \leq \vect{a} \}$.

\begin{lemma}
\label{lem:extend-SM}
Let $T$ and $F$ be subsets of $\{0,1\}^n$ \textup{(}possibly empty\textup{)} such that $\vect{u} \nleq \overline{\vect{u}'}$ for all $\vect{u}, \vect{u}' \in T$, $\overline{\vect{v}} \nleq \vect{v}'$ for all $\vect{v}, \vect{v}' \in F$, and $\vect{u} \nleq \vect{v}$ for all $\vect{u} \in T$ and $\vect{v} \in F$.
Then the following statements hold.
\begin{enumerate}[label=\upshape{(\roman*)}]
\item\label{lem:extend-SM:intersection}
${\uparrow} T \cap {\downarrow} F = \emptyset$,
${\uparrow} T \cap {\downarrow} \overline{T} = \emptyset$ and ${\uparrow} \overline{F} \cap {\downarrow} F = \emptyset$.
\item\label{lem:extend-SM:upset}
There exists an upset $U$ of $\{0,1\}^n$ of cardinality $2^{n-1}$ such that $T \subseteq U$, $F \subseteq \overline{U}$, and $U \cap \overline{U} = \emptyset$.
\item\label{lem:extend-SM:function}
Consequently, there exists an $n$\hyp{}ary function $f \in \clSM$ such that $f(\vect{a}) = 1$ for all $\vect{a} \in T$ and $f(\vect{b}) = 0$ for all $\vect{b} \in F$.
\end{enumerate}
\end{lemma}

\begin{proof}
\ref{lem:extend-SM:intersection}
Suppose, to the contrary, that there is an $\vect{a} \in {\uparrow} T \cap {\downarrow} F$.
Then there exist $\vect{u} \in T$ and $\vect{v} \in F$ such that $\vect{u} \leq \vect{a} \leq \vect{v}$, which contradicts the assumption that $\vect{u} \nleq \vect{v}$.
We conclude that ${\uparrow} T \cap {\downarrow} F = \emptyset$.

Now suppose, to the contrary, that there is an $\vect{a} \in {\uparrow} T \cap {\downarrow} \overline{T}$.
Then there exist $\vect{u} \in T$ and $\vect{v} \in \overline{T}$ such that $\vect{u} \leq \vect{a} \leq \vect{v}$.
Since $\vect{v} = \overline{\overline{\vect{v}}}$ and $\overline{\vect{v}} \in T$, this contradicts our hypothesis on $T$.
We conclude that ${\uparrow} T \cap {\downarrow} \overline{T} = \emptyset$, and a similar proof shows that ${\uparrow} \overline{F} \cap {\downarrow} F = \emptyset$.

\ref{lem:extend-SM:upset}
Note that the cardinality of any subset $S \subseteq \{0,1\}^n$ satisfying $S \cap \overline{S} = \emptyset$ is bounded above by $2^{n-1}$.
We claim that if $S$ is an upset of $\{0,1\}^n$ such that $S \cap \overline{S} = \emptyset$, then there exists an upset $U$ of cardinality $2^{n-1}$ such that $S \subseteq U$ and $U \cap \overline{U} = \emptyset$.
We prove the claim by induction on the number $2^{n-1} - \card{S}$.
The basis of induction is the case when $\card{S} = 2^{n-1}$; in this case the claim is trivial.
Suppose the claim holds when $\card{S} \geq m$ ($m \leq 2^{n-1}$).
Consider now an upset $S$ with $\card{S} = m - 1$ satisfying $S \cap \overline{S} = \emptyset$.
Let $\vect{w}$ be a maximal element of $\{0,1\}^n \setminus (S \cup \overline{S})$.
We claim that $S' := S \cup \{\vect{w}\}$ is an upset of $\{0,1\}^n$ satisfying $S' \cap \overline{S'} = \emptyset$.
Since $\vect{w} \in \{0,1\}^n \setminus (S \cup \overline{S})$, we also have $\overline{\vect{w}} \in \{0,1\}^n \setminus (S \cup \overline{S})$;
therefore
\[
\begin{split}
S' \cap \overline{S'}
&
= (S \cup \{\vect{w}\}) \cap (\overline{S} \cup \{\overline{\vect{w}}\})
\\ &
= (S \cap \overline{S}) \cup (S \cap \{\overline{\vect{w}}\}) \cup (\{\vect{w}\} \cap \overline{S}) \cup (\{\vect{w}\} \cap \{\overline{\vect{w}}\})
= \emptyset.
\end{split}
\]
It remains to show that $S'$ is an upset.
Since $S$ is an upset, we only need to verify that $\vect{u} \in S'$ for every $\vect{u} \in \{0,1\}^n$ such that $\vect{w} \leq \vect{u}$.
If $\vect{u} = \vect{w}$ then we are done, so assume that $\vect{w} < \vect{u}$.
Since $\vect{w}$ is a maximal element of $\{0,1\}^n \setminus (S \cup \overline{S})$, it follows that $\vect{u} \in S \cup \overline{S}$.
Suppose, to the contrary, that $\vect{u} \in \overline{S}$.
Since $\overline{S}$ is a downset and $\vect{w} < \vect{u}$, it follows that $\vect{w} in \overline{S}$.
But this is a contradiction because $\vect{w} \in \{0,1\}^n \setminus (S \cup \overline{S})$.
We conclude that $\vect{u} \in S \subseteq S'$, as desired.

We have shown that $S'$ is an upset satisfying $S' \cap \overline{S'} = \emptyset$.
By the induction hypothesis, there exists an upset $U$ of cardinality $2^{n-1}$ such that $S' \subseteq U$ and $U \cap \overline{U} = \emptyset$; this choice is also good for the given upset $S$ because $S \subseteq S'$.

The statement follows by considering the upset $S := {\uparrow} (T \cup \overline{F})$.
By part \ref{lem:extend-SM:intersection}, it satisfies
\[
\begin{split}
S \cap \overline{S}
&
= {\uparrow} (T \cup \overline{F}) \cap \overline{{\uparrow} (T \cup \overline{F})}
= ({\uparrow} T \cup {\uparrow} \overline{F}) \cap \overline{({\uparrow} T \cup {\uparrow} \overline{F})}
\\ &
= ({\uparrow} T \cup {\uparrow} \overline{F}) \cap (\overline{{\uparrow} T} \cup \overline{{\uparrow} \overline{F}})
= ({\uparrow} T \cup {\uparrow} \overline{F}) \cap ({\downarrow} \overline{T} \cup {\downarrow} F)
\\ &
= ({\uparrow} T \cap {\downarrow} \overline{T}) \cup ({\uparrow} T \cap {\downarrow} F) \cup ({\uparrow} \overline{F} \cap {\downarrow} \overline{T}) \cup ({\uparrow} \overline{F} \cap {\downarrow} F)
= \emptyset,
\end{split}
\]
so there exists an upset $U$ of cardinality $2^{n-1}$ such that $S \subseteq U$ and $U \cap \overline{U} = \emptyset$.
It clearly holds that $T \subseteq S \subseteq U$.
Furthermore $\overline{F} \subseteq S \subseteq U$, so $F \subseteq \overline{U}$.

\ref{lem:extend-SM:function}
Let $U$ be the upset provided by part \ref{lem:extend-SM:upset}, and define the function $f \colon \{0,1\}^n \to \{0,1\}$ by the rule $f(\vect{a}) = 1$ if $\vect{a} \in U$ and $f(\vect{a}) = 0$ if $\vect{a} \in \overline{U}$.
Since $U \cap \overline{U} = \emptyset$, the function is well defined and self\hyp{}dual.
Since $\card{U} = \card{\overline{U}} = 2^{n-1}$, it follows that $U \cup \overline{U} = \{0,1\}^n$; therefore $f$ is a total function.
Since $U$ is an upset, the function is monotone, and since $T \subseteq U$ and $F \subseteq \overline{U}$, the function takes the prescribed values on $T$ and $F$.
\end{proof}

\begin{definition}
\label{def:helpful}
Let $G$ be a set of Boolean functions, and for each $n \in \IN$, denote by $G_n$ the set of all $n$\hyp{}ary minors of functions in $G$.
Let $f$ be an $n$\hyp{}ary Boolean function.
We say that $f$ is \emph{$G$\hyp{}bisectable} if
the following three conditions hold:
\begin{enumerate}[label={\upshape(\Alph*)}]
\item\label{helpful:true} For all $\vect{a}, \vect{a}' \in f^{-1}(1)$ there exists a $\tau \in G_n$ such that $\tau(\vect{a}) = \tau(\vect{a}') = 1$.
\item\label{helpful:false} For all $\vect{b}, \vect{b}' \in f^{-1}(0)$ there exists a $\tau \in G_n$ such that $\tau(\vect{b}) = \tau(\vect{b}') = 0$.
\item\label{helpful:both} For all $\vect{a} \in f^{-1}(1)$ and for all $\vect{b} \in f^{-1}(0)$ there exists a $\tau \in G_n$ such that $\tau(\vect{a}) = 1$ and $\tau(\vect{b}) = 0$.
\end{enumerate}
We say that a class $C \subseteq \clAll$ is \emph{$G$\hyp{}bisectable} if every function in $C$ is $G$\hyp{}bisectable.
\end{definition}

\begin{lemma}
\label{lem:helpful:step}
Let $G \subseteq \clAll$ and $f \in \clAll$.
If $f$ is $G$\hyp{}bisectable, then $f \in \gen{G}$.
\end{lemma}

\begin{proof}
Let $f \in C$ with $\arity{f} = n$, and assume that $f$ is $G$\hyp{}bisectable.
Let $N$ be the cardinality of $G_n$ (there are only a finite number of $n$\hyp{}ary Boolean functions, so the set $G_n$ is certainly finite),
and assume that $\varphi_1, \dots, \varphi_N$ is an enumeration of the functions in $G_n$ in some fixed order.
Let $\varphi \colon \{0,1\}^n \to \{0,1\}^N$, $\varphi(\vect{a}) := (\varphi_1(\vect{a}), \dots, \varphi_N(\vect{a}))$.
Let $T := \varphi(f^{-1}(1))$ and $F := \varphi(f^{-1}(0))$.
Let $\vect{u}, \vect{v} \in T$; then there exist $\vect{a}, \vect{a}' \in f^{-1}(1)$ such that $f(\vect{a}) = \vect{u}$ and $f(\vect{a}') = \vect{v}$.
By condition \ref{helpful:true}, there exists an index $i \in \nset{N}$ such that $\varphi_i(\vect{a}) = \varphi_i(\vect{a}') = 1$; hence $\vect{u} = \varphi(\vect{a}) \nleq \overline{\varphi(\vect{a}')} = \overline{\vect{v}}$.
Similarly, if $\vect{u}, \vect{v} \in F$, then there exist $\vect{b}, \vect{b}' \in f^{-1}(0)$ such that $f(\vect{b}) = \vect{u}$ and $f(\vect{b}') = \vect{v}$.
By condition \ref{helpful:false}, there exists an index $j \in \nset{N}$ such that $\varphi_j(\vect{b}) = \varphi_j(\vect{b}') = 0$; hence $\overline{\vect{u}} = \overline{\varphi(\vect{b})} \nleq \varphi(\vect{b}') = \vect{v}$.
If $\vect{u} \in T$ and $\vect{v} \in F$, then there exist $\vect{a} \in f^{-1}(1)$ and $\vect{b} \in f^{-1}(0)$ such that $f(\vect{a}) = \vect{u}$ and $f(\vect{b}) = \vect{v}$.
By condition \ref{helpful:both}, there exists an index $k \in \nset{N}$ such that $\varphi_k(\vect{a}) = 1$ and $\varphi_k(\vect{b}) = 0$; hence $\vect{u} = \varphi(\vect{a}) \nleq \varphi(\vect{b}) = \vect{v}$.
Therefore the sets $T$ and $F$ satisfy the hypotheses of Lemma~\ref{lem:extend-SM}, and it follows that there exists an $N$\hyp{}ary function $h \in \clSM$ such that $T \subseteq h^{-1}(1)$ and $F \subseteq h^{-1}(0)$.
Then it holds that $f = h \circ \varphi = h(\varphi_1, \dots, \varphi_N)$.
Since $h \in \clSM$ and $\varphi_1, \dots, \varphi_N \in G \, \clIc$, we conclude with the help of Lemma~\ref{lem:F-closure} that $f \in \clSM (G \, \clIc) = \gen{G}$.
\end{proof}

\begin{lemma}
\label{lem:helpful}
Let $K$ be a $(\clIc,\clSM)$\hyp{}stable class, $F \subseteq K$, and $G \subseteq \gen{F}$.
\begin{enumerate}[label=\upshape{(\roman*)}]
\item\label{lem:helpful:general}
If $K \subseteq \gen{G}$, then $\gen{F} = K$.
\item\label{lem:helpful:bisectable}
If $K$ is $G$\hyp{}bisectable, then $\gen{F} = K$.
\end{enumerate}
\end{lemma}

\begin{proof}
\ref{lem:helpful:general}
Since $G \subseteq \gen{F}$, it follows by the general properties of closure operators that $\gen{G} \subseteq \gen{\gen{F}} = \gen{F}$.
Since $F \subseteq K$ and $K$ is $(\clIc,\clSM)$\hyp{}stable, we have $\gen{F} \subseteq \gen{K} = K$.
Putting these inclusions together with the assumption $K \subseteq \gen{G}$, we obtain $K \subseteq \gen{G} \subseteq \gen{F} \subseteq K$.

\ref{lem:helpful:bisectable}
We have $K \subseteq \gen{G}$ by Lemma~\ref{lem:helpful:step}, so the statement follows by part \ref{lem:helpful:general}.
\end{proof}

We now have the tools at hand for showing that each class listed in Theorem~\ref{thm:SM-stable} is generated by any subset that is suggested by Figure~\ref{fig:SM-stable}.
We start from the bottom of the lattice and proceed gradually upwards.

\begin{proposition}
\label{prop:empty}
$\gen{\emptyset} = \clEmpty$.
\end{proposition}

\begin{proof}
Trivial.
\end{proof}

\begin{proposition}
\label{prop:constants}
\leavevmode
\begin{enumerate}[label=\upshape{(\roman*)}]
\item\label{prop:constants:D0Ca}
Let $a \in \{0,1\}$.
For any $f \in \clVaka{a}$, we have $\gen{f} = \clVaka{a}$.
\item\label{prop:constants:D0}
For any $f, g \in \clVak$ with $f \notin \clVako$ and $g \notin \clVaki$, we have $\gen{f,g} = \clVak$.
\end{enumerate}
\end{proposition}

\begin{proof}
\ref{prop:constants:D0Ca}
We have $f = \cf{n}{a}$ for some $n \in \IN_{+}$.
We obtain all constant functions taking value $a$ as minors of $f$, so
$\clVaka{a} \subseteq \gen{f} \subseteq \clVaka{a}$.

\ref{prop:constants:D0}
We have $f = \cf{m}{1}$ and $g = \cf{n}{0}$ for some $m, n \in \IN_{+}$.
By part \ref{prop:constants:D0Ca}, $\gen{f} = \clVaki$ and $\gen{g} = \clVako$, so
$\clVak = \clVako \cup \clVaki = \gen{g} \cup \gen{f} \subseteq \gen{f,g} \subseteq \clVak$.
\end{proof}

\begin{lemma}
\label{lem:apu:minors}
\leavevmode
\begin{enumerate}[label=\upshape{(\roman*)}]
\item\label{lem:apu:minors:01}
If $f \in \clOI$, then $\id$ is a minor of $f$.
\item\label{lem:apu:minors:10}
If $f \in \clIO$, then $\neg$ is a minor of $f$.
\item\label{lem:apu:minors:00}
If $f \in \clOO$, then $0$ is a minor of $f$.
\item\label{lem:apu:minors:11}
If $f \in \clII$, then $1$ is a minor of $f$.
\end{enumerate}
\end{lemma}

\begin{proof}
By identifying all arguments, we obtain the unary minor $f'$ of $f$ that satisfies $f'(a) = f(a, \dots, a)$ for $a \in \{0,1\}$.
\end{proof}

\begin{proposition}
\label{prop:SM}
For any $f \in \clSM$, we have $\gen{f} = \clSM$.
\end{proposition}

\begin{proof}
We have $\id \in \gen{f}$ by Lemma~\ref{lem:apu:minors}\ref{lem:apu:minors:01}, and consequently $\clSM = \clSM \, \clIc = \clSM (\{\id\} \clIc) = \gen{\id} \subseteq \gen{f} \subseteq \clSM$.
\end{proof}

\begin{lemma}
\label{lem:apu:minors:01-M}
If $f \in \clOI \setminus \clM$, then $f$ has a ternary minor $f'$ satisfying $f'(0,0,0) = 0$, $f'(1,0,0) = 1$, $f'(1,1,0) = 0$, $f'(1,1,1) = 1$.
\end{lemma}

\begin{proof}
Since $f \in \clOI$, we have $f(\vect{0}) = 0$ and $f(\vect{1}) = 1$.
Since $f \notin \clM$, there exist tuples $\vect{a}, \vect{b} \in \{0,1\}^n$ such that $\vect{0} < \vect{a} < \vect{b} <  \vect{1}$ and $1 = \vect{a} > \vect{b} = 0$.
Without loss of generality, we may assume that $\vect{a} = 1^i 0^{j+k}$ and $\vect{b} = 1^{i+j} 0^k$ for some $i, j, k \in \IN_{+}$.
Let $f'$ be the ternary minor obtained from $f$ by identifying three blocks of arguments: the first $i$ arguments, the next $j$ arguments, and the last $k$ arguments.
Then $f'(0,0,0) = f(\vect{0}) = 0$, $f'(1,0,0) = f(\vect{a}) = 1$, $f'(1,1,0) = f(\vect{b}) = 0$, and $f'(1,1,1) = f(\vect{1}) = 1$.
\end{proof}

\begin{proposition}
\label{prop:Sc}
For any $f \in \clSc \setminus \clSM$, we have $\gen{f} = \clSc$.
\end{proposition}

\begin{proof}
By Lemma~\ref{lem:apu:minors:01-M}, there is a ternary minor $f' \leq f$ such that $f'(0,0,0) = 0$, $f'(1,0,0) = 1$, $f'(1,1,0) = 0$, $f(1,1,1)$.
Moreover, since $\clSc$ is minor\hyp{}closed, the function $f'$ is self\hyp{}dual, so $f'(0,1,1) = 0$ and $f'(0,0,1) = 1$.
By Lemma~\ref{lem:helpful} it suffices to show that $\clSc$ is $G$\hyp{}bisectable for $G := \{f'\}$.
Note that $\id \leq f'$.
So, let $\theta \in \clSc$.
We verify that conditions \ref{helpful:true}, \ref{helpful:false}, and \ref{helpful:both} of Definition~\ref{def:helpful} are satisfied.

Condition \ref{helpful:true}:
Let $\vect{a}, \vect{a}' \in \theta^{-1}(1)$.
We have $\vect{a} \neq \overline{\vect{a}'}$ and $\vect{0} \notin \{\vect{a}, \vect{a}'\}$.
Then one of the following cases must occur, for some $i$, $j$, $k$:
(i) $\big( \begin{smallmatrix} a_i \\ a'_i \end{smallmatrix} \big) = \big( \begin{smallmatrix} 1 \\ 1 \end{smallmatrix} \big)$, in which case $\id_i(\vect{a}) = \id_i(\vect{a}') = 1$;
(ii) $\big( \begin{smallmatrix} a_i & a_j & a_k \\ a'_i & a'_j & a'_k \end{smallmatrix} \big) = \big( \begin{smallmatrix} 0 & 1 & 0 \\ 0 & 0 & 1 \end{smallmatrix} \big)$, in which case $f'_{jik}(\vect{a}) = f'_{jik}(\vect{a}') = 1$.

Condition \ref{helpful:false}:
Let $\vect{b}, \vect{b}' \in \theta^{-1}(0)$.
We have $\vect{b} \neq \overline{\vect{b}'}$ and $\vect{1} \notin \{\vect{b}, \vect{b}'\}$.
Then one of the following cases must occur, for some $i$, $j$, $k$:
(i) $\big( \begin{smallmatrix} b_i \\ b'_i \end{smallmatrix} \big) = \big( \begin{smallmatrix} 0 \\ 0 \end{smallmatrix} \big)$, in which case $\id_i(\vect{b}) = \id_i(\vect{b}') = 0$;
(ii) $\big( \begin{smallmatrix} b_i & b_j & b_k \\ b'_i & b'_j & b'_k \end{smallmatrix} \big) = \big( \begin{smallmatrix} 1 & 0 & 1 \\ 1 & 1 & 0 \end{smallmatrix} \big)$, in which case $f'_{kij}(\vect{b}) = f'_{kij}(\vect{b}') = 0$.

Condition \ref{helpful:both}:
Let $\vect{a} \in \theta^{-1}(1)$ and $\vect{b} \in \theta^{-1}(0)$.
We have $\vect{0} \neq \vect{a} \neq \vect{b} \neq \vect{1}$.
Then one of the following cases must occur, for some $i$, $j$, $k$:
(i) $\big( \begin{smallmatrix} a_i \\ b_i \end{smallmatrix} \big) = \big( \begin{smallmatrix} 1 \\ 0 \end{smallmatrix} \big)$, in which case $\id_i(\vect{a}) = 1$, $\id_i(\vect{b}) = 0$;
(ii) $\big( \begin{smallmatrix} a_i & a_j & a_k \\ b_i & b_j & b_k \end{smallmatrix} \big) = \big( \begin{smallmatrix} 0 & 1 & 0 \\ 1 & 1 & 0 \end{smallmatrix} \big)$, in which case $f'_{jik}(\vect{a}) = 1$, $f'_{jik}(\vect{b}) = 0$.
\end{proof}

\begin{proposition}
\label{prop:S}
For any $f, g \in \clS$ with $f \notin \clSc$ and $g \notin \clScneg$, we have\linebreak $\gen{f,g} = \clS$.
\end{proposition}

\begin{proof}
We have $f \in \clScneg$ and $g \in \clSc$.
By Lemma~\ref{lem:apu:minors}, we have $\id \leq g$ and $\neg \leq f$.
By Lemma~\ref{lem:helpful}\ref{lem:helpful:general}, it suffices to show that $\clS \subseteq \gen{\id, \neg}$.
Since $\clS = \clonegen{\mu, \neg}$, this follows by Lemma~\ref{lem:str-terms} and Remark~\ref{rem:str-terms}
from the following equivalences of terms:
\begin{gather*}
\neg(\neg(x_1)) \equiv \id(x_1), \\
\neg(\mu(x_1, x_2, x_3)) \equiv \mu(\neg(x_1), \neg(x_2), \neg(x_3)).
\qedhere
\end{gather*}
\end{proof}

\begin{proposition}
\label{prop:R}
\leavevmode
\begin{enumerate}[label=\upshape{(\roman*)}]
\item\label{prop:R:R00}
For any $f \in \clReflOO \setminus \clVako$, we have $\gen{f} = \clReflOO$.
\item\label{prop:R:R00C}
For any $f, g \in \clReflOOC$ such that $f \notin \clReflOO$ and $g \notin \clVak$, we have $\gen{f, g} = \clReflOOC$.
\item\label{prop:R:R}
For any $f, g \in \clRefl$ such that $f \notin \clReflOOC$ and $g \notin \clReflIIC$, we have $\gen{f, g} = \clRefl$.
\end{enumerate}
\end{proposition}

\begin{proof}
\ref{prop:R:R00}
Since $f \in \clOO$, we have $f(\vect{0}) = f(\vect{1}) = 0$.
Since $f \notin \clVako$, there exists a tuple $\vect{a}$ such that $f(\vect{a}) = 1$.
Moreover, since $f \in \clRefl$, we have $f(\overline{\vect{a}}) = 1$.
By identifying arguments in a suitable way, we obtain $+$ as a minor of $f$.
By Lemma~\ref{lem:helpful}, it suffices to show that $\clReflOO$ is $G$\hyp{}bisectable for $G := \{\mathord{+}\}$.
So, let $\theta \in \clReflOO$.

Condition \ref{helpful:true}:
Let $\vect{a}, \vect{a}' \in \theta^{-1}(1)$.
We have $\{\vect{a}, \vect{a}'\} \cap \{\vect{0}, \vect{1}\} = \emptyset$.
Then one of the following cases must occur, for some $i$, $j$:
(i) $\big( \begin{smallmatrix} a_i & a_j \\ a'_i & a'_j \end{smallmatrix} \big) = \big( \begin{smallmatrix} 0 & 1 \\ 0 & 1 \end{smallmatrix} \big)$, in which case $\mathord{+}_{ij}(\vect{a}) = \mathord{+}_{ij}(\vect{a}') = 1$;
(i) $\big( \begin{smallmatrix} a_i & a_j \\ a'_i & a'_j \end{smallmatrix} \big) = \big( \begin{smallmatrix} 0 & 1 \\ 1 & 0 \end{smallmatrix} \big)$, in which case $\mathord{+}_{ij}(\vect{a}) = \mathord{+}_{ij}(\vect{a}') = 1$.

Condition \ref{helpful:false} holds because $0 \leq \mathord{+} \in G$.

Condition \ref{helpful:both}:
Let $\vect{a} \in \theta^{-1}(1)$ and $\vect{b} \in \theta^{-1}(0)$.
We have $\vect{a} \neq \vect{b}$, $\vect{a} \neq \overline{\vect{b}}$, $\vect{a} \notin \{\vect{0}, \vect{1}\}$.
Then one of the following cases must occur, for some $i$, $j$:
(i) $\big( \begin{smallmatrix} a_i & a_j \\ b_i & b_j \end{smallmatrix} \big) = \big( \begin{smallmatrix} 0 & 1 \\ 1 & 1 \end{smallmatrix} \big)$, in which case $\mathord{+}_{ij}(\vect{a}) = 1$, $\mathord{+}_{ij}(\vect{b}) = 0$;
(ii) $\big( \begin{smallmatrix} a_i & a_j \\ b_i & b_j \end{smallmatrix} \big) = \big( \begin{smallmatrix} 1 & 0 \\ 0 & 0 \end{smallmatrix} \big)$, in which case $\mathord{+}_{ij}(\vect{a}) = 1$, $\mathord{+}_{ij}(\vect{a}') = 0$.

\ref{prop:R:R00C}
We have $f \in \clVaki$ and $g \in \clReflOO \setminus \clVak$.
By Lemma~\ref{prop:constants}\ref{prop:constants:D0Ca} and part \ref{prop:R:R00}, it holds that $\gen{f} = \clVaki$ and $\gen{g} = \clReflOO$.
Therefore
$\clReflOOC
= \clReflOO \cup \clVaki
= \gen{f} \cup \gen{g}
\subseteq \gen{f, g}
\subseteq \clReflOOC$.

\ref{prop:R:R}
We have $f \in \clReflII \setminus \clVak$ and $g \in \clReflOO \setminus \clVak$, so
by part \ref{prop:R:R00} and Lemma~\ref{lem:IcSM-stable-negation},
$\gen{g} = \clReflOO$ and $\gen{f} = \clReflII$.
Therefore
$\clRefl
= \clReflOO \cup \clReflII
= \gen{f} \cup \gen{g}
\subseteq \gen{f, g}
\subseteq \clRefl$.
\end{proof}

\begin{lemma}
\label{lem:apu:minors:Smin00-C}
If $f \in \clSminOO \setminus \clVak$, then $\mathord{\nrightarrow}$ is a minor of $f$.
\end{lemma}

\begin{proof}
Since $f \in \clOO$, we have $f(\vect{0}) = f(\vect{1}) = 0$.
Since $f \notin \clVak$, there is a tuple $\vect{a}$ such that $f(\vect{a}) = 1$.
Since $f \in \clSmin$, we must have $f(\overline{\vect{a}}) = 0$.
Assume, without loss of generality, that $\vect{a} = 1^i 0^{n-i}$.
By identifying the first $i$ arguments and the last $n-i$ arguments, we obtain $\mathord{\nrightarrow}$ as a minor of $f$.
\end{proof}

\begin{proposition}
\label{prop:UWneg}
For any $f \in (\clUWneg) \setminus \clVako$, we have $\gen{f} = \clUWneg$.
\end{proposition}

\begin{proof}
Let $f \in (\clUWneg) \setminus \clVako$.
By Lemma~\ref{lem:apu:minors:Smin00-C}, we have $\mathord{\nrightarrow} \leq f$.
By Lemma~\ref{lem:helpful} it suffices to show that $\clUWneg$ is $G$\hyp{}bisectable for $G := \{\mathord{\nrightarrow}\}$.
So, let $\theta \in \clUWneg$.

Condition \ref{helpful:true}:
Let $\vect{a}, \vect{a}' \in \theta^{-1}(1)$.
Since $\theta \in \clU$, there exists an $i$ such that $a_i = a'_i = 1$, and since $\theta \in \clWneg$, there exists a $j$ such that $a_j = a'_j = 0$, and we have $\mathord{\nrightarrow_{ij}}(\vect{a}) = \mathord{\nrightarrow_{ij}}(\vect{b}) = 1$.

Condition \ref{helpful:false} holds because $0 \leq \mathord{\nrightarrow} \in G$.

Condition \ref{helpful:both}:
Let $\vect{a} \in \theta^{-1}(1)$ and $\vect{b} \in \theta^{-1}(0)$.
We have $\vect{a} \neq \vect{b}$ and $\vect{a} \notin \{\vect{0}, \vect{1}\}$.
Then there exist an $i$ such that $a_i \neq b_i$
and a $j$ such that $a_j \neq a_i$.
Then we have $\tau(\vect{a}) = 1$ and $\tau(\vect{b}) = 0$ for some $\tau \in \{\mathord{\nrightarrow}_{ij}, \mathord{\nrightarrow}_{ji}\}$.
\end{proof}

\begin{proposition}
\label{prop:UC}
For any $f \in \clUOO \setminus (\clUWneg)$, we have $\gen{f} = \clUOO$.
\end{proposition}

\begin{proof}
Let $f \in \clUOO \setminus (\clUWneg)$.
Since $f \in \clOO$, we have $f(\vect{0}) = f(\vect{1}) = 0$.
Since $f \notin \clWneg$, there exist tuples $\vect{a}$ and $\vect{b}$ such that $f(\vect{a}) = f(\vect{b}) = 1$ and $\vect{a} \vee \vect{b} = \vect{1}$.
Since $f \in \clU$, it holds that $\vect{a} \wedge \vect{b} \neq \vect{0}$ and $f(\overline{\vect{a}}) = f(\overline{\vect{b}}) = 0$.
By identifying arguments in a suitable way, we see that $f$ has a ternary minor $f'$ satisfying
$f'(0,0,0) = 0$,
$f'(0,0,1) = 0$,
$f'(0,1,0) = 0$,
$f'(1,1,1) = 0$,
$f'(1,1,0) = 1$,
$f'(1,0,1) = 1$.
By Lemma~\ref{lem:helpful} it suffices to show that $\clUOO$ is $G$\hyp{}bisectable for $G := \{f'\}$.
Note that $\mathord{\nrightarrow} = f'_{112}$.
So, let $\theta \in \clUOO$.

Condition \ref{helpful:true}:
Let $\vect{a}, \vect{a}' \in \theta^{-1}(1)$.
We have $\vect{a} \wedge \vect{a}' \neq \vect{0}$ and $\{\vect{a}, \vect{a}'\} \cap \{\vect{0}, \vect{1}\} = \emptyset$.
Then one of the following cases must occur, for some $i$, $j$, $k$:
(i) $\big( \begin{smallmatrix} a_i & a_j \\ a'_i & a'_j \end{smallmatrix} \big) = \big( \begin{smallmatrix} 1 & 0 \\ 1 & 0 \end{smallmatrix} \big)$, in which case $\mathord{\nrightarrow}_{ij}(\vect{a}) = \mathord{\nrightarrow}_{ij}(\vect{a}') = 1$;
(ii) $\big( \begin{smallmatrix} a_i & a_j & a_k \\ a'_i & a'_j & a'_k \end{smallmatrix} \big) = \big( \begin{smallmatrix} 1 & 0 & 1 \\ 1 & 1 & 0 \end{smallmatrix} \big)$, in which case  $f'_{ijk}(\vect{a}) = f'_{ijk}(\vect{a}') = 1$.

Condition \ref{helpful:false} holds because $0 \leq f' \in G$.

Condition \ref{helpful:both}: Shown as in the proof of Proposition~\ref{prop:UWneg}.
\end{proof}

\begin{proposition}
\label{prop:McU2}
For any $f \in \clMcU \setminus \clSM$, we have $\gen{f} = \clMcU$.
\end{proposition}

\begin{proof}
Let $f \in \clMcU \setminus \clSM$.
Since $f \in \clOI$, we have $f(\vect{0}) = 0$ and $f(\vect{1}) = 1$.
Since $f \notin \clS$, there exists a tuple $\vect{u}$ such that $f(\vect{u}) = f(\overline{\vect{u}})$. 
Since $\vect{u} \wedge \overline{\vect{u}} = \vect{0}$ and $f \in \clU$, we must have $f(\vect{u}) = f(\overline{\vect{u}}) = 0$.
By identifying arguments in a suitable way, we obtain $\wedge$ as a minor of $f$.
Note that $\id \leq \mathord{\wedge}$.

By Lemma~\ref{lem:helpful}\ref{lem:helpful:general}, it suffices to show that $\clMcU \subseteq \gen{\id, \mathord{\wedge}}$.
Since $\clMcU = \clonegen{\mu, \mathord{\wedge}}$, this follows by Lemma~\ref{lem:str-terms} and Remark~\ref{rem:str-terms}
from the following equivalences of terms:
\begin{gather*}
\mathord{\wedge}(\mu(x_1, x_2, x_3), x_4) \equiv \mu(\mathord{\wedge}(x_1, x_4), \mathord{\wedge}(x_2, x_4), \mathord{\wedge}(x_3, x_4)), \\
\mathord{\wedge}(\mathord{\wedge}(x_1, x_2), x_3) \equiv \mu(\mathord{\wedge}(x_1, x_2), \mathord{\wedge}(x_1, x_3), \mathord{\wedge}(x_2, x_3)).
\qedhere
\end{gather*}
\end{proof}

\begin{proposition}
\label{prop:MU2}
For any $f, g \in \clMU$ with $f \notin \clMcU$ and $g \notin \clVako$, we have $\gen{f,g} = \clMU$.
\end{proposition}

\begin{proof}
Since $\clMU = \clMcU \cup \clVako$ and $\clMcU \cap \clVako = \emptyset$, we have $f \in \clVako$ and $g \in \clMcU$.
Therefore $0 \leq f$ and $\id \leq g$.
By Lemma~\ref{lem:helpful}\ref{lem:helpful:general}, it suffices to show that $\clMU \subseteq \gen{\id, 0}$.
Since $\clMU = \clonegen{\mu, 0}$, this follows by Lemma~\ref{lem:str-terms} and Remark~\ref{rem:str-terms}
from the following equivalences of terms:
\begin{gather*}
0(0(x_1)) \equiv \mu(0(x_1), 0(x_1), 0(x_1)), \\
0(\mu(x_1, x_2, x_3)) \equiv \mu(0(x_1), 0(x_2), 0(x_3)).
\qedhere
\end{gather*}
\end{proof}

\begin{proposition}
\label{prop:TcU2}
For any $f \in \clTcU \setminus \clMcU$, we have $\gen{f} = \clTcU$.
\end{proposition}

\begin{proof}
By Lemma~\ref{lem:apu:minors:01-M}, $f$ has a ternary minor $f'$ satisfying $f'(0,0,0) = 0$, \linebreak $f'(1,0,0) = 1$, $f'(1,1,0) = 0$, $f'(1,1,1) = 1$.
Since $\clTcU$ is minor\hyp{}closed, we have $f' \in \clTcU$, and therefore $f'(\vect{u}) = 0$ for all $\vect{u} \leq (0,1,1)$, that is,
$f'(0,1,0) = 0$, $f'(0,0,1) = 0$, and $f'(0,1,1) = 0$.
By Lemma~\ref{lem:helpful} it suffices to show that $\clTcU$ is $G$\hyp{}bisectable for $G := \{f'\}$.
Note that $\id \leq f'$ and $\mathord{\wedge} = f'_{112}$.
So, let $\theta \in \clTcU$.

Condition \ref{helpful:true}:
Let $\vect{a}, \vect{a}' \in \theta^{-1}(1)$.
We have $\vect{a} \wedge \vect{a}' \neq \vect{0}$.
Therefore $\id_i(\vect{a}) = \id_i(\vect{a}') = 1$ for some $i$.

Condition \ref{helpful:false}:
Let $\vect{b}, \vect{b}' \in \theta^{-1}(0)$.
Then $\vect{b}$ and $\vect{b}'$ are distinct from $\vect{1}$, so there exist indices $i$ and $j$ such that $b_i = b'_j = 0$.
Then we have $\mathord{\wedge}_{ij}(\vect{b}) = \mathord{\wedge}_{ij}(\vect{b}') = 0$.

Condition \ref{helpful:both}:
Let $\vect{a} \in \theta^{-1}(1)$ and $\vect{b} \in \theta^{-1}(0)$.
We have $\vect{a} \neq \vect{b}$, $\vect{a} \neq \vect{0}$, $\vect{b} \neq \vect{1}$.
Then one of the following cases must occur, for some $i$, $j$, $k$:
(i) $\big( \begin{smallmatrix} a_i \\ b_i \end{smallmatrix} \big) = \big( \begin{smallmatrix} 1 \\ 0 \end{smallmatrix} \big)$, in which case $\id_i(\vect{a}) = 1$, $\id_i(\vect{b}) = 0$;
(ii) $\big( \begin{smallmatrix} a_i & a_j & a_k \\ b_i & b_j & b_k \end{smallmatrix} \big) = \big( \begin{smallmatrix} 0 & 1 & 0 \\ 1 & 1 & 0 \end{smallmatrix} \big)$, in which case $f'_{jik}(\vect{a}) = 1$, $f'_{jik}(\vect{b}) = 0$.
\end{proof}

\begin{proposition}
\label{prop:TcU2C}
For any $f, g \in \clTcUCO$ with $f \notin \clTcU$ and $g \notin \clMU$, we have $\gen{f,g} = \clTcUCO$.
\end{proposition}

\begin{proof}
Since $\clTcU$ and $\clVako$ are disjoint and $\clVako \subseteq \clMU$, we have $f \in \clVako$ and $g \in \clTcU \setminus \clMcU$.
By Propositions~\ref{prop:constants}\ref{prop:constants:D0Ca} and \ref{prop:TcU2}, we have $\gen{f} = \clVako$ and $\gen{g} = \clTcU$.
Consequently,
$\clTcUCO
= \gen{g} \cup \gen{f}
\subseteq \gen{f, g}
\subseteq \clTcUCO$.
\end{proof}

\begin{proposition}
\label{prop:U2}
For any $f, g \in \clU$ with $f \notin \clUOO$ and $g \notin \clTcUCO$, we have $\gen{f,g} = \clU$.
\end{proposition}

\begin{proof}
We have $f \in \clTcU$, so $\id \leq f$.
Since $g \in \clUOO \setminus \clVak \subseteq \clSminOO \setminus \clVak$, we have $\mathord{\nrightarrow} \leq g$ by Lemma~\ref{lem:apu:minors:Smin00-C}.
By Lemma~\ref{lem:helpful}\ref{lem:helpful:general}, it suffices to show that $\clU \subseteq \gen{\id, \mathord{\nrightarrow}}$.
Since $\clU = \clonegen{\mu, \mathord{\nrightarrow}}$, this follows by Lemma~\ref{lem:str-terms} and Remark~\ref{rem:str-terms}
from the following equivalences of terms:
\begin{gather*}
\mathord{\nrightarrow}(\mu(x_1, x_2, x_3), x_4) \equiv \mu(\mathord{\nrightarrow}(x_1, x_4), \mathord{\nrightarrow}(x_2, x_4), \mathord{\nrightarrow}(x_3, x_4)), \\
\mathord{\nrightarrow}(x_1, \mu(x_2, x_3, x_4)) \equiv \mu(\mathord{\nrightarrow}(x_1, x_2), \mathord{\nrightarrow}(x_1, x_3), \mathord{\nrightarrow}(x_1, x_4)), \\
\mathord{\nrightarrow}(\mathord{\nrightarrow}(x_1, x_2), x_3) \equiv \mu(\mathord{\nrightarrow}(x_1, x_2), \mathord{\nrightarrow}(x_1, x_3), \mathord{\nrightarrow}(x_2, x_1)), \\
\mathord{\nrightarrow}(x_1, \mathord{\nrightarrow}(x_2, x_3)) \equiv \mu(\id(x_1), \id(x_3), \mathord{\nrightarrow}(x_1, x_2)).
\qedhere
\end{gather*}
\end{proof}

\begin{proposition}
\label{prop:M}
\leavevmode
\begin{enumerate}[label=\upshape{(\roman*)}]
\item\label{prop:M:Mc}
For any $f, g \in \clMc$ with $f \notin \clMcU$ and $g \notin \clMcW$, we have $\gen{f, g} = \clMc$.
\item\label{prop:M:M0}
For any $f, g \in \clMo$ with $f \notin \clMc$ and $g \notin \clMU$, we have $\gen{f, g} = \clMo$.
\item\label{prop:M:M}
For any $f, g, h \in \clM$ with $f \notin \clMo$, $g \notin \clMi$, and $h \notin \clVak$, we have \linebreak $\gen{f, g, h} = \clM$.
\end{enumerate}
\end{proposition}

\begin{proof}
\ref{prop:M:Mc}
Since $f \in \clOI$, we have $f(\vect{0}) = 0$ and $f(\vect{1}) = 1$.
Since $f \notin \clU$, there exist tuples $\vect{a}$ and $\vect{b}$ such that $f(\vect{a}) = f(\vect{b}) = 1$ and $\vect{a} \wedge \vect{b} = \vect{0}$, i.e., $\vect{b} \leq \overline{\vect{a}}$.
Since $f \in \clM$, we have $f(\overline{\vect{a}}) = 1$.
By identifying arguments, we obtain $\vee$ as a minor of $f$, i.e., $\mathord{\vee} \in \gen{f}$.
In a similar way we can show that $\mathord{\wedge} \in \gen{g}$.
Note that $\id$ is a minor of $\vee$ and $\wedge$.
By Lemma~\ref{lem:helpful}\ref{lem:helpful:general}, it suffices to show that $\clMc \subseteq \gen{\id, \mathord{\wedge}, \mathord{\vee}}$.
Since $\clMc = \clonegen{\mu, \mathord{\wedge}, \mathord{\vee}}$, this follows by Lemma~\ref{lem:str-terms} and Remark~\ref{rem:str-terms}
from the following equivalences of terms:
\begin{gather*}
\mathord{\wedge}(\mathord{\wedge}(x_1, x_2), x_3) \equiv \mu(\mathord{\wedge}(x_1, x_2), \mathord{\wedge}(x_1, x_3), \mathord{\wedge}(x_2, x_3)), \\
\mathord{\wedge}(\mathord{\vee}(x_1, x_2), x_3) \equiv \mu(\mathord{\wedge}(x_1, x_3), \mathord{\wedge}(x_2, x_3), \id(x_3)), \\
\mathord{\wedge}(\mu(x_1, x_2, x_3), x_4) \equiv \mu(\mathord{\wedge}(x_1, x_4), \mathord{\wedge}(x_2, x_4), \mathord{\wedge}(x_3, x_4)), \\
\mathord{\vee}(\mathord{\vee}(x_1, x_2), x_3) \equiv \mu(\mathord{\vee}(x_1, x_2), \mathord{\vee}(x_1, x_3), \mathord{\vee}(x_2, x_3)), \\
\mathord{\vee}(\mathord{\wedge}(x_1, x_2), x_3) \equiv \mu(\mathord{\vee}(x_1, x_3), \mathord{\vee}(x_2, x_3), \id(x_3)), \\
\mathord{\vee}(\mu(x_1, x_2, x_3), x_4) \equiv \mu(\mathord{\vee}(x_1, x_4), \mathord{\vee}(x_2, x_4), \mathord{\vee}(x_3, x_4)).
\end{gather*}

\ref{prop:M:M0}
Since $f \in \clMo \setminus \clMc$, we have $f = 0$.
Since $g \notin \clU$, there exist tuples $\vect{a}$ and $\vect{b}$ such that $g(\vect{a}) = g(\vect{b}) = 1$ and $\vect{a} \wedge \vect{b} = \vect{0}$, i.e., $\vect{b} \leq \overline{\vect{a}}$.
Since $g \in \clM$, we have $g(\overline{\vect{a}}) = g(\vect{1}) = 1$.
Since $g \in \clOX$, we have $g(\vect{0}) = 0$.
It follows that $\id$ and $\mathord{\vee}$ are minors of $g$.
Since $\mu(\id_1, \id_2, 0) = \mathord{\wedge}$, we have $\mathord{\wedge} \in \gen{f, g}$.
It now follows from Proposition~\ref{prop:constants}\ref{prop:constants:D0Ca} and part \ref{prop:M:Mc} that
$\clMo
= \clMc \cup \clVako
= \gen{\mathord{\wedge}, \mathord{\vee}} \cup \gen{0}
\subseteq \gen{f, g}
\subseteq \clMo$.

\ref{prop:M:M}
We have $f = 1$, $g = 0$, and $h \in \clMc$, so $\id \leq h$.
Since $\mu(\id_1, \id_2, 0) = \mathord{\wedge}$ and $\mu(\id_1, \id_2, 1) = \mathord{\vee}$, we have $\mathord{\wedge}, \mathord{\vee} \in \gen{f, g, h}$.
It follows from Proposition~\ref{prop:constants}\ref{prop:constants:D0} and part \ref{prop:M:Mc} that
$\clM
= \clMc \cup \clVak
= \gen{\mathord{\wedge}, \mathord{\vee}} \cup \gen{0, 1}
\subseteq \gen{f, g, h}
\subseteq \clM$.
\end{proof}

\begin{proposition}
\label{prop:Smin00}
\leavevmode
\begin{enumerate}[label=\upshape{(\roman*)}]
\item\label{prop:Smin00:Smin00}
For any $f, g \in \clSminOO$ with $f \notin \clUOO$ and $g \notin \clWnegOO$, we have $\gen{f, g} = \clSminOO$.
\item\label{prop:Smin00:Smin01}
For any $f, g \in \clSminOI$ with $f \notin \clTcU$ and $g \notin \clSc$, we have $\gen{f, g} = \clSminOI$.
\end{enumerate}
\end{proposition}

\begin{proof}
\ref{prop:Smin00:Smin00}
Since $f \in \clOO$, we have $f(\vect{0}) = f(\vect{1}) = 0$.
Since $f \notin \clU$, there exist tuples $\vect{a}$ and $\vect{b}$ such that $f(\vect{a}) = f(\vect{b}) = 1$ and $\vect{a} \wedge \vect{b} = \vect{0}$, i.e., $\vect{a} \leq \overline{\vect{b}}$.
Since $f \in \clSmin$, we have $f(\overline{\vect{a}}) = f(\overline{\vect{b}}) = 0$.
Consequently, $\vect{a} \neq \overline{\vect{b}}$; therefore $\vect{a} < \overline{\vect{b}}$, so $\vect{a} \vee \vect{b} \neq \vect{1}$.
Similarly, $g \in \clSminOO \setminus \clWnegOO$ implies that $g(\vect{0}) = g(\vect{1}) = 0$ and there exist tuples $\vect{c}$ and $\vect{d}$ such that $g(\vect{c}) = g(\vect{d}) = 1$, $\vect{c} \vee \vect{d} = \vect{1}$, $\vect{c} \wedge \vect{d} \neq \vect{0}$.
By identifying arguments, we obtain ternary minors $f' \leq f$ and $g' \leq g$ satisfying
$f'(0,0,0) = 0$, $f'(1,0,0) = 1$, $f'(0,1,0) = 1$,
$f'(1,1,1) = 0$, $f'(0,1,1) = 0$, $f'(1,0,1) = 0$,
$g'(0,0,0) = 0$, $g'(1,0,0) = 0$, $g'(0,1,0) = 0$,
$g'(1,1,1) = 0$, $g'(0,1,1) = 1$, $g'(1,0,1) = 1$.
By Lemma~\ref{lem:helpful} it suffices to show that $\clSminOO$ is $G$\hyp{}bisectable with $G := \{f', g'\}$.
Note that $\mathord{\nrightarrow} = f'_{122}$.
So, let $\theta \in \clSminOO$.

Condition \ref{helpful:true}:
Let $\vect{a}, \vect{a}' \in \theta^{-1}(1)$.
We have $\vect{a} \neq \overline{\vect{a}'}$ and $\{\vect{a}, \vect{a}'\} \cap \{\vect{0}, \vect{1}\} = \emptyset$.
Then one of the following cases must occur, for some $i$, $j$, $k$:
(i) $\big( \begin{smallmatrix} a_i & a_j \\ a'_i & a'_j \end{smallmatrix} \big) = \big( \begin{smallmatrix} 1 & 0 \\ 1 & 0 \end{smallmatrix} \big)$, in which case $\mathord{\nrightarrow}_{ij}(\vect{a}) = \mathord{\nrightarrow}_{ij}(\vect{a}') = 1$;
(ii) $\big( \begin{smallmatrix} a_i & a_j & a_k \\ a'_i & a'_j & a'_k \end{smallmatrix} \big) = \big( \begin{smallmatrix} 0 & 1 & 0 \\ 0 & 0 & 1 \end{smallmatrix} \big)$, in which case $f'_{kji}(\vect{a}) = f'_{kji}(\vect{a}) = 1$;
(iii) $\big( \begin{smallmatrix} a_i & a_j & a_k \\ a'_i & a'_j & a'_k \end{smallmatrix} \big) = \big( \begin{smallmatrix} 1 & 0 & 1 \\ 1 & 1 & 0 \end{smallmatrix} \big)$, in which case $g'_{kji}(\vect{a}) = g'_{kji}(\vect{a}) = 1$.

Condition \ref{helpful:false} holds because $0 \leq f' \in G$.

Condition \ref{helpful:both}: Shown as in the proof of Proposition~\ref{prop:UWneg}.

\ref{prop:Smin00:Smin01}
Since $f \in \clOI$, we have $f(\vect{0}) = 0$ and $f(\vect{1}) = 1$.
Since $f \notin \clU$, there exist tuples $\vect{a}$ and $\vect{b}$ such that $f(\vect{a}) = f(\vect{b}) = 1$ and $\vect{a} \wedge \vect{b} = \vect{0}$.
Since $f \in \clSmin$, we must have $f(\overline{\vect{a}}) = f(\overline{\vect{b}}) = 0$; hence $\vect{a} \neq \overline{\vect{b}}$.
By identifying arguments, we obtain a ternary minor $f' \leq f$ satisfying
$f'(1,1,1) = 1$, $f'(1,0,0) = 1$, $f'(0,1,0) = 1$,
$f'(0,0,0) = 0$, $f'(0,1,1) = 0$, $f'(1,0,1) = 0$.
Since $g \in \clOI$, we have $g(\vect{0}) = 0$ and $g(\vect{1}) = 1$.
Since $g \notin \clSc$, there exists a tuple $\vect{u}$ such that $g(\vect{u}) = g(\overline{\vect{u}}) =: a$.
Since $g \in \clSmin$, we must have $a = 0$.
By identifying arguments, we get $\mathord{\wedge}$ as a minor of $g$.
By Lemma~\ref{lem:helpful} it suffices to show that $\clSminOI$ is $G$\hyp{}bisectable with $G := \{f', \mathord{\wedge}\}$.
Note that $\id \leq f'$.
So, let $\theta \in \clSminOI$.

Condition \ref{helpful:true}:
Let $\vect{a}, \vect{a}' \in \theta^{-1}(1)$.
We have $\vect{a} \neq \overline{\vect{a}'}$ and $\vect{a} \neq \vect{0} \neq \vect{a}'$.
Then one of the following cases must occur, for some $i$, $j$, $k$:
(i) $\big( \begin{smallmatrix} a_i \\ a'_i \end{smallmatrix} \big) = \big( \begin{smallmatrix} 1 \\ 1 \end{smallmatrix} \big)$, in which case $\id_i(\vect{a}) = \id_i(\vect{a}') = 1$;
(ii) $\big( \begin{smallmatrix} a_i & a_j & a_k \\ a'_i & a'_j & a'_k \end{smallmatrix} \big) = \big( \begin{smallmatrix} 0 & 1 & 0 \\ 0 & 0 & 1 \end{smallmatrix} \big)$, in which case $f'_{kji}(\vect{a}) = f'_{kji}(\vect{a}') = 1$.

Condition \ref{helpful:false}:
Let $\vect{b}, \vect{b}' \in \theta^{-1}(0)$.
We have $\vect{b} \neq \vect{1} \neq \vect{b}'$.
Then there exist indices $i$ and $j$ such that $b_i = b'_j = 0$.
Therefore $\mathord{\wedge_{ij}}(\vect{b}) = \mathord{\wedge_{ij}}(\vect{b}') = 0$.

Condition \ref{helpful:both}:
Let $\vect{a} \in \theta^{-1}(1)$ and $\vect{b} \in \theta^{-1}(0)$.
We have $\vect{0} \neq \vect{a} \neq \vect{b} \neq \vect{1}$.
Then one of the following cases must occur, for some $i$, $j$, $k$:
(i) $\big( \begin{smallmatrix} a_i \\ b_i \end{smallmatrix} \big) = \big( \begin{smallmatrix} 1 \\ 0 \end{smallmatrix} \big)$, in which case $\id_i(\vect{a}) = 1$, $\id_i(\vect{b}) = 0$;
(ii) $\big( \begin{smallmatrix} a_i & a_j & a_k \\ b_i & b_j & b_k \end{smallmatrix} \big) = \big( \begin{smallmatrix} 0 & 1 & 0 \\ 1 & 1 & 0 \end{smallmatrix} \big)$, in which case $f'_{kji}(\vect{a}) = 1$, $f'_{kji}(\vect{b}) = 0$.
\end{proof}

\begin{proposition}
\label{prop:Smin01C0}
For any $f, g \in \clSminOICO$ with $f \notin \clSminOI$ and $g \notin \clTcUCO$, we have $\gen{f, g} = \clSminOICO$.
\end{proposition}

\begin{proof}
We have $f = 0$.
Since $g \in \clSminOI \setminus \clTcU$, the argument in the proof of Proposition~\ref{prop:Smin00}\ref{prop:Smin00:Smin01} shows that $g$ has a ternary minor $g'$ satisfying
$g'(1,1,1) = 1$, \linebreak $g'(1,0,0) = 1$, $g'(0,1,0) = 1$,
$g'(0,0,0) = 0$, $g'(0,1,1) = 0$, $g'(1,0,1) = 0$.
By Lemma~\ref{lem:helpful} it suffices to show that $\clSminOICO$ is $G$\hyp{}bisectable with $G := \{0, g'\}$.
Note that $\id \leq g'$.
So, let $\theta \in \clSminOICO$.

Condition \ref{helpful:false} holds because $0 \in G$.
Conditions \ref{helpful:true} and \ref{helpful:both} hold vacuously if $\theta \in \clVako$.
That conditions \ref{helpful:true} and \ref{helpful:both} are satisfied for $\theta \in \clSminOI$ is shown in the same way as in the proof of Proposition~\ref{prop:Smin00}\ref{prop:Smin00:Smin01}.
\end{proof}

\begin{proposition}
\label{prop:Smin0X}
For any $f, g, h \in \clSminOX$ with $f \notin \clSminOO$, $g \notin \clSminOICO$, $h \notin \clU$, we have $\gen{f, g, h} = \clSminOX$.
\end{proposition}

\begin{proof}
We have $f \in \clSminOI$, so $\id \leq f$ by Lemma~\ref{lem:apu:minors}\ref{lem:apu:minors:01}, and we have $g \in \clSminOO \setminus \clVako$, so $\mathord{\nrightarrow} \leq g$ by Lemma~\ref{lem:apu:minors:Smin00-C}.
Since $h \notin \clU$, there exist tuples $\vect{a}$ and $\vect{b}$ such that $h(\vect{a}) = h(\vect{b}) = 1$ and $\vect{a} \wedge \vect{b} = \vect{0}$; moreover $\vect{a}$ and $\vect{b}$ are distinct from $\vect{0}$ because $h \in \clOX$.
Since $h \in \clSmin$, $h(\overline{\vect{a}}) = h(\overline{\vect{b}}) = 0$ and consequently $\vect{a} \neq \overline{\vect{b}}$.
Therefore $h$ has a ternary minor $h'$ satisfying
$h'(0,1,1) = 0$, $h'(1,0,1) = 0$,
$h'(1,0,0) = 1$, $h'(0,1,0) = 1$, $h'(0,0,0) = 0$.
By Lemma~\ref{lem:helpful} it suffices to show that $\clSminOX$ is $G$\hyp{}bisectable with $G := \{\id, \mathord{\nrightarrow}, h'\}$.
Note that $0 \leq \mathord{\nrightarrow}$.
So, let $\theta \in \clSminOX$.

Condition \ref{helpful:true}:
Let $\vect{a}, \vect{a}' \in \theta^{-1}(1)$.
We have $\vect{a} \neq \overline{\vect{a}'}$ and $\vect{a} \neq \vect{0} \neq \vect{a}'$.
Then one of the following cases must occur, for some $i$, $j$, $k$:
(i) $\big( \begin{smallmatrix} a_i \\ a'_i \end{smallmatrix} \big) = \big( \begin{smallmatrix} 1 \\ 1 \end{smallmatrix} \big)$, in which case $\id_i(\vect{a}) = \id_i(\vect{a}') = 1$;
(ii) $\big( \begin{smallmatrix} a_i & a_j & a_k \\ a'_i & a'_j & a'_k \end{smallmatrix} \big) = \big( \begin{smallmatrix} 0 & 1 & 0 \\ 0 & 0 & 1 \end{smallmatrix} \big)$, in which case $h'_{kji}(\vect{a}) = h'_{kji}(\vect{a}') = 1$.

Condition \ref{helpful:false} holds because $0 \leq \mathord{\nrightarrow} \in G$.

Condition \ref{helpful:both}:
Let $\vect{a} \in \theta^{-1}(1)$ and $\vect{b} \in \theta^{-1}(0)$.
We have $\vect{0} \neq \vect{a} \neq \vect{b}$.
Then one of the following cases must occur, for some $i$, $j$:
(i) $\big( \begin{smallmatrix} a_i \\ b_i \end{smallmatrix} \big) = \big( \begin{smallmatrix} 1 \\ 0 \end{smallmatrix} \big)$, in which case $\id_i(\vect{a}) = 1$, $\id_i(\vect{b}) = 0$;
(ii) $\big( \begin{smallmatrix} a_i & a_j \\ b_i & b_j \end{smallmatrix} \big) = \big( \begin{smallmatrix} 0 & 1 \\ 1 & 1 \end{smallmatrix} \big)$, in which case $\mathord{\nrightarrow}_{ji}(\vect{a}) = 1$, $\mathord{\nrightarrow}_{ji}(\vect{b}) = 0$.
\end{proof}

\begin{proposition}
\label{prop:SminNeq}
For any $f, g, h \in \clSminNeq$ with $f \notin \clSminOI$, $g \notin \clSminIO$, $h \notin \clS$, we have $\gen{f, g, h} = \clSminNeq$.
\end{proposition}

\begin{proof}
We have $f \in \clSminIO$ and $g \in \clSminOI$, so $\neg \leq f$ and $\id \leq g$.
Since $h \notin \clS$, there exists a tuple $\vect{a}$ such that $h(\vect{a}) = h(\overline{\vect{a}}) =: a$; since $h \in \clSmin$, we must have $a = 0$; note that $\vect{a} \notin \{\vect{0}, \vect{1}\}$ because $h(\vect{0}) \neq h(\vect{1})$.
Thus $h$ has a binary minor $h'$ that satisfies $h'(0,1) = h'(1,0) = 0$.
By Lemma~\ref{lem:helpful} it suffices to show that $\clSminNeq$ is $G$\hyp{}bisectable with $G := \{\id, \neg, h'\}$.
So, let $\theta \in \clSminNeq$.

Condition \ref{helpful:true}:
Let $\vect{a}, \vect{a}' \in \theta^{-1}(1)$.
We have $\vect{a} \neq \overline{\vect{a}'}$, so there is an $i$ such that $a_i = a'_i$.
Then $\tau(\vect{a}) = \tau(\vect{a}') = 1$ for some $\tau \in \{\id_i, \neg_i\}$.

Condition \ref{helpful:false}:
Let $\vect{b}, \vect{b}' \in \theta^{-1}(0)$.
We have $\{\vect{b}, \vect{b}'\} \neq \{\vect{0}, \vect{1}\}$.
If there is an $i$ such that $b_i = b'_i$, then $\tau(\vect{b}) = \tau(\vect{b}') = 0$ for some $\tau \in \{\id_i, \neg_i\}$.
Otherwise there exist $j$ and $k$ such that $b_j = b'_k = 1$, $b'_j = b_k = 0$, and we have $h'_{jk}(\vect{b}) = h'_{jk}(\vect{b}') = 0$.

Condition \ref{helpful:both}
Let $\vect{a} \in \theta^{-1}(1)$ and $\vect{b} \in \theta^{-1}(0)$.
We have $\vect{a} \neq \vect{b}$, so there exists an $i$ such that $a_i \neq b_i$.
Then $\tau(\vect{a}) = 1$ and $\tau(\vect{b}) = 0$ for some $\tau \in \{\id_i, \neg_i\}$.
\end{proof}

\begin{proposition}
\label{prop:Smin}
For any $f, g, h \in \clSmin$ with $f \notin \clSminNeq$, $g \notin \clSminOX$, $h \notin \clSminXO$, we have $\gen{f, g, h} = \clSmin$.
\end{proposition}

\begin{proof}
We have $f \in \clSminOO$, $g \in \clSminIO$, and $h \in \clSminOI$, so $0 \leq f$, $\neg \leq g$, and $\id \leq h$.
By Lemma~\ref{lem:helpful} it suffices to show that $\clSmin$ is $G$\hyp{}bisectable with $G := \{0, \id, \neg\}$.
So, let $\theta \in \clSmin$.

Condition \ref{helpful:true}:
Let $\vect{a}, \vect{a}' \in \theta^{-1}(1)$.
We have $\vect{a} \neq \overline{\vect{a}'}$, so there is an $i$ such that $a_i = a'_i$.
Then $\tau(\vect{a}) = \tau(\vect{a}') = 1$ for some $\tau \in \{\id_i, \neg_i\}$.

Condition \ref{helpful:false} holds because $0 \in G$.

Condition \ref{helpful:both}:
Let $\vect{a} \in \theta^{-1}(1)$ and $\vect{b} \in \theta^{-1}(0)$.
Then $\vect{a} \neq \vect{b}$, so there is an $i$ such that $a_i \neq b_i$,
and we have $\tau(\vect{a}) = 1$ and $\tau(\vect{b}) = 0$ for some $\tau \in \{\id_i, \neg_i\}$.
\end{proof}

\begin{proposition}
\label{prop:00}
\leavevmode
\begin{enumerate}[label=\upshape{(\roman*)}]
\item\label{prop:00:00}
For any $f, g \in \clOO$ with $f \notin \clSminOO$ and $g \notin \clReflOO$, we have $\gen{f, g} = \clOO$.
\item\label{prop:00:00C}
For any $f, g \in \clOOC$ with $f \notin \clOO$ and $g \notin \clReflOOC$, we have $\gen{f, g} = \clOOC$.
\end{enumerate}
\end{proposition}

\begin{proof}
\ref{prop:00:00}
Since $f \in \clOO$, we have $f(\vect{0}) = f(\vect{1}) = 0$; since $f \notin \clSmin$, there exists a tuple $\vect{a}$ such that $f(\vect{a}) = f(\overline{\vect{a}}) = 1$.
Therefore $\mathord{+} \leq f$.
Since $g \in \clOO$, we have $g(\vect{0}) = g(\vect{1}) = 0$; since $g \notin \clRefl$, there exists a tuple $\vect{b}$ such that $g(\vect{b}) \neq g(\overline{\vect{b}})$.
Therefore $\mathord{\nrightarrow} \leq g$.
By Lemma~\ref{lem:helpful} it suffices to show that $\clOO$ is $G$\hyp{}bisectable with $G := \{\mathord{+}, \mathord{\nrightarrow}\}$.
Note that $0$ is a minor of both $\mathord{+}$ and $\mathord{\nrightarrow}$.
So, let $\theta \in \clOO$.

Condition \ref{helpful:true}:
Let $\vect{a}, \vect{a}' \in \theta^{-1}(1)$.
We have $\{\vect{a}, \vect{a}'\} \cap \{\vect{0}, \vect{1}\} = \emptyset$.
Then one of the following cases must occur, for some $i$, $j$:
(i) $\big( \begin{smallmatrix} a_i & a_j \\ a'_i & a'_j \end{smallmatrix} \big) = \big( \begin{smallmatrix} 1 & 0 \\ 1 & 0 \end{smallmatrix} \big)$, in which case $\mathord{\nrightarrow}_{ij}(\vect{a}) = \mathord{\nrightarrow}_{ij}(\vect{a}') = 1$;
(ii) $\big( \begin{smallmatrix} a_i & a_j \\ a'_i & a'_j \end{smallmatrix} \big) = \big( \begin{smallmatrix} 0 & 1 \\ 1 & 0 \end{smallmatrix} \big)$, in which case $\mathord{+}_{ij}(\vect{a}) = \mathord{+}_{ij}(\vect{a}') = 1$.

Condition \ref{helpful:false} holds because $0 \leq \mathord{+} \in G$.

Condition \ref{helpful:both}: Shown as in the proof of Proposition~\ref{prop:UWneg}.

\ref{prop:00:00C}
We have $f = 1$.
Since $g \in \clOO$, we have $g(\vect{0}) = g(\vect{1}) = 0$;
since $g \notin \clRefl$, there exists a tuple $\vect{a}$ such that $g(\vect{a}) \neq g(\overline{\vect{a}})$.
Therefore $\mathord{\nrightarrow} \leq g$.
Clearly also $0 \leq g$.
We have $\mu(\mathord{\nrightarrow}, \mathord{\nrightarrow}_{21}, 1) = \mathord{+}$, so $\mathord{+} \in \gen{\mathord{\nrightarrow}, 1}$.
Part \ref{prop:00:00} and Proposition~\ref{prop:constants}\ref{prop:constants:D0} yield
$\clOOC
= \gen{\mathord{+}, \mathord{\nrightarrow}} \cup \gen{0, 1}
\subseteq \gen{f, g}
\subseteq \clOOC$.
\end{proof}

\begin{proposition}
\label{prop:01}
\leavevmode
\begin{enumerate}[label=\upshape{(\roman*)}]
\item\label{prop:01:01}
For any $f, g, h \in \clOI$ with $f \notin \clSminOI$, $g \notin \clSmajOI$, and $h \notin \clMc$, we have $\gen{f, g, h} = \clOI$.
\item\label{prop:01:01C0}
For any $f, g, h \in \clOICO$ with $f \notin \clOI$, $g \notin \clSminOICO$, and $h \notin \clMo$, we have $\gen{f, g, h} = \clOICO$.
\item\label{prop:01:01C}
For any $f, g, h \in \clOIC$ with $f \notin \clOICO$, $g \notin \clOICI$, and $h \notin \clM$, we have $\gen{f, g, h} = \clOIC$.
\end{enumerate}
\end{proposition}

\begin{proof}
\ref{prop:01:01}
Since $f \in \clOI$, we have $f(\vect{0}) = 0$ and $f(\vect{1}) = 1$; since $f \notin \clSmin$, there exists a tuple $\vect{a}$ such that $f(\vect{a}) = f(\overline{\vect{a}}) = 1$. Therefore $\mathord{\vee} \leq f$.
Since $g \in \clOI$, we have $g(\vect{0}) = 0$ and $g(\vect{1}) = 1$; since $g \notin \clSmaj$, there exists a tuple $\vect{b}$ such that $g(\vect{b}) = g(\overline{\vect{b}}) = 0$. Therefore $\mathord{\wedge} \leq g$.
Since $h \in \clOI \setminus \clM$, it follows from Lemma~\ref{lem:apu:minors:01-M} that $h$ has a ternary minor $h'$ that satisfies $h'(0,0,0) = 0$, $h'(1,0,0) = 1$, $h'(1,1,0) = 0$, $h'(1,1,1) = 1$.
By Lemma~\ref{lem:helpful} it suffices to show that $\clOI$ is $G$\hyp{}bisectable with $G := \{\mathord{\vee}, \mathord{\wedge}, h'\}$.
Note that $\id$ is a minor of $f$, $g$, and $h$.
So, let $\theta \in \clOI$.

Condition \ref{helpful:true}:
Let $\vect{a}, \vect{a}' \in \theta^{-1}(1)$.
We have $\vect{a} \neq \vect{0} \neq \vect{a}'$, so there exist $i$ and $j$ such that $a_i = a'_j = 1$.
Then $\mathord{\vee}_{ij}(\vect{a}) = \mathord{\vee}_{ij}(\vect{a}') = 1$.

Condition \ref{helpful:false}
Let $\vect{b}, \vect{b}' \in \theta^{-1}(0)$.
We have $\vect{b} \neq \vect{1} \neq \vect{b}'$, so there exist $i$ and $j$ such that $b_i = b'_j = 0$.
Then $\mathord{\wedge}_{ij}(\vect{b}) = \mathord{\wedge}_{ij}(\vect{b}') = 0$.

Condition \ref{helpful:both}
Let $\vect{a} \in \theta^{-1}(1)$ and $\vect{b} \in \theta^{-1}(0)$.
We have $\vect{0} \neq \vect{a} \neq \vect{b} \neq \vect{1}$.
Then one of the following cases must occur, for some $i$, $j$, $k$:
(i) $\big( \begin{smallmatrix} a_i \\ b_i \end{smallmatrix} \big) = \big( \begin{smallmatrix} 1 \\ 0 \end{smallmatrix} \big)$, in which case $\id_i(\vect{a}) = 1$, $\id_i(\vect{b}) = 0$;
(ii) $\big( \begin{smallmatrix} a_i & a_j & a_k \\ b_i & b_j & b_k \end{smallmatrix} \big) = \big( \begin{smallmatrix} 0 & 1 & 0 \\ 1 & 1 & 0 \end{smallmatrix} \big)$, in which case $h'_{jik}(\vect{a}) = 1$, $h'_{jik}(\vect{b}) = 0$.

\ref{prop:01:01C0}
We have $f = 0$, $g \in \clOI \setminus \clSminOI$, $h \in \clOI \setminus \clMc$.
By Proposition~\ref{prop:constants}\ref{prop:constants:D0Ca} and part \ref{prop:01:01}, we have
$\clOICO = \gen{g,h} \cup \gen{0} \subseteq \gen{f,g,h} \subseteq \clOICO$.

\ref{prop:01:01C}
We have $f = 1$ and $g = 0$.
By Lemma~\ref{lem:apu:minors:01-M}, $h$ has a ternary minor $h'$ that satisfies $h'(0,0,0) = 0$, $h'(1,0,0) = 1$, $h'(1,1,0) = 0$, $h'(1,1,1) = 1$.
By Lemma~\ref{lem:helpful} it suffices to show that $\clOIC$ is $G$\hyp{}bisectable with $G := \{0, 1, h'\}$.
Note that $\id \leq h'$.
So, let $\theta \in \clOIC$.

Conditions \ref{helpful:true} and \ref{helpful:false} hold because $0, 1 \in G$.
If $\theta \in \clVak$, then condition \ref{helpful:both} holds vacuously.
In the case when $\theta \in \clOI$, condition \ref{helpful:both} can be shown to hold in the same way as in the proof of statement \ref{prop:01:01}.
\end{proof}

\begin{proposition}
\label{prop:0X}
\leavevmode
\begin{enumerate}[label=\upshape{(\roman*)}]
\item\label{prop:0X:0X}
For any $f, g, h \in \clOX$ with $f \notin \clOO$, $g \notin \clSminOX$, $h \notin \clOICO$, we have $\gen{f, g, h} = \clOX$.
\item\label{prop:0X:0XC}
For any $f, g, h \in \clOXC$ with $f \notin \clOX$, $g \notin \clOOC$, $h \notin \clOIC$, we have $\gen{f, g, h} = \clOXC$.
\end{enumerate}
\end{proposition}

\begin{proof}
\ref{prop:0X:0X}
We have $f \in \clOI$, so $\id \leq f$.
Since $g \in \clOX$, we have $g(\vect{0}) = 0$;
since $g \notin \clSmin$, there exists a tuple $\vect{a}$ such that $g(\vect{a}) = g(\overline{\vect{a}}) = 1$.
Therefore $g$ has a binary minor $g'$ satisfying $g'(0,1) = 1$, $g'(1,0) = 1$, $g'(0,0) = 0$.
We have $h \in \clOO \setminus \clVako$.
Thus $h(\vect{0}) = h(\vect{1}) = 0$ and there exists a tuple $\vect{b}$ with $h(\vect{b}) = 1$.
Therefore $h$ has a binary minor $h'$ satisfying $h'(0,0) = 0$, $h'(1,1) = 0$, $h'(0,1) = 1$.
By Lemma~\ref{lem:helpful} it suffices to show that $\clOX$ is $G$\hyp{}bisectable with $G := \{\id, g', h'\}$.
Note that $0 \leq h'$.
So, let $\theta \in \clOX$.

Condition \ref{helpful:true}:
Let $\vect{a}, \vect{a}' \in \theta^{-1}(1)$.
We have $\vect{a} \neq \vect{0} \neq \vect{a}'$.
Then one of the following cases must occur, for some $i$, $j$:
(i) $\big( \begin{smallmatrix} a_i \\ a'_i \end{smallmatrix} \big) = \big( \begin{smallmatrix} 1 \\ 1 \end{smallmatrix} \big)$, in which case $\id_i(\vect{a}) = \id_i(\vect{a}') = 1$;
(ii) $\big( \begin{smallmatrix} a_i & a_j \\ a'_i & a'_j \end{smallmatrix} \big) = \big( \begin{smallmatrix} 0 & 1 \\ 1 & 0 \end{smallmatrix} \big)$, in which case $g'_{ij}(\vect{a}) = g'_{ij}(\vect{a}') = 1$.

Condition \ref{helpful:false} holds because $0 \leq h' \in G$.

Condition \ref{helpful:both}:
Let $\vect{a} \in \theta^{-1}(1)$ and $\vect{b} \in \theta^{-1}(0)$.
We have $\vect{0} \neq \vect{a} \neq \vect{b}$.
Then one of the following cases must occur, for some $i$, $j$:
(i) $\big( \begin{smallmatrix} a_i \\ b_i \end{smallmatrix} \big) = \big( \begin{smallmatrix} 1 \\ 0 \end{smallmatrix} \big)$, in which case $\id_i(\vect{a}) = 1$, $\id_i(\vect{b}) = 0$;
(ii) $\big( \begin{smallmatrix} a_i & a_j \\ b_i & b_j \end{smallmatrix} \big) = \big( \begin{smallmatrix} 0 & 1 \\ 1 & 1 \end{smallmatrix} \big)$, in which case $h'_{ij}(\vect{a}) = 1$, $h'_{ij}(\vect{b}) = 0$.

\ref{prop:0X:0XC}
We have $f = 1$, and $g \in \clOI$, so $\id \leq g$.
Since $h \in \clOO \setminus \clVak$, $h$ has a binary minor $h'$ satisfying $h'(0,0) = 0$, $h'(1,1) = 0$, $h'(0,1) = 1$ (see the proof of part \ref{prop:0X:0X}).
We have also $0 \leq h$.
Note that $\mu(h', h'_{21}, 1) = \mathord{+}$, so $\mathord{+} \in \gen{f, g, h}$.
By part \ref{prop:0X:0X} and Proposition~\ref{prop:constants}\ref{prop:constants:D0}, we have $\gen{0,1} = \clVak$ and $\gen{\id, \mathord{+}, h} = \clOX$.
Therefore
$\clOXC
= \gen{\id, \mathord{+}, h} \cup \gen{0,1}
\subseteq \gen{f, g, h}
\subseteq \clOXC$.
\end{proof}

\begin{proposition}
\label{prop:NeqEq}
\leavevmode
\begin{enumerate}[label=\upshape{(\roman*)}]
\item\label{prop:NeqEq:Neq}
For any $e, f, g, h \in \clNeq$ with $e \notin \clOI$, $f \notin \clIO$, $g \notin \clSminNeq$, $h \notin \clSmajNeq$, we have $\gen{e, f, g, h} = \clNeq$.
\item\label{prop:NeqEq:Eq}
For any $f, g, h \in \clEq$ with $f \notin \clOOC$, $g \notin \clIIC$, $h \notin \clRefl$, we have $\gen{f, g, h} = \clEq$.
\end{enumerate}
\end{proposition}

\begin{proof}
\ref{prop:NeqEq:Neq}
We have $e \in \clIO$, so $\neg \leq e$.
We have $f \in \clOI$, so $\id \leq f$.
Since $g \in \clNeq$, we have $g(\vect{0}) \neq g(\vect{1})$; since $g \notin \clSmin$, there exists a tuple $\vect{a} \in \{0,1\}^n \setminus \{\vect{0}, \vect{1}\}$ such that $g(\vect{a}) = g(\overline{\vect{a}}) = 1$. Thus $g$ has a binary minor $g'$ satisfying $g'(1,0) = 1$, $g'(0,1) = 1$.
Similarly, we can show that $h$ has a binary minor $h'$ satisfying $h'(1,0) = 0$, $h'(0,1) = 0$.
By Lemma~\ref{lem:helpful} it suffices to show that $\clNeq$ is $G$\hyp{}bisectable with $G := \{\id, \neg, g', h'\}$.
So, let $\theta \in \clNeq$.

Condition \ref{helpful:true}:
Let $\vect{a}, \vect{a}' \in \theta^{-1}(1)$.
We have $\{\vect{a}, \vect{a}'\} \neq \{\vect{0}, \vect{1}\}$.
If there is an $i$ such that $a_i = a'_i$, then $\tau(\vect{a}) = \tau(\vect{a}') = 1$ for some $\tau \in \{\id_i, \neg_i\}$.
Otherwise there exist $j$ and $k$ such that $a_j = a'_k = 1$ and $a'_j = a_k = 0$.
Then $g'_{jk}(\vect{a}) = g'_{jk}(\vect{a}') = 1$.

Condition \ref{helpful:false}:
Let $\vect{b}, \vect{b}' \in \theta^{-1}(0)$.
We have $\{\vect{b}, \vect{b}'\} \neq \{\vect{0}, \vect{1}\}$.
If there is an $i$ such that $b_i = b'_i$, then $\tau(\vect{b}) = \tau(\vect{b}') = 0$ for some $\tau \in \{\id_i, \neg_i\}$.
Otherwise there exist $j$ and $k$ such that $a_j = a'_k = 1$ and $a'_j = a_k = 0$.
Then $h'_{jk}(\vect{a}) = h'_{jk}(\vect{a}') = 0$.

Condition \ref{helpful:both}:
Let $\vect{a} \in \theta^{-1}(1)$ and $\vect{b} \in \theta^{-1}(0)$.
Then $\vect{a} \neq \vect{b}$, so there is an $i$ such that $a_i \neq b_i$, and we have $\tau(\vect{a}) = 1$ and $\tau(\vect{b}) = 0$ for some $\tau \in \{\id_i, \neg_i\}$.

\ref{prop:NeqEq:Eq}
We have $f \in \clII \setminus \clVak$, so $f$ has a binary minor $f'$ satisfying $f'(1,1) = 1$, $f'(0,0) = 1$, $f'(1,0) = 0$.
Similarly, $g \in \clOO \setminus \clVak$, so $g$ has a binary minor $g'$ satisfying $g'(0,0) = 0$, $g'(1,1) = 0$, $g'(1,0) = 1$.
Since $h \in \clEq \setminus \clRefl$, there exists a tuple $\vect{a} \notin \{\vect{0}, \vect{1}\}$ such that $h(\vect{a}) \neq h(\overline{\vect{a}})$.
Therefore $h$ has a binary minor $h'$ satisfying $h'(1,0) = 1$, $h'(0,1) = 0$.
By Lemma~\ref{lem:helpful} it suffices to show that $\clEq$ is $G$\hyp{}bisectable with $G := \{f', g', h'\}$.
Note that $1 \leq f'$ and $0 \leq g'$.
So, let $\theta \in \clEq$.

Conditions \ref{helpful:true} and \ref{helpful:false} hold because $0 \leq g' \in G$ and $1 \leq f' \in G$.

Condition \ref{helpful:both}:
Let $\vect{a} \in \theta^{-1}(1)$ and $\vect{b} \in \theta^{-1}(0)$.
We have $\vect{a} \neq \vect{b}$ and $\{\vect{a}, \vect{b}\} \neq \{\vect{0}, \vect{1}\}$.
Then one of the following cases must occur, for some $i$, $j$ and $x, y \in \{0,1\}$:
(i) $\big( \begin{smallmatrix} a_i & a_j \\ b_i & b_j \end{smallmatrix} \big) = \big( \begin{smallmatrix} x & \overline{x} \\ \overline{x} & x \end{smallmatrix} \big)$, in which case $\tau(\vect{a}) = 1$ and $\tau(\vect{b}) = 0$ for some $\tau \in \{h'_{ij}, h'_{ji}\}$;
(ii) $\big( \begin{smallmatrix} a_i & a_j \\ b_i & b_j \end{smallmatrix} \big) = \big( \begin{smallmatrix} x & y \\ \overline{x} & y \end{smallmatrix} \big)$, in which case $\tau(\vect{a}) = 1$ and $\tau(\vect{b}) = 0$ for some $\tau \in \{f'_{ij}, f'_{ji}, g'_{ij}, g'_{ji}\}$.
\end{proof}

\begin{proposition}
\label{prop:three}
\leavevmode
\begin{enumerate}[label=\upshape{(\roman*)}]
\item\label{prop:three:not10}
For any $f, g, h \in \clEiio$ with $f \notin \clOXC$, $g \notin \clXIC$, $h \notin \clEq$, we have $\gen{f, g, h} = \clEiio$.
\item\label{prop:three:not11}
For any $e, f, g, h \in \clEiii$ with $e \notin \clOX$, $f \notin \clXO$, $g \notin \clNeq$, $h \notin \clSmin$, we have $\gen{e, f, g, h} = \clEiii$.
\end{enumerate}
\end{proposition}

\begin{proof}
\ref{prop:three:not10}
We have $f \in \clII \setminus \clVak$; hence $f$ has a binary minor $f'$ satisfying $f'(1,1) = 1$, $f'(0,0) = 1$, $f'(1,0) = 0$.
We have $g \in \clOO \setminus \clVak$; hence $g$ has a binary minor $g'$ satisfying $g'(1,1) = 0$, $g'(0,0) = 0$, $g'(1,0) = 1$.
We have $h \in \clOI$; hence $\id \leq h$.
By Lemma~\ref{lem:helpful} it suffices to show that $\clEiio$ is $G$\hyp{}bisectable with $G := \{\id, f', g'\}$.
Note that $1 \leq f'$ and $0 \leq g'$.
So, let $\theta \in \clEiio$.

Conditions \ref{helpful:true} and \ref{helpful:false} hold because $1 \leq f' \in G$ and $0 \leq g' \in G$.

Condition \ref{helpful:both}:
Let $\vect{a} \in \theta^{-1}(1)$ and $\vect{b} \in \theta^{-1}(0)$.
We have $\vect{a} \neq \vect{b}$ and $(\vect{a}, \vect{b}) \neq (\vect{0}, \vect{1})$.
Then one of the following cases must occur, for some $i$, $j$:
(i) $\big( \begin{smallmatrix} a_i \\ b_i \end{smallmatrix} \big) = \big( \begin{smallmatrix} 1 \\ 0 \end{smallmatrix} \big)$, in which case $\id_i(\vect{a}) = 1$, $\id_i(\vect{b}) = 0$;
(ii) $\big( \begin{smallmatrix} a_i & a_j \\ b_i & b_j \end{smallmatrix} \big) = \big( \begin{smallmatrix} 0 & 0 \\ 1 & 0 \end{smallmatrix} \big)$, in which case $f'_{ij}(\vect{a}) = 1$, $f'_{ij}(\vect{b}) = 0$;
(iii) $\big( \begin{smallmatrix} a_i & a_j \\ b_i & b_j \end{smallmatrix} \big) = \big( \begin{smallmatrix} 0 & 1 \\ 1 & 1 \end{smallmatrix} \big)$, in which case $g'_{ji}(\vect{a}) = 1$, $g'_{ji}(\vect{b}) = 0$.

\ref{prop:three:not11}
We have $e \in \clIO$, $f \in \clOI$, $g \in \clOO$, so $\neg \leq e$, $\id \leq f$, $0 \leq g$.
Since $h \notin \clSmin$, there exists a tuple $\vect{a}$ such that $h(\vect{a}) = h(\overline{\vect{a}}) = 1$; since $h \notin \clII$, we must have $\vect{a} \notin \{\vect{0}, \vect{1}\}$.
Therefore $h$ has a binary minor $h'$ satisfying $h'(1,0) = 1$, $h'(0,1) = 1$.
By Lemma~\ref{lem:helpful} it suffices to show that $\clEiii$ is $G$\hyp{}bisectable with $G := \{0, \id, \neg, h'\}$.
So, let $\theta \in \clEiii$.

Condition \ref{helpful:true}:
Let $\vect{a}, \vect{a}' \in \theta^{-1}(1)$.
We have $\{\vect{a}, \vect{a}'\} \neq \{\vect{0}, \vect{1}\}$.
If there is an $i$ such that $a_i = a'_i$, then $\tau(\vect{a}) = \tau(\vect{a}') = 1$ for some $\tau \in \{\id_i, \neg_i\}$.
Otherwise there exist indices $j$ and $k$ such that $a_j = a'_k = 1$ and $a'_j = a_k = 0$.
Then $h'_{jk}(\vect{a}) = h'_{jk}(\vect{a}') = 1$.

Condition \ref{helpful:false} holds because $0 \in G$.

Condition \ref{helpful:both}:
Let $\vect{a} \in \theta^{-1}(1)$ and $\vect{b} \in \theta^{-1}(0)$.
Then $\vect{a} \neq \vect{b}$, so there is an $i$ such that $a_i \neq b_i$, and we have $\tau(\vect{a}) = 1$ and $\tau(\vect{b}) = 0$ for some $\tau \in \{\id_i, \neg_i\}$.
\end{proof}

\begin{proposition}
\label{prop:all}
For any $e, f, g, h \in \clAll$ with $e \notin \clEioo$, $f \notin \clEiii$, $g \notin \clEioi$, $h \notin \clEiio$, we have $\gen{e, f, g, h} = \clAll$.
\end{proposition}

\begin{proof}
We have $e \in \clOO$, $f \in \clII$, $g \in \clOI$, $h \in \clIO$, so $0 \leq e$, $1 \leq f$, $\id \leq g$, $\neg \leq h$.
By Lemma~\ref{lem:helpful} it suffices to show that $\clAll$ is $G$\hyp{}bisectable with $G := \{0, 1, \id, \neg\}$.
So, let $\theta \in \clAll$.

Conditions \ref{helpful:true} and \ref{helpful:false} hold because $0, 1 \in G$.
Condition \ref{helpful:both} holds because if $\vect{a} \in \theta^{-1}(1)$ and $\vect{b} \in \theta^{-1}(0)$, then $\vect{a} \neq \vect{b}$, so there is an $i$ with $a_i \neq b_i$, and we have $\tau(\vect{a}) = 1$ and $\tau(\vect{b}) = 0$ for some $\tau \in \{\id_i, \neg_i\}$.
\end{proof}

\begin{proof}[Proof of Theorem~\ref{thm:SM-stable}]
By Proposition~\ref{prop:SM-stable:sufficiency}, the listed classes are $(\clIc,\clSM)$\hyp{}stable.
Propositions \ref{prop:empty}, \ref{prop:constants}, \ref{prop:SM}, \ref{prop:Sc}, \ref{prop:S}, \ref{prop:R}, \ref{prop:UWneg}, \ref{prop:UC}, \ref{prop:McU2}, \ref{prop:MU2}, \ref{prop:TcU2}, \ref{prop:TcU2C}, \ref{prop:U2}, \ref{prop:M}, \ref{prop:Smin00}, \ref{prop:Smin01C0}, \ref{prop:Smin0X}, \ref{prop:SminNeq}, \ref{prop:Smin}, \ref{prop:00}, \ref{prop:01}, \ref{prop:0X}, \ref{prop:NeqEq}, \ref{prop:three}, \ref{prop:all},
together with Lemma~\ref{lem:IcSM-stable-negation},
show that if $K$ is one of these classes and $F$ is a subset of $K$ that is not included in any of the lower covers of $K$, then $\gen{F} = K$.
It follows that any set $F \subseteq \clAll$ generates one of the listed classes, and there are hence no further $(\clIc,\clSM)$\hyp{}stable classes.
\end{proof}


\section{\texorpdfstring{$(C_1,C_2)$\hyp{}stable classes with $\clSM \subseteq C_2$}{(C1,C2)-stable classes with SM subset of C2}}
\label{sec:C1C2}

Theorem~\ref{thm:SM-stable} allows us to describe all $(C_1,C_2)$\hyp{}stable classes of Boolean functions for clones $C_1$ and $C_2$ such that $C_1$ is arbitrary and $\clSM \subseteq C_2$.
Since $(C_1,C_2)$\hyp{}stability implies $(\clIc,\clSM)$\hyp{}stability whenever $\clSM \subseteq C_2$, by Lemma~\ref{lem:stable-impl-stable}, it suffices to search for $(C_1,C_2)$\hyp{}stable classes among the $(\clIc,\clSM)$\hyp{}stable ones.
To this end, we determine, for each $(\clIc,\clSM)$\hyp{}stable class $K$, all clones $C_1$ and $C_2$ for which it holds that $K C_1 \subseteq K$ and $C_2 K \subseteq K$.
The results are summarized in the following theorem and Table~\ref{table:SM-stable}.

\begin{theorem}
\label{thm:C1C2-stability}
For each $(\clIc,\clSM)$\hyp{}stable class $K$, as determined in Theorem~\ref{thm:SM-stable}, there exist clones $C_1^K$ and $C_2^K$, as prescribed in Table~\ref{table:SM-stable}, such that for every clone $C$, it holds that $K C \subseteq K$ if and only if $C \subseteq C_1^K$, and $C K \subseteq K$ if and only if $C \subseteq C_2^K$.
\end{theorem}


\begin{table}
\small
\begin{tabular}{ccc@{\qquad\qquad\!\!\!\!}cccc}
\toprule
$K$ & $K C \subseteq K$ & $C K \subseteq K$ & $K$ & $K C \subseteq K$ & $C K \subseteq K$ & Prop. \\
& iff $C \subseteq \ldots$ & iff $C \subseteq \ldots$ & & iff $C \subseteq \ldots$ & iff $C \subseteq \ldots$ \\
\midrule
$\clAll$ & $\clAll$ & $\clAll$ & & & & \ref{prop:LR:clones} \\
$\clEiio$ & $\clOI$ & $\clM$ & $\clEioi$ & $\clOI$ & $\clM$ & \ref{prop:LR:Ei} \\
$\clEiii$ & $\clOI$ & $\clU$ & $\clEioo$ & $\clOI$ & $\clW$ & \ref{prop:LR:Ei} \\
$\clEq$ & $\clOI$ & $\clAll$ & & & & \ref{prop:LR:EqNeq} \\
$\clNeq$ & $\clOI$ & $\clS$ & & & & \ref{prop:LR:EqNeq} \\
$\clOXC$ & $\clOX$ & $\clM$ & $\clIXC$ & $\clOX$ & $\clM$ & \ref{prop:LR:OXC} \\
$\clXOC$ & $\clXI$ & $\clM$ & $\clXIC$ & $\clXI$ & $\clM$ & \ref{prop:LR:OXC} \\
$\clOX$ & $\clOX$ & $\clOX$ & $\clIX$ & $\clOX$ & $\clXI$ & \ref{prop:LR:clones} \\
$\clXO$ & $\clXI$ & $\clOX$ & $\clXI$ & $\clXI$ & $\clXI$ & \ref{prop:LR:clones} \\
$\clOOC$ & $\clOI$ & $\clM$ & $\clIIC$ & $\clOI$ & $\clM$ & \ref{prop:LR:OOC} \\
$\clOO$ & $\clOI$ & $\clOX$ & $\clII$ & $\clOI$ & $\clXI$ & \ref{prop:LR:OO} \\
$\clOIC$ & $\clOI$ & $\clM$ & $\clIOC$ & $\clOI$ & $\clM$ & \ref{prop:LR:OIC} \\
$\clOICO$ & $\clOI$ & $\clMo$ & $\clIOCI$ & $\clOI$ & $\clMi$ & \ref{prop:LR:OICO} \\
$\clOICI$ & $\clOI$ & $\clMi$ & $\clIOCO$ & $\clOI$ & $\clMo$ & \ref{prop:LR:OICO} \\
$\clOI$ & $\clOI$ & $\clOI$ & $\clIO$ & $\clOI$ & $\clOI$ & \ref{prop:LR:clones} \\
$\clSmin$ & $\clS$ & $\clU$ & $\clSmaj$ & $\clS$ & $\clW$ & \ref{prop:LR:Smin} \\
$\clSminNeq$ & $\clSc$ & $\clSM$ & $\clSmajNeq$ & $\clSc$ & $\clSM$ & \ref{prop:LR:SminNeq} \\
$\clSminOX$ & $\clSc$ & $\clU$ & $\clSmajIX$ & $\clSc$ & $\clW$ & \ref{prop:LR:SminOX} \\
$\clSminXO$ & $\clSc$ & $\clU$ & $\clSmajXI$ & $\clSc$ & $\clW$ & \ref{prop:LR:SminOX} \\
$\clSminOICO$ & $\clSc$ & $\clMU$ & $\clSmajIOCI$ & $\clSc$ & $\clMW$ & \ref{prop:LR:SminOICO} \\
$\clSminIOCO$ & $\clSc$ & $\clMU$ & $\clSmajOICI$ & $\clSc$ & $\clMW$ & \ref{prop:LR:SminOICO} \\
$\clSminOI$ & $\clSc$ & $\clTcU$ & $\clSmajIO$ & $\clSc$ & $\clTcW$ & \ref{prop:LR:SminOI} \\
$\clSminIO$ & $\clSc$ & $\clTcU$ & $\clSmajOI$ & $\clSc$ & $\clTcW$ & \ref{prop:LR:SminOI} \\
$\clSminOO$ & $\clSc$ & $\clU$ & $\clSmajII$ & $\clSc$ & $\clW$ & \ref{prop:LR:SminOO} \\
$\clS$ & $\clS$ & $\clS$ & & & & \ref{prop:LR:clones} \\ 
$\clSc$ & $\clSc$ & $\clSc$ & $\clScneg$ & $\clSc$ & $\clSc$ & \ref{prop:LR:clones} \\
$\clM$ & $\clM$ & $\clM$ & $\clMneg$ & $\clM$ & $\clM$ & \ref{prop:LR:clones} \\
$\clMo$ & $\clMo$ & $\clMo$ & $\clMoneg$ & $\clMo$ & $\clMi$ & \ref{prop:LR:clones} \\
$\clMi$ & $\clMi$ & $\clMi$ & $\clMineg$ & $\clMi$ & $\clMo$ & \ref{prop:LR:clones} \\
$\clMc$ & $\clMc$ & $\clMc$ & $\clMcneg$ & $\clMc$ & $\clMc$ & \ref{prop:LR:clones} \\
$\clSM$ & $\clSM$ & $\clSM$ & $\clSMneg$ & $\clSM$ & $\clSM$ & \ref{prop:LR:clones} \\
$\clU$ & $\clU$ & $\clU$ & $\clUneg$ & $\clU$ & $\clW$ & \ref{prop:LR:clones} \\
$\clTcUCO$ & $\clTcU$ & $\clMU$ & $\clTcUnegCI$ & $\clTcU$ & $\clMW$ & \ref{prop:LR:TcUCO} \\
$\clTcU$ & $\clTcU$ & $\clTcU$ & $\clTcUneg$ & $\clTcU$ & $\clTcW$ & \ref{prop:LR:clones} \\
$\clMU$ & $\clMU$ & $\clMU$ & $\clMUneg$ & $\clMU$ & $\clMW$ & \ref{prop:LR:clones} \\
$\clMcU$ & $\clMcU$ & $\clMcU$ & $\clMcUneg$ & $\clMcU$ & $\clMcW$ & \ref{prop:LR:clones} \\
$\clW$ & $\clW$ & $\clW$ & $\clWneg$ & $\clW$ & $\clU$ & \ref{prop:LR:clones} \\
$\clTcWCI$ & $\clTcW$ & $\clMW$ & $\clTcWnegCO$ & $\clTcW$ & $\clMU$ & \ref{prop:LR:TcUCO} \\
$\clTcW$ & $\clTcW$ & $\clTcW$ & $\clTcWneg$ & $\clTcW$ & $\clTcU$ & \ref{prop:LR:clones} \\
$\clMW$ & $\clMW$ & $\clMW$ & $\clMWneg$ & $\clMW$ & $\clMU$ & \ref{prop:LR:clones} \\
$\clMcW$ & $\clMcW$ & $\clMcW$ & $\clMcWneg$ & $\clMcW$ & $\clMcU$ & \ref{prop:LR:clones} \\
$\clUOO$ & $\clTcU$ & $\clU$ & $\clUnegII$ & $\clTcU$ & $\clW$ & \ref{prop:LR:UOO} \\
$\clWII$ & $\clTcW$ & $\clW$ & $\clWnegOO$ & $\clTcW$ & $\clU$ & \ref{prop:LR:UOO} \\
$\clUWneg$ & $\clSM$ & $\clU$ & $\clWUneg$ & $\clSM$ & $\clW$ & \ref{prop:LR:UWneg} \\
$\clRefl$ & $\clS$ & $\clAll$ & & & & \ref{prop:LR:R} \\
$\clReflOOC$ & $\clSc$ & $\clM$ & $\clReflIIC$ & $\clSc$ & $\clM$ & \ref{prop:LR:ROOC} \\
$\clReflOO$ & $\clSc$ & $\clOX$ & $\clReflII$ & $\clSc$ & $\clXI$ & \ref{prop:LR:ROO} \\
$\clVak$ & $\clAll$ & $\clAll$ & & & & \ref{prop:LR:const} \\
$\clVako$ & $\clAll$ & $\clOX$ & $\clVaki$ & $\clAll$ & $\clXI$ & \ref{prop:LR:const} \\
$\clEmpty$ & $\clAll$ & $\clAll$ & & & & \ref{prop:LR:empty} \\
\bottomrule
\end{tabular}
\smallskip
\caption{$(\clIc,\clSM)$\hyp{}stable classes and their stability under right and left composition with clones of Boolean functions.}
\label{table:SM-stable}
\end{table}


The remainder of this section is devoted to the proof of Theorem~\ref{thm:C1C2-stability}.
With the help of the following lemma, we can reduce the number of cases to consider.
Once we have established necessary and sufficient stability conditions for a class $K$, we get ones for the negation, the inner negation, and the dual of $K$ for free.

\begin{lemma}
\label{lem:neg-dual}
Let $K \subseteq \clAll$ and let $C$ be a clone.
\begin{enumerate}[label=\upshape{(\roman*)}]
\item\label{lem:neg-dual:KC}
The following conditions are equivalent:
\begin{enumerate}[label=\upshape{(\alph*)}]
\item\label{lem:neg-dual:KC:K} $K C \subseteq K$,
\item\label{lem:neg-dual:KC:K-} $\overline{K} C \subseteq \overline{K}$,
\item\label{lem:neg-dual:KC:Kn} $K^\mathrm{n} C^\mathrm{d} \subseteq K^\mathrm{n}$,
\item\label{lem:neg-dual:KC:Kd} $K^\mathrm{d} C^\mathrm{d} \subseteq K^\mathrm{d}$.
\end{enumerate}
\item\label{lem:neg-dual:CK}
The following conditions are equivalent:
\begin{enumerate}[label=\upshape{(\alph*)}]
\item\label{lem:neg-dual:CK:K} $C K \subseteq K$,
\item\label{lem:neg-dual:CK:Kn} $C K^\mathrm{n} \subseteq K^\mathrm{n}$,
\item\label{lem:neg-dual:CK:K-} $C^\mathrm{d} \overline{K} \subseteq \overline{K}$,
\item\label{lem:neg-dual:CK:Kd} $C^\mathrm{d} K^\mathrm{d} \subseteq K^\mathrm{d}$.
\end{enumerate}
\end{enumerate}
\end{lemma}

\begin{proof}
\ref{lem:neg-dual:KC}
We show first that condition \ref{lem:neg-dual:KC:K} implies \ref{lem:neg-dual:KC:K-} and \ref{lem:neg-dual:KC:Kn}.
Assume that $K C \subseteq K$.
In order to show that $\overline{K} C \subseteq \overline{K}$, let $\overline{f} \in \overline{K}$, for some $f \in K$, and $g_1, \dots, g_n \in C$.
Then $\overline{f}(g_1, \dots, g_n) = \overline{f(g_1, \dots, g_n)} \in \overline{K}$ because $f(g_1, \dots, g_n) \in K$.
In order to show that $K^\mathrm{n} C^\mathrm{d} \subseteq K^\mathrm{n}$, let $f^\mathrm{n} \in K^\mathrm{n}$, for some $f \in K$, and $g_1^\mathrm{d}, \dots, g_n^\mathrm{d} \in C^\mathrm{d}$, for some $g_1, \dots, g_n \in C$.
Then $f^\mathrm{n}(g_1^\mathrm{d}, \dots, g_n^\mathrm{d}) = (f(g_1, \dots, g_n))^\mathrm{n} \in K^\mathrm{n}$ because $f(g_1, \dots, g_n) \in K$.

The equivalence of conditions \ref{lem:neg-dual:KC:K}, \ref{lem:neg-dual:KC:K-}, \ref{lem:neg-dual:KC:Kn}, and \ref{lem:neg-dual:KC:Kd} would follow if we show the implications 
$\text{\ref{lem:neg-dual:KC:K}} \Rightarrow \text{\ref{lem:neg-dual:KC:K-}} \Rightarrow \text{\ref{lem:neg-dual:KC:Kd}} \Rightarrow \text{\ref{lem:neg-dual:KC:Kn}} \Rightarrow \text{\ref{lem:neg-dual:KC:K}}$.
We have shown above that the implication $\text{\ref{lem:neg-dual:KC:K}} \Rightarrow \text{\ref{lem:neg-dual:KC:K-}}$ is valid.
The implication $\text{\ref{lem:neg-dual:KC:K-}} \Rightarrow \text{\ref{lem:neg-dual:KC:Kd}}$ follows by considering $\overline{K}$ in place of $K$,
$\text{\ref{lem:neg-dual:KC:Kd}} \Rightarrow \text{\ref{lem:neg-dual:KC:Kn}}$ follows by considering $K^\mathrm{d}$ in place of $K$ and $C^\mathrm{d}$ in place of $C$,
and $\text{\ref{lem:neg-dual:KC:Kn}} \Rightarrow \text{\ref{lem:neg-dual:KC:K}}$ follows by considering $K^\mathrm{n}$ in place of $K$ and $C^\mathrm{d}$ in place of $C$.

\ref{lem:neg-dual:CK}
We show first that condition \ref{lem:neg-dual:CK:K} implies \ref{lem:neg-dual:CK:Kn} and \ref{lem:neg-dual:CK:K-}.
Assume that \linebreak $C K \subseteq K$.
In order to show that $C K^\mathrm{n} \subseteq K^\mathrm{n}$, let $f \in C$ and $g_1^\mathrm{n}, \dots, g_n^\mathrm{n} \in K^\mathrm{n}$, for some $g_1, \dots, g_n \in K$.
Then $f(g_1^\mathrm{n}, \dots, g_n^\mathrm{n}) = (f(g_1, \dots, g_n))^\mathrm{n} \in K^\mathrm{n}$ because \linebreak $f(g_1, \dots, g_n) \in K$.
In order to show that $C^\mathrm{d} \overline{K} \subseteq \overline{K}$, let $f^\mathrm{d} \in C^\mathrm{d}$, for some $f \in C$, and $\overline{g_1}, \dots, \overline{g_n} \in \overline{K}$, for some $g_1, \dots, g_n \in K$.
Then $f^\mathrm{d}(\overline{g_1}, \dots, \overline{g_n}) = \overline{f}(g_1, \dots, g_n) = \overline{f(g_1, \dots, g_n)} \in \overline{K}$ because $f(g_1, \dots, g_n) \in K$.

The equivalence of conditions \ref{lem:neg-dual:CK:K}, \ref{lem:neg-dual:CK:Kn}, \ref{lem:neg-dual:CK:K-}, and \ref{lem:neg-dual:CK:Kd} follows by considering $\overline{K}$, $K^\mathrm{n}$, or $K^\mathrm{d}$ in place of $K$ and $C$ or $C^\mathrm{d}$ in place of $C$.
\end{proof}

The following lemma is our main tool for proving the sufficiency of the conditions.
It provides a sufficient condition for the intersection of two classes for which a sufficient condition is known.
Applying this lemma, we can proceed from the top downwards in the lattice of $(\clIc,\clSM)$\hyp{}stable classes.

\begin{lemma}
\label{lem:intersection}
Let $K_1, K_2 \subseteq \clAll$, and let $C_1$ and $C_2$ be clones.
\begin{enumerate}[label=\upshape{(\roman*)}]
\item\label{lem:intersection:R} If $K_1 C_1 \subseteq K_1$ and $K_2 C_2 \subseteq K_2$, then $(K_1 \cap K_2) (C_1 \cap C_2) \subseteq K_1 \cap K_2$.
\item\label{lem:intersection:L} If $C_1 K_1 \subseteq K_1$ and $C_2 K_2 \subseteq K_2$, then $(C_1 \cap C_2) (K_1 \cap K_2) \subseteq K_1 \cap K_2$.
\end{enumerate}
\end{lemma}

\begin{proof}
\ref{lem:intersection:R}
Assume that $K_1 C_1 \subseteq K_1$ and $K_2 C_2 \subseteq K_2$.
Using the monotonicity of function class composition, we get
$(K_1 \cap K_2) (C_1 \cap C_2) \subseteq K_1 C_1 \subseteq K_1$ and
$(K_1 \cap K_2) (C_1 \cap C_2) \subseteq K_2 C_2 \subseteq K_2$;
hence $(K_1 \cap K_2) (C_1 \cap C_2) \subseteq K_1 \cap K_2$.

\ref{lem:intersection:L}
Assume that $C_1 K_1 \subseteq K_1$ and $C_2 K_2 \subseteq K_2$.
Using the monotonicity of function class composition, we get
$(C_1 \cap C_2) (K_1 \cap K_2) \subseteq C_1 K_1 \subseteq K_1$ and
$(C_1 \cap C_2) (K_1 \cap K_2) \subseteq C_2 K_2 \subseteq K_2$;
hence $(C_1 \cap C_2) (K_1 \cap K_2) \subseteq K_1 \cap K_2$.
\end{proof}

The following lemma will be applied frequently without explicit mention.

\begin{lemma}
\label{lem:Cns}
Let $C$ be a clone.
\begin{enumerate}[label=\upshape{(\roman*)}]
\item\label{lem:Cns:OX}
If $C \nsubseteq \clOX$, then $C$ contains $1$ or $\neg$.
\item\label{lem:Cns:OI}
If $C \nsubseteq \clOI$, then $C$ contains $0$, $1$, or $\neg$.
\item\label{lem:Cns:M}
If $C \nsubseteq \clM$, then $C$ contains a function $f$ satisfying $f(0,0,1) = 1$ and $f(0,1,1) = 0$.
\item\label{lem:Cns:S}
If $C \nsubseteq \clS$, then $C$ contains a function $f$ satisfying $f(0,1) = f(1,0)$.
\item\label{lem:Cns:U}
If $C \nsubseteq \clU$, then $C$ contains a function $f$ satisfying $f(0,0,1) = 1$ and $f(0,1,0) = 1$.
\end{enumerate}
\end{lemma}

\begin{proof}
\ref{lem:Cns:OX}
There exists a function $g \in C$ such that $g(\vect{0}) = 1$.
By identifying all arguments, we obtain a unary minor $f$ of $g$ satisfying $f(0) = 1$.
There are two unary functions satisfying this condition, namely $1$ and $\neg$.
Since $C$ is minor\hyp{}closed, the resulting minor is also a member of $C$.

\ref{lem:Cns:OI}
There exists a function $g \in C$ such that $g(\vect{0}) = 1$ or $g(\vect{1}) = 0$.
By identifying all arguments, we obtain a unary minor $f$ of $g$ satisfying $f(0) = 1$ or $f(1) = 0$.
There are three unary functions satisfying this condition: $0$, $1$ and $\neg$.
Since $C$ is minor\hyp{}closed, the resulting minor is also a member of $C$.

\ref{lem:Cns:M}
There exist a function $g \in C$ and tuples $\vect{a}$ and $\vect{b}$ such that $\vect{a} \leq \vect{b}$ and $g(\vect{a}) > g(\vect{b})$; hence $g(\vect{a}) = 1$ and $g(\vect{b}) = 0$.
Without loss of generality, we may assume that $\vect{a} = 0^{i+j} 1^k$ and $\vect{b} = 0^i 1^{j+k}$ for some $i, j, k \in \IN_{+}$.
(We must have $j > 0$ because $\vect{a} < \vect{b}$. We may also assume that $i > 0$ and $k > 0$ by introducing fictitious arguments if necessary.)
By identifying the first $i$ arguments, the next $j$ arguments, and the last $k$ arguments, we obtain a ternary minor $f$ of $g$ satisfying $f(0,0,1) = 1$ and $f(0,1,1) = 0$.
Since $C$ is minor\hyp{}closed, we have $f \in C$.

\ref{lem:Cns:S}
There exist a function $g \in C$ and a tuple $\vect{a}$ such that $g(\vect{a}) = g(\overline{\vect{a}})$.
Without loss of generality, we may assume that $\vect{a} = 0^i 1^j$ for some $i, j \in \IN_{+}$.
(We clearly must have $i > 0$ or $j > 0$. We may assume that both $i$ and $j$ are nonzero by introducing a fictitious argument if necessary.)
By identifying the first $i$ arguments and the last $j$ arguments, we obtain a binary minor $f$ of $g$ satisfying $g(0,1) = g(1,0)$.
Since $C$ is minor\hyp{}closed, we have $f \in C$.

\ref{lem:Cns:U}
There exist a function $g \in C$ and tuples $\vect{a}$ and $\vect{b}$ such that $g(\vect{a}) = g(\vect{b}) = 1$ and $\vect{a} \wedge \vect{b} = \vect{0}$.
Without loss of generality, we may assume that $\vect{a} = 0^i 0^j 1^k$ and $\vect{b} = 0^i 1^j 0^k$ for some $i, j, k \in \IN_{+}$.
By identifying the first $i$ arguments, the next $j$ arguments, and the last $k$ arguments, we obtain a ternary minor $f$ of $g$ satisfying $f(0,0,1) = f(0,1,0) = 1$.
Since $C$ is minor\hyp{}closed, we have $f \in C$.
\end{proof}

In order to prove the sufficiency of the conditions, we will make use of the noninclusions established in the following lemma.

\begin{lemma}
\label{lem:KCCK}
Let $K$ be a $(\clIc,\clSM)$\hyp{}stable class and $C$ a clone.
\begin{enumerate}[label=\upshape{(\arabic*)}]
\item
\begin{enumerate}[label=\upshape{(\alph*)}]
\item\label{KC:OI-Eq-notin-Eiio}
If $C \nsubseteq \clOI$ and $\clEq \subseteq K$, then $K C \nsubseteq \clEiio$.
\item\label{KC:OI-Neq-notin-Eiii}
If $C \nsubseteq \clOI$ and $\clNeq \subseteq K$, then $K C \nsubseteq \clEiii$.
\item\label{KC:OX-ROO-notin-OXC}
If $C \nsubseteq \clOX$ and $\clReflOO \subseteq K$, then $K C \nsubseteq \clOXC$.
\item\label{KC:OI-ROO-notin-OOC}
If $C \nsubseteq \clOI$ and $\clReflOO \subseteq K$, then $K C \nsubseteq \clOOC$.
\item\label{KC:OI-OI-notin-OIC}
If $C \nsubseteq \clOI$ and $\clOI \subseteq K$, then $K C \nsubseteq \clOIC$.
\item\label{KC:S-Sc-notin-Smin}
If $C \nsubseteq \clS$ and $\clSc \subseteq K$, then $K C \nsubseteq \clSmin$.
\item\label{KC:OI-McU-notin-Neq}
If $C \nsubseteq \clOI$ and $\clMcU \subseteq K$, then $K C \nsubseteq \clNeq$.
\item\label{KC:OI+S-SM-notin-Eiio}
If $C \nsubseteq \clOI$, $C \subseteq \clS$, and $\clSM \subseteq K$, then $K C \nsubseteq \clEiio$.
\item\label{KC:OI-UWneg-notin-Eq}
If $C \nsubseteq \clOI$ and $\clUWneg \subseteq K$, then $K C \nsubseteq \clEq$.
\item\label{KC:S+OI-SminOO-notin-Smin}
If $C \nsubseteq \clS$, $C \subseteq \clOI$, and $\clSminOO \subseteq K$, then $K C \nsubseteq \clSmin$.
\item\label{KC:U-SM-notin-U}
If $C \nsubseteq \clU$ and $\clSM \subseteq K$, then $K C \nsubseteq \clU$.
\item\label{KC:U-UWneg-notin-U}
If $C \nsubseteq \clU$ and $\clUWneg \subseteq K$, then $K C \nsubseteq \clU$.
\item\label{KC:S-UWneg-notin-UWneg}
If $C \nsubseteq \clS$ and $\clUWneg \subseteq K$, then $K C \nsubseteq \clUWneg$.
\item\label{KC:M+S-UWneg-notin-U}
If $C \nsubseteq \clM$, $C \subseteq \clS$, and $\clUWneg \subseteq K$, then $K C \nsubseteq \clU$.
\item\label{KC:S-ROO-notin-R}
If $C \nsubseteq \clS$ and $\clReflOO \subseteq K$, then $K C \nsubseteq \clRefl$.
\end{enumerate}
\item
\begin{enumerate}[resume,label=\upshape{(\alph*)}]
\item\label{CK:M-OIC-notin-Eiio}
If $C \nsubseteq \clM$ and $\clOIC \subseteq K$, then $C K \nsubseteq \clEiio$.
\item\label{CK:U-Smin-notin-Eiii}
If $C \nsubseteq \clU$ and $\clSmin \subseteq K$, then $C K \nsubseteq \clEiii$.
\item\label{CK:S-S-notin-Neq}
If $C \nsubseteq \clS$ and $\clS \subseteq K$, then $C K \nsubseteq \clNeq$.
\item\label{CK:M-ROOC-notin-OXC}
If $C \nsubseteq \clM$ and $\clReflOOC \subseteq K$, then $C K \nsubseteq \clOXC$.
\item\label{CK:OX-ROO-notin-Eiii}
If $C \nsubseteq \clOX$ and $\clReflOO \subseteq K$, then $C K \nsubseteq \clEiii$.
\item\label{CK:Mo-OICO-notin-OICO}
If $C \nsubseteq \clMo$ and $\clOICO \subseteq K$, then $C K \nsubseteq \clOICO$.
\item\label{CK:U-UWneg-notin-Smin}
If $C \nsubseteq \clU$ and $\clUWneg \subseteq K$, then $C K \nsubseteq \clSmin$.
\item\label{CK:M+S-SminNeq-notin-Smin}
If $C \nsubseteq \clM$, $C \subseteq S$, and $\clSminNeq \subseteq K$, then $C K \nsubseteq \clSmin$.
\item\label{CK:MU-TcUCO-notin-SminOICO}
If $C \nsubseteq \clMU$ and $\clTcUCO \subseteq K$, then $C K \nsubseteq \clSminOICO$.
\item\label{CK:OI-SM-notin-OI}
If $C \nsubseteq \clOI$ and $\clSM \subseteq K$, then $C K \nsubseteq \clOI$.
\item\label{CK:U-McU-notin-Smin}
If $C \nsubseteq \clU$ and $\clMcU \subseteq K$, then $C K \nsubseteq \clSmin$.
\end{enumerate}
\end{enumerate}
\end{lemma}

\begin{proof}
\ref{KC:OI-Eq-notin-Eiio}
Since $C$ contains $0$, $1$, or $\neg$, and $\mathord{+}, \mathord{\leftrightarrow}, \mathord{\rightarrow} \in \clEq$,
we have that $\mathord{\leftrightarrow}(\id, 0) = \neg$, $\mathord{+}(\id, 1) = \neg$, or $\mathord{\rightarrow}(\neg, \id) = \neg$ is in $K C$, but $\neg \notin \clEiio$.

\ref{KC:OI-Neq-notin-Eiii}
Since $C$ contains $0$, $1$, or $\neg$, and $\id, \neg, \mathord{\vee} \in \clNeq$,
we have that $\neg(0) = 1$, $\id(1) = 1$, or $\mathord{\vee}(\id, \neg) = 1$ is in $K C$, but $1 \notin \clEiii$.

\ref{KC:OX-ROO-notin-OXC}
Since $C$ contains $1$ or $\neg$, and $\mathord{+} \in \clReflOO$,
we have that $\mathord{+}(\id, 1) = \neg$ or $\mathord{+}(\id_1, \neg_2) = \mathord{\leftrightarrow}$ is in $K C$, but $\neg, \mathord{\leftrightarrow} \notin \clOXC$.

\ref{KC:OI-ROO-notin-OOC}
Since $C$ contains $0$, $1$, or $\neg$, and $\mathord{+} \in \clReflOO$,
we have that $\mathord{+}(\id, 0) = \id$, $\mathord{+}(\id, 1) = \neg$, or $\mathord{+}(\id_1, \neg_2) = \mathord{\leftrightarrow}$ is in $K C$, but $\id, \neg, \mathord{\leftrightarrow} \notin \clOOC$.

\ref{KC:OI-OI-notin-OIC}
Since $C$ contains $0$, $1$, or $\neg$, and $\id, \mathord{\oplus_3} \in \clOI$,
we have that $K C$ contains $\mathord{\oplus_3}(\id_1, \id_2, 0) = \mathord{+}$, $\mathord{\oplus_3}(\id_1, \id_2, 1) = \mathord{\leftrightarrow}$, or $\id(\neg) = \neg$, but $\mathord{+}, \mathord{\leftrightarrow}, \neg \notin \clOIC$.

\ref{KC:S-Sc-notin-Smin}
There is $f \in C$ such that $f(0,1) = f(1,0) =: a$, and we have $\id, \mathord{\oplus_3} \in \clSc$.
If $a = 1$, then $id(f) = f \in K C$, but $f \notin \clSmin$.
If $a = 0$, then $\varphi := \mathord{\oplus_3}(\id_1, \id_2, f) \in K C$ but $\varphi \notin \clSmin$ because $\varphi(0,1) = \mathord{\oplus_3}(0,1,0) = 1$, $\varphi(1,0) = \mathord{\oplus_3}(1,0,0) = 1$.

\ref{KC:OI-McU-notin-Neq}
Since $C$ contains $0$, $1$, or $\neg$, and $\id, \mathord{\wedge} \in \clMcU$, we have that $\id(0) = 0$, $\id(1) = 1$, or $\mathord{\wedge}(\id, \neg) = 0$ is in $K C$, but $0, 1 \notin \clNeq$.

\ref{KC:OI+S-SM-notin-Eiio}
Since $\neg \in C$ and $\id \in \clSM$, we have $\id(\neg) = \neg \in K C$, but $\neg \notin \clEiio$.

\ref{KC:OI-UWneg-notin-Eq}
Since $C$ contains $0$, $1$, or $\neg$, and $\mathord{\rightarrow} \in \clUWneg$, we have that $\mathord{\rightarrow}(\id, 0) = \id$, $\mathord{\rightarrow}(1, \id) = \neg$, or $\mathord{\rightarrow}(\id, \neg)$ is in $K C$, but $\id, \neg \notin \clEq$.

\ref{KC:S+OI-SminOO-notin-Smin}
There is $f \in C$ such that $f(0,1) = f(1,0)$, $f(0,0) = 0$, $f(1,1) = 1$; in other words, $f \in \{\mathord{\wedge}, \mathord{\vee}\}$.
The ternary functions $g$ and $h$ satisfying $g^{-1}(1) = \{(1,0,0), (0,1,0), (0,0,1)\}$ and $h^{-1}(1) = \{(1,1,0), (1,0,1), (0,1,1)\}$ are members of $\clSminOO$.
Hence either $\alpha := g(\id_1, \id_2, \mathord{\wedge})$ or $\beta := h(\id_1, \id_2, \mathord{\wedge})$ is in $K C$, but $\alpha, \beta \notin \clSmin$ because
$\alpha(0,1) = g(0,1,0) = 1$, $\alpha(1,0) = g(1,0,0) = 1$, $\beta(0,1) = h(0,1,1) = 1$, $\beta(1,0) = h(1,0,1) = 1$.

\ref{KC:U-SM-notin-U}
The class $C$ contains an $f \notin \clU$.
Since $\id \in \clSM$, we have $\id(f) = f \in K C$.

\ref{KC:U-UWneg-notin-U}
There is $f \in C$ such that $f(0,0,1) = f(0,1,0) = 1$.
Since $\mathord{\nrightarrow} \in \clUWneg$, we have $\varphi := \mathord{\nrightarrow}(f, \id_1) \in K C$, but $\varphi \notin \clU$ because
$\varphi(0,0,1) = \mathord{\nrightarrow}(f(0,0,1), 0) = \mathord{\nrightarrow}(1,0) = 1$,
$\varphi(0,1,0) = \mathord{\nrightarrow}(f(0,1,0), 0) = \mathord{\nrightarrow}(1,0) = 1$,
and $(0,0,1) \wedge (0,1,0) = \vect{0}$.

\ref{KC:S-UWneg-notin-UWneg}
There is $f \in C$ such that $f(0,1) = f(1,0) =: a$.
We have $\mathord{\nrightarrow} \in \clUWneg$.
If $a = 0$, then $\varphi := \mathord{\nrightarrow}(\id_1, f_{23}) \in K C$, but $\varphi \notin \clWneg$ because $\varphi(1,0,1) = \varphi(1,1,0) = 1$ and $(1,0,1) \vee (1,1,0) = \vect{1}$.
If $a = 1$, then $\gamma := \mathord{\nrightarrow}(f_{12}, \id_3) \in K C$, but $\gamma \notin \clU$ because $\gamma(0,1,0) = \gamma(1,0,0) = 1$ and $(0,1,0) \wedge (1,0,0) = \vect{0}$.
Consequently, $KC \nsubseteq \clUWneg$.

\ref{KC:M+S-UWneg-notin-U}
There is $f \in C$ such that $f(0,0,1) = 1$, $f(0,1,1) = 0$, $f(1,1,0) = 0$, $f(1,0,0) = 1$.
Since $\mathord{\nrightarrow} \in \clUWneg$, we have $\delta := \mathord{\nrightarrow}(f_{123}, \id_4) \in K C$, but $\delta \notin \clU$ because $\delta(0,0,1,0) = \delta(1,0,0,0) = 1$ and $(0,0,1,0) \wedge (1,0,0,0) = \vect{0}$.
We conclude that $(\clUWneg) C \nsubseteq \clU$.

\ref{KC:S-ROO-notin-R}
There is $f \in C$ such that $f(0,1) = f(1,0) =: a$.
Since $\mathord{+} \in \clReflOO$, we have $\varphi := \mathord{+}(\id_1, f) \in K C$, but $\varphi \notin \clRefl$ because
$\varphi(0,1) = \mathord{+}(0,a) = a$, $\varphi(1,0) = \mathord{+}(1,a) = \overline{a}$.

\ref{CK:M-OIC-notin-Eiio}
There is $f \in C$ such that $f(0,0,1) = 1$, $f(0,1,1) = 0$.
Since $0, 1, \id \in \clOIC$, we have $\varphi := f(0, \id, 1) \in C K$, but $\varphi \notin \clEiio$ because $\varphi(0) = f(0,0,1) = 1$, $\varphi(1) = f(0,1,1) = 0$.

\ref{CK:U-Smin-notin-Eiii}
There is $f \in C$ such that $f(0,0,1) = f(0,1,0) = 1$.
Since $0, \id, \neg \in \clSmin$, we have $\varphi := f(0, \id_1, \neg_2) \in C K$, but $\varphi \notin \clEiii$ because $\varphi(0,0) = f(0,0,1) = 1$, $\varphi(1,1) = f(0,1,0) = 1$.

\ref{CK:S-S-notin-Neq}
There is $f \in C$ such that $f(0,1) = f(1,0)$.
Since $\id, \neg \in \clS$, we have $\varphi := f(\id, \neg) \in C K$, but $\varphi \notin \clNeq$ because $\varphi(0) = f(0,1) = f(1,0) = \varphi(1)$.

\ref{CK:M-ROOC-notin-OXC}
There is $f \in C$ such that $f(0,0,1) = 1$, $f(0,1,1) = 0$.
Since $0, 1, \mathord{+} \in \clReflOOC$, we have $\varphi := f(0, \mathord{+}, 1) \in C K$, but $\varphi \notin \clOXC$ because $\varphi(0,0) = f(0,0,1) = 1$, $\varphi(0,1) = f(0,1,1) = 0$.

\ref{CK:OX-ROO-notin-Eiii}
Since $C$ contains $1$ or $\neg$, and $0 \in \clReflOO$,
we have that $1(0) = 1$ or $\neg(0) = 1$ is in $C K$, but $1 \notin \clEiii$.

\ref{CK:Mo-OICO-notin-OICO}
We have $C \nsubseteq \clM$ or $C \nsubseteq \clOX$.
If $C \nsubseteq \clM$, then there is $f \in C$ such that $f(0,0,1) = 1$, $f(0,1,1) = 0$.
Since $0, \id \in \clOICO$, we have $\varphi := f(0, \id_1, \id_2) \in C K$, but $\varphi \notin \clXIC$ because $\varphi(1,1) = f(0,1,1) = 0$, $\varphi(0,1) = f(0,0,1) = 1$.
If $C \subseteq \clOX$, then $C$ contains $1$ or $\neg$.
Then $1(\id) = 1$ or $\neg(\id) = \neg$ is in $C K$, but $1 \notin \clOX$ and $\neg \notin \clEiio$.
Consequently, $C K \nsubseteq (\clXIC) \cap \clOX \cap \clEiio = \clOIC$.

\ref{CK:U-UWneg-notin-Smin}
There is $f \in C$ such that $f(0,0,1) = f(0,1,0) = 1$.
Since $0, \mathord{\nrightarrow} \in \clUWneg$, we have $\varphi := f(0, \mathord{\nrightarrow}_{12}, \mathord{\nrightarrow}_{21}) \in C K$, but $\varphi \notin \clSmin$ because $\varphi(0,1) = f(0,0,1) = 1$, $\varphi(1,0) = f(0,1,0) = 1$.

\ref{CK:M+S-SminNeq-notin-Smin}
There is $f \in C$ such that $f(0,0,1) = 1$, $f(0,1,1) = 0$, $f(1,1,0) = 0$, $f(1,0,0) = 1$.
Since $\id, \neg, \mathord{\wedge} \in \clSminNeq$, we have $\varphi := f(\id_1, \mathord{\wedge}, \neg_1) \in C K$, but $\varphi \notin \clSmin$ because $\varphi(0,1) = f(0,0,1) = 1$, $\varphi(1,0) = f(1,0,0) = 1$.

\ref{CK:MU-TcUCO-notin-SminOICO}
The class $C$ contains either an $f$ such that $f(0,0,1) = 1$, $f(0,1,1) = 0$ or a $g$ such that $g(0,0,1) = g(0,1,0) = 1$.
Since $0, \id \in \clTcUCO$, we have that $\varphi := f(0, \id_1, \id_2)$ or $\gamma := g(0, \id_1, \id_2)$ is in $C K$,
but $\varphi, \gamma \notin \clSminOICO$ because
$\varphi(0,1) = f(0,0,1) = 1$, $\varphi(1,1) = f(0,1,1) = 0$,
$\gamma(0,1) = g(0,0,1) = 1$, $\gamma(1,0) = g(0,1,0) = 1$.

\ref{CK:OI-SM-notin-OI}
There is an $f \in C$ such that $f \notin \clOI$.
Since $\id \in \clSM$, it holds that $f(\id_1, \dots, \id_n) = f \in C K$.

\ref{CK:U-McU-notin-Smin}
There is $f \in C$ such that $f(0,0,1) = f(0,1,0) = 1$.
Since $\id, \mathord{\wedge} \in \clMcU$, we have $\varphi := f(\mathord{\wedge}, \id_1, \id_2) \in C K$, but $\varphi \notin \clSmin$ because $\varphi(0,1) = f(0,0,1) = 1$, $\varphi(1,0) = f(0,1,0) = 1$.
\end{proof}

\begin{proposition}
\label{prop:LR:clones}
For any clone $C$ and for $K$ being one of the classes
$\clAll$, $\clOX$, $\clXI$, $\clOI$, $\clM$, $\clMo$, $\clMi$, $\clMc$, $\clU$, $\clTcU$, $\clMU$, $\clMcU$, $\clW$, $\clTcW$, $\clMW$, $\clMcW$, $\clS$, $\clSc$, and $\clSM$,
the following inclusions are equivalent:
\begin{align*}
& C \subseteq K, \\
& K C \subseteq K, &
& \overline{K} C \subseteq \overline{K}, &
& K^\mathrm{n} C^\mathrm{d} \subseteq K^\mathrm{n}, &
& K^\mathrm{d} C^\mathrm{d} \subseteq K^\mathrm{d}, \\
& C K \subseteq K, &
& C K^\mathrm{n} \subseteq K^\mathrm{n}, &
& C^\mathrm{d} \overline{K} \subseteq \overline{K}, &
& C^\mathrm{d} K^\mathrm{d} \subseteq K^\mathrm{d}.
\end{align*}
\end{proposition}

\begin{proof}
Since the classes listed in the statement are clones, the result follows immediately from Lemmata~\ref{lem:clones-stable} and \ref{lem:neg-dual}.
\end{proof}

\begin{proposition}
\label{prop:LR:Ei}
Let $C$ be a clone.
\begin{enumerate}[label=\upshape{(\roman*)}]
\item\label{prop:LR:Ei:Eiio-R}
$\clEiio C \subseteq \clEiio$ if and only if $C \subseteq \clOI$.
\item\label{prop:LR:Ei:Eiio-L}
$C \clEiio \subseteq \clEiio$ if and only if $C \subseteq \clM$.
\item\label{prop:LR:Ei:Eiii-R}
$\clEiii C \subseteq \clEiii$ if and only if $C \subseteq \clOI$.
\item\label{prop:LR:Ei:Eiii-L}
$C \clEiii \subseteq \clEiii$ if and only if $C \subseteq \clU$.
\end{enumerate}
\end{proposition}

\begin{proof}
\ref{prop:LR:Ei:Eiio-R}
Assume first that $C \subseteq \clOI$.
Let $f \in \clEiio$ and $g_1, \dots, g_n \in C$.
Then $g_i(\vect{0}) = 0$ and $g_i(\vect{1}) = 1$ for all $i \in \nset{n}$, so we have
\[
\begin{split}
\lhs
f(g_1, \dots, g_n)(\vect{0})
= f(g_1(\vect{0}), \dots, g_n(\vect{0}))
= f(\vect{0})
\leq f(\vect{1})
\\ &
= f(g_1(\vect{1}), \dots, g_n(\vect{1}))
= f(g_1, \dots, g_n)(\vect{1}),
\end{split}
\]
which implies that $f(g_1, \dots, g_n) \in \clEiio$.
Therefore $\clEiio C \subseteq \clEiio$.
Conversely, if $C \nsubseteq \clOI$, then $\clEiio C \nsubseteq \clEiio$ by Lemma~\ref{lem:KCCK}\ref{KC:OI-Eq-notin-Eiio}.

\ref{prop:LR:Ei:Eiio-L}
Assume first that $C \subseteq \clM$.
Let $f \in C$ and $g_1, \dots, g_n \in \clEiio$.
Then $g_i(\vect{0}) \leq g_i(\vect{1})$ for all $i \in \nset{n}$, so we have
\[
f(g_1, \dots, g_n)(\vect{0}) = f(g_1(\vect{0}), \dots, g_1(\vect{0})) \leq f(g_1(\vect{1}), \dots, f(g_n(\vect{1}))) = f(g_1, \dots, g_n)(\vect{1}).
\]
Therefore $f(g_1, \dots, g_n) \in \clEiio$, and we have $C \clEiio \subseteq \clEiio$.
Conversely, if $C \nsubseteq \clM$, then $C \clEiio \nsubseteq \clEiio$ by Lemma~\ref{lem:KCCK}\ref{CK:M-OIC-notin-Eiio}.

\ref{prop:LR:Ei:Eiii-R}
Assume first that $C \subseteq \clOI$.
Let $f \in \clEiii$ and $g_1, \dots, g_n \in C$.
Then $g_i(\vect{0}) = 0$ and $g_i(\vect{1}) = 1$ for all $i \in \nset{n}$, so we have
\begin{align*}
f(g_1, \dots, g_n)(\vect{0}) = f(g_1(\vect{0}), \dots, g_n(\vect{0})) = f(\vect{0}), \\
f(g_1, \dots, g_n)(\vect{1}) = f(g_1(\vect{1}), \dots, g_n(\vect{1})) = f(\vect{1}),
\end{align*}
which implies that $f(g_1, \dots, g_n) \in \clEiii$.
Therefore $\clEiii C \subseteq \clEiii$.
Conversely, if $C \nsubseteq \clOI$, then $\clEiii C \nsubseteq \clEiii$ by Lemma~\ref{lem:KCCK}\ref{KC:OI-Neq-notin-Eiii}.

\ref{prop:LR:Ei:Eiii-L}
Assume first that $C \subseteq \clU$.
Let $f \in C$ and $g_1, \dots, g_n \in \clEiii$.
Then $g_i(\vect{0}) \wedge g_i(\vect{1}) = 0$ for all $i \in \nset{n}$.
Therefore $(g_1(\vect{0}), \dots, g_n(\vect{0}))$ and $(g_1(\vect{1}), \dots, g_n(\vect{1}))$ cannot both be true points of $f$, so we have
\[
f(g_1, \dots, g_n)(\vect{0}) \wedge f(g_1, \dots, g_n)(\vect{1})
= f(g_1(\vect{0}), \dots, g_n(\vect{0})) \wedge f(g_1(\vect{1}), \dots, g_n(\vect{1}))
= 0,
\]
which means that $f(g_1, \dots, g_n) \in \clEiii$.
Conversely, if $C \nsubseteq \clU$, then we have $C \clEiii \nsubseteq \clEiii$ by Lemma~\ref{lem:KCCK}\ref{CK:U-Smin-notin-Eiii}
\end{proof}

\begin{proposition}
\label{prop:LR:EqNeq}
Let $C$ be a clone.
\begin{enumerate}[label=\upshape{(\roman*)}]
\item\label{prop:LR:EqNeq:Eq-R}
$\clEq C \subseteq \clEq$ if and only if $C \subseteq \clOI$.
\item\label{prop:LR:EqNeq:Eq-L}
$C \clEq \subseteq \clEq$ if and only if $C \subseteq \clAll$.
\item\label{prop:LR:EqNeq:Neq-R}
$\clNeq C \subseteq \clNeq$ if and only if $C \subseteq \clOI$.
\item\label{prop:LR:EqNeq:Neq-L}
$C \clNeq \subseteq \clNeq$ if and only if $C \subseteq \clS$.
\end{enumerate}
\end{proposition}

\begin{proof}
\ref{prop:LR:EqNeq:Eq-R}
Assume first that $C \subseteq \clOI$.
Since $\clEq = \clEiio \cap \clEioi$ and $\clOI = \clOI \cap \clOI$, Lemma~\ref{lem:intersection} and Proposition~\ref{prop:LR:Ei}\ref{prop:LR:Ei:Eiio-R}, together with Lemma~\ref{lem:neg-dual}, imply that $\clEq C \subseteq \clEq$.
Conversely, if $C \nsubseteq \clOI$, then $\clEq C \nsubseteq \clEiio \supseteq \clEq$ by Lemma~\ref{lem:KCCK}\ref{KC:OI-Eq-notin-Eiio}.

\ref{prop:LR:EqNeq:Eq-L}
For any $f \in \clAll$ and $g_1, \dots, g_n \in \clEq$, it holds that
\[
f(g_1, \dots, g_n)(\vect{0}) = f(g_1(\vect{0}), \dots, g_n(\vect{0})) = f(g_1(\vect{1}), \dots, g_n(\vect{1})) = f(g_1, \dots, g_n)(\vect{1}),
\]
so $f(g_1, \dots, g_n) \in \clEq$.
This shows that $C \clEq \subseteq \clEq$ for any clone $C$.

\ref{prop:LR:EqNeq:Neq-R}
Assume first that $C \subseteq \clOI$.
Since $\clNeq = \clEiii \cap \clEioo$ and $\clOI = \clOI \cap \clOI$, Lemma~\ref{lem:intersection} and Proposition~\ref{prop:LR:Ei}\ref{prop:LR:Ei:Eiii-R}, together with Lemma~\ref{lem:neg-dual}, imply that $\clNeq C \subseteq \clNeq$.
Conversely, if $C \nsubseteq \clOI$, then $\clNeq C \nsubseteq \clEiii \supseteq \clNeq$ by Lemma~\ref{lem:KCCK}\ref{KC:OI-Neq-notin-Eiii}.

\ref{prop:LR:EqNeq:Neq-L}
Assume first that $C \subseteq \clS$.
Let $f \in C$ and $g_1, \dots, g_n \in \clNeq$.
Then $g_i(\vect{0}) \neq g_i(\vect{1})$ for all $i \in \nset{n}$, so we have
\[
f(g_1, \dots, g_n)(\vect{0}) = f(g_1(\vect{0}), \dots, g_n(\vect{0})) \neq f(g_1(\vect{1}), \dots, g_n(\vect{1})) = f(g_1, \dots, g_n)(\vect{1}),
\]
which implies that $f(g_1, \dots, g_n) \in \clNeq$.
Therefore $C \clNeq \subseteq \clNeq$.
Conversely, if $C \nsubseteq \clS$, then $C \clNeq \nsubseteq \clNeq$ by Lemma~\ref{lem:KCCK}\ref{CK:S-S-notin-Neq}.
\end{proof}

\begin{proposition}
\label{prop:LR:OXC}
Let $C$ be a clone.
\begin{enumerate}[label=\upshape{(\roman*)}]
\item\label{prop:LR:OXC:OXC-R}
$(\clOXC) C \subseteq \clOXC$ if and only if $C \subseteq \clOX$.
\item\label{prop:LR:OXC:OXC-L}
$C (\clOXC) \subseteq \clOXC$ if and only if $C \subseteq \clM$.
\end{enumerate}
\end{proposition}

\begin{proof}
\ref{prop:LR:OXC:OXC-R}
Assume first that $C \subseteq \clOX$.
Let $f \in \clOXC$ and $g_1, \dots, g_n \in C$.
If $f \in \clVak$, then $f(g_1, \dots, g_n) \in \clVak \subseteq \clOXC$.
If $f \in \clOX$, then $f(g_1, \dots, g_n)(\vect{0}) = f(g_1(\vect{0}), \dots, g_n(\vect{0})) = f(\vect{0}) = 0$, so $f(g_1, \dots, g_n) \in \clOX \subseteq \clOXC$.
Therefore $(\clOXC) C \subseteq \clOXC$.
Conversely, if $C \nsubseteq \clOX$, then $(\clOXC) C \nsubseteq \clOXC$ by Lemma~\ref{lem:KCCK}\ref{KC:OX-ROO-notin-OXC}.

\ref{prop:LR:OXC:OXC-L}
Assume first that $C \subseteq \clM$.
Let $f \in C$ and $g_1, \dots, g_n \in \clOXC$.
Observe that, for any $i \in \nset{n}$ and for every $\vect{a} \in \{0,1\}^n$, it holds that $g_i(\vect{0}) \leq g_i(\vect{a})$ (if $g_i \in \clOX$, then $g_i(\vect{0}) = 0 \leq g_i(\vect{a})$; if $g_i \in \clVak$, then $g_i(\vect{0}) = g_i(\vect{a})$).
Since $f$ is monotone, it follows that
\[
f(g_1, \dots, g_n)(\vect{0}) = f(g_1(\vect{0}), \dots, g_n(\vect{0})) \leq f(g_1(\vect{a}), \dots, g_n(\vect{a})) = f(g_1, \dots, g_n)(\vect{a}),
\]
for every $\vect{a} \in \{0,1\}^n$.
If $f(g_1, \dots, g_n)(\vect{0}) = 0$, then $f(g_1, \dots, g_n) \in \clOX$.
If $f(g_1, \dots, g_n)(\vect{0}) = 1$, then $f(g_1, \dots, g_n) = 1 \in \clVak$.
Therefore $C (\clOXC) \subseteq \clOXC$.
Conversely, if $C \nsubseteq \clM$, then $C (\clOXC) \nsubseteq \clEiio \supseteq \clOXC$ by Lemma~\ref{lem:KCCK}\ref{CK:M-OIC-notin-Eiio}.
\end{proof}

\begin{proposition}
\label{prop:LR:OOC}
Let $C$ be a clone.
\begin{enumerate}[label=\upshape{(\roman*)}]
\item\label{prop:LR:OOC:OOC-R}
$(\clOOC) C \subseteq \clOOC$ if and only if $C \subseteq \clOI$.
\item\label{prop:LR:OOC:OOC-L}
$C (\clOOC) \subseteq \clOOC$ if and only if $C \subseteq \clM$.
\end{enumerate}
\end{proposition}

\begin{proof}
\ref{prop:LR:OOC:OOC-R}
Assume first that $C \subseteq \clOI$.
Since $\clOOC = (\clOXC) \cap (\clXOC)$ and $\clOI = \clOX \cap \clXI$, Lemma~\ref{lem:intersection} and Proposition~\ref{prop:LR:OXC}\ref{prop:LR:OXC:OXC-R}, together with Lemma~\ref{lem:neg-dual}, imply that $(\clOOC) C \subseteq \clOOC$.
Conversely, if $C \nsubseteq \clOI$, then $(\clOOC) C \nsubseteq \clOOC$ by Lemma~\ref{lem:KCCK}\ref{KC:OI-ROO-notin-OOC}.

\ref{prop:LR:OOC:OOC-L}
Assume first that $C \subseteq \clM$.
Since $\clOOC = (\clOXC) \cap (\clXOC)$ and $\clM = \clM \cap \clM$, Lemma~\ref{lem:intersection} and Proposition~\ref{prop:LR:OXC}\ref{prop:LR:OXC:OXC-L}, together with Lemma~\ref{lem:neg-dual}, imply that $C (\clOOC) \subseteq \clOOC$.
Conversely, if $C \nsubseteq \clM$, then $C (\clOOC) \nsubseteq \clOXC \supseteq \clOOC$ by Lemma~\ref{lem:KCCK}\ref{CK:M-ROOC-notin-OXC}.
\end{proof}

\begin{proposition}
\label{prop:LR:OO}
Let $C$ be a clone.
\begin{enumerate}[label=\upshape{(\roman*)}]
\item\label{prop:LR:OO:OO-R}
$\clOO C \subseteq \clOO$ if and only if $C \subseteq \clOI$.
\item\label{prop:LR:OO:OO-L}
$C \clOO \subseteq \clOO$ if and only if $C \subseteq \clOX$.
\end{enumerate}
\end{proposition}

\begin{proof}
\ref{prop:LR:OO:OO-R}
Assume first that $C \subseteq \clOI$.
Since $\clOO = \clOX \cap \clXO$ and $\clOI = \clOX \cap \clXI$, Lemma~\ref{lem:intersection} and Proposition~\ref{prop:LR:clones} imply that $\clOO C \subseteq \clOO$.
Conversely, if $C \nsubseteq \clOI$, then $\clOO C \nsubseteq \clOOC \supseteq \clOO$ by Lemma~\ref{lem:KCCK}\ref{KC:OI-ROO-notin-OOC}.

\ref{prop:LR:OO:OO-L}
Assume first that $C \subseteq \clOX$.
Since $\clOO = \clOX \cap \clXO$ and $\clOX = \clOX \cap \clOX$, Lemma~\ref{lem:intersection} and Proposition~\ref{prop:LR:clones} imply that $C \clOO \subseteq \clOO$.
Conversely, if $C \nsubseteq \clOX$, then $C \clOO \nsubseteq \clEiii \supseteq \clOO$ by Lemma~\ref{lem:KCCK}\ref{CK:OX-ROO-notin-Eiii}.
\end{proof}

\begin{proposition}
\label{prop:LR:OIC}
Let $C$ be a clone.
\begin{enumerate}[label=\upshape{(\roman*)}]
\item\label{prop:LR:OIC:OIC-R}
$(\clOIC) C \subseteq \clOIC$ if and only if $C \subseteq \clOI$.
\item\label{prop:LR:OIC:OIC-L}
$C (\clOIC) \subseteq \clOIC$ if and only if $C \subseteq \clM$.
\end{enumerate}
\end{proposition}

\begin{proof}
\ref{prop:LR:OIC:OIC-R}
Assume first that $C \subseteq \clOI$.
Since $\clOIC = (\clOXC) \cap (\clXIC)$ and $\clOI = \clOX \cap \clXI$, Lemma~\ref{lem:intersection} and Proposition~\ref{prop:LR:OXC}\ref{prop:LR:OXC:OXC-R}, together with Lemma~\ref{lem:neg-dual}, imply that $(\clOIC) C \subseteq \clOIC$.
Conversely, if $C \nsubseteq \clOI$, then $(\clOIC) C \subseteq \clOIC$ by Lemma~\ref{lem:KCCK}\ref{KC:OI-OI-notin-OIC}.

\ref{prop:LR:OIC:OIC-L}
Assume first that $C \subseteq \clM$.
Since $\clOIC = (\clOXC) \cap (\clXIC)$ and $\clM = \clM \cap \clM$, Lemma~\ref{lem:intersection} and Proposition~\ref{prop:LR:OXC}\ref{prop:LR:OXC:OXC-L}, together with Lemma~\ref{lem:neg-dual}, imply that $C (\clOIC) \subseteq \clOIC$.
Conversely, if $C \nsubseteq \clM$, then $C (\clOIC) \nsubseteq \clEiio \supseteq \clOIC$ by Lemma~\ref{lem:KCCK}\ref{CK:M-OIC-notin-Eiio}.
\end{proof}

\begin{proposition}
\label{prop:LR:OICO}
Let $C$ be a clone.
\begin{enumerate}[label=\upshape{(\roman*)}]
\item\label{prop:LR:OICO:OICO-R}
$(\clOICO) C \subseteq \clOICO$ if and only if $C \subseteq \clOI$.
\item\label{prop:LR:OICO:OICO-L}
$C (\clOICO) \subseteq \clOICO$ if and only if $C \subseteq \clMo$.
\end{enumerate}
\end{proposition}

\begin{proof}
\ref{prop:LR:OICO:OICO-R}
Assume first that $C \subseteq \clOI$.
Since $\clOICO = \clOX \cap (\clOIC)$ and $\clOI = \clOX \cap \clOI$, Lemma~\ref{lem:intersection} and Propositions~\ref{prop:LR:clones} and \ref{prop:LR:OIC}\ref{prop:LR:OIC:OIC-R}
imply that $(\clOICO) C \subseteq \clOICO$.
Conversely, if $C \nsubseteq \clOI$, then $(\clOICO) C \subseteq \clOIC \supseteq \clOICO$ by Lemma~\ref{lem:KCCK}\ref{KC:OI-OI-notin-OIC}.

\ref{prop:LR:OICO:OICO-L}
Assume first that $C \subseteq \clMo$.
Since $\clOICO = \clOX \cap (\clOIC)$ and $\clMo = \clOX \cap \clM$, Lemma~\ref{lem:intersection} and Propositions~\ref{prop:LR:clones} and \ref{prop:LR:OIC}\ref{prop:LR:OIC:OIC-L} imply that $C (\clOICO) \subseteq \clOICO$.
Conversely, if $C \nsubseteq \clMo$, then $C (\clOICO) \nsubseteq \clOICO$ by Lemma~\ref{lem:KCCK}\ref{CK:Mo-OICO-notin-OICO}.
\end{proof}

\begin{proposition}
\label{prop:LR:Smin}
Let $C$ be a clone.
\begin{enumerate}[label=\upshape{(\roman*)}]
\item\label{prop:LR:Smin:Smin-R}
$\clSmin C \subseteq \clSmin$ if and only if $C \subseteq \clS$.
\item\label{prop:LR:Smin:Smin-L}
$C \clSmin \subseteq \clSmin$ if and only if $C \subseteq \clU$.
\end{enumerate}
\end{proposition}

\begin{proof}
\ref{prop:LR:Smin:Smin-R}
Assume first that $C \subseteq \clS$.
Let $f \in \clSmin$ and $g_1, \dots, g_n \in C$.
Then, for every $\vect{a} \in \{0,1\}^n$,
\[
\begin{split}
\lhs
f(g_1, \dots, g_n)(\vect{a}) \wedge f(g_1, \dots, g_n)(\overline{\vect{a}})
= f(g_1(\vect{a}), \dots, g_n(\vect{a})) \wedge f(g_1(\overline{\vect{a}}), \dots, g_n(\overline{\vect{a}}))
\\ &
= f(g_1(\vect{a}), \dots, g_n(\vect{a})) \wedge f(\overline{g_1(\vect{a})}, \dots, \overline{g_n(\vect{a})})
= 0,
\end{split}
\]
where the second equality holds because each $g_i \in \clS$, and the last equality holds because $f \in \clSmin$.
This shows that $f(g_1, \dots, g_n) \in \clSmin$.
Therefore $\clSmin C \subseteq \clSmin$.
Conversely, if $C \nsubseteq \clS$, then $\clSmin C \nsubseteq \clSmin$ by Lemma~\ref{lem:KCCK}\ref{KC:S-Sc-notin-Smin}.

\ref{prop:LR:Smin:Smin-L}
Assume first that $C \subseteq \clU$.
Let $f \in C$ and $g_1, \dots, g_n \in \clSmin$.
Suppose, to the contrary, that $f(g_1, \dots, g_n)(\vect{a}) \wedge f(g_1, \dots, g_n)(\overline{\vect{a}}) = 1$ for some $\vect{a}$.
Then $f(g_1(\vect{a}), \dots, g_n(\vect{a})) = f(g_1(\overline{\vect{a}}), \dots, g_n(\overline{\vect{a}})) = 1$.
Since $f \in \clU$, there exists an $i \in \nset{n}$ such that $g_i(\vect{a}) = g_i(\overline{\vect{a}}) = 1$.
But then $g_i(\vect{a}) \wedge g_i(\overline{\vect{a}}) = 1$, contradicting the fact that $g_i \in \clSmin$.
We conclude that $f(g_1, \dots, g_n) \in \clSmin$.
Therefore $C \clSmin \subseteq \clSmin$.
Conversely, if $C \nsubseteq \clU$, then $C \clSmin \nsubseteq \clSmin$ by Lemma~\ref{lem:KCCK}\ref{CK:U-UWneg-notin-Smin} (or by Lemma~\ref{lem:KCCK}\ref{CK:U-Smin-notin-Eiii}).
\end{proof}

\begin{proposition}
\label{prop:LR:SminNeq}
Let $C$ be a clone.
\begin{enumerate}[label=\upshape{(\roman*)}]
\item\label{prop:LR:SminNeq:SminNeq-R}
$\clSminNeq C \subseteq \clSminNeq$ if and only if $C \subseteq \clSc$.
\item\label{prop:LR:SminNeq:SminNeq-L}
$C \clSminNeq \subseteq \clSminNeq$ if and only if $C \subseteq \clSM$.
\end{enumerate}
\end{proposition}

\begin{proof}
\ref{prop:LR:SminNeq:SminNeq-R}
Assume first that $C \subseteq \clSc$.
Since $\clSminNeq = \clNeq \cap \clSmin$ and $\clSc = \clOI \cap \clS$, Lemma~\ref{lem:intersection} and Propositions~\ref{prop:LR:EqNeq}\ref{prop:LR:EqNeq:Neq-R} and \ref{prop:LR:Smin}\ref{prop:LR:Smin:Smin-R} imply that $\clSminNeq C \subseteq \clSminNeq$.

Conversely, assume that $C \nsubseteq \clSc$.
Then $C \nsubseteq \clS$ or $C \nsubseteq \clOI$.
If $C \nsubseteq \clS$, then $\clSminNeq C \nsubseteq \clSmin$ by Lemma~\ref{lem:KCCK}\ref{KC:S-Sc-notin-Smin}.
If $C \nsubseteq \clOI$, then $\clSminNeq C \nsubseteq \clNeq$ by Lemma~\ref{lem:KCCK}\ref{KC:OI-McU-notin-Neq}.
Consequently, $\clSminNeq C \nsubseteq \clSmin \cap \clNeq = \clSminNeq$.

\ref{prop:LR:SminNeq:SminNeq-L}
We have already established in Theorem~\ref{thm:SM-stable} that $C \clSminNeq \subseteq \clSminNeq$ if $C \subseteq \clSM$.
Conversely, assume that $C \nsubseteq \clSM$.
Then $C \nsubseteq \clS$ or $C \nsubseteq \clM$.
If $C \nsubseteq \clS$, then $C \clSminNeq \nsubseteq \clNeq$ by Lemma~\ref{lem:KCCK}\ref{CK:S-S-notin-Neq}.
Otherwise we have $C \nsubseteq \clM$ and $C \subseteq \clS$; in this case we have $C \clSminNeq \nsubseteq \clSmin$ by Lemma~\ref{lem:KCCK}\ref{CK:M+S-SminNeq-notin-Smin}.
Consequently, $C \clSminNeq \nsubseteq \clNeq \cap \clSmin = \clSminNeq$.
\end{proof}

\begin{proposition}
\label{prop:LR:SminOX}
Let $C$ be a clone.
\begin{enumerate}[label=\upshape{(\roman*)}]
\item\label{prop:LR:SminOX:SminOX-R}
$\clSminOX C \subseteq \clSminOX$ if and only if $C \subseteq \clSc$.
\item\label{prop:LR:SminOX:SminOX-L}
$C \clSminOX \subseteq \clSminOX$ if and only if $C \subseteq \clU$.
\end{enumerate}
\end{proposition}

\begin{proof}
\ref{prop:LR:SminOX:SminOX-R}
Assume first that $C \subseteq \clSc$.
Since $\clSminOX = \clOX \cap \clSmin$ and $\clSc = \clOX \cap \clS$, Lemma~\ref{lem:intersection} and Propositions~\ref{prop:LR:clones} and \ref{prop:LR:Smin}\ref{prop:LR:Smin:Smin-R} imply that $\clSminOX C \subseteq \clSminOX$.

Conversely, assume that $C \nsubseteq \clSc$.
Then $C \nsubseteq \clS$ or $C \nsubseteq \clOI$.
If $C \nsubseteq \clS$, then $\clSminOX C \nsubseteq \clSmin$ by Lemma~\ref{lem:KCCK}\ref{KC:S-Sc-notin-Smin}.
Otherwise $C \nsubseteq \clOI$ and $C \subseteq \clS$; in this case we have $\clSminOX C \nsubseteq \clEiio$ by Lemma~\ref{lem:KCCK}\ref{KC:OI+S-SM-notin-Eiio}.
Consequently, $\clSminOX C \nsubseteq \clSmin \cap \clEiio = \clSminOX$.

\ref{prop:LR:SminOX:SminOX-L}
Assume first that $C \subseteq \clU$.
Since $\clSminOX = \clOX \cap \clSmin$ and $\clU = \clOX \cap \clU$, Lemma~\ref{lem:intersection} and Propositions~\ref{prop:LR:clones} and \ref{prop:LR:Smin}\ref{prop:LR:Smin:Smin-L} imply that $C \clSminOX \subseteq \clSminOX$.
Conversely, if $C \nsubseteq \clU$, then $C \clSminOX \nsubseteq \clSmin \supseteq \clSminOX$ by Lemma~\ref{lem:KCCK}\ref{CK:U-UWneg-notin-Smin}.
\end{proof}

\begin{proposition}
\label{prop:LR:SminOICO}
Let $C$ be a clone.
\begin{enumerate}[label=\upshape{(\roman*)}]
\item\label{prop:LR:SminOICO:SminOICO-R}
$(\clSminOICO) C \subseteq \clSminOICO$ if and only if $C \subseteq \clSc$.
\item\label{prop:LR:SminOICO:SminOICO-L}
$C (\clSminOICO) \subseteq \clSminOICO$ if and only if $C \subseteq \clMU$.
\end{enumerate}
\end{proposition}

\begin{proof}
\ref{prop:LR:SminOICO:SminOICO-R}
Assume first that $C \subseteq \clSc$.
Since $\clSminOICO = (\clOICO) \cap \clSminOX$ and $\clSc = \clOI \cap \clSc$, Lemma~\ref{lem:intersection} and Propositions~\ref{prop:LR:OICO}\ref{prop:LR:OICO:OICO-R} and \ref{prop:LR:SminOX}\ref{prop:LR:SminOX:SminOX-R} imply that $(\clSminOICO) C \subseteq \clSminOICO$.

Conversely, assume that $C \nsubseteq \clSc$.
Then $C \nsubseteq \clS$ or $C \nsubseteq \clOI$.
If $C \nsubseteq \clS$, then $(\clSminOICO) C \nsubseteq \clSmin$ by Lemma~\ref{lem:KCCK}\ref{KC:S-Sc-notin-Smin}.
Otherwise $C \nsubseteq \clOI$ and $C \subseteq \clS$; in this case we have $(\clSminOICO) C \nsubseteq \clEiio$ by Lemma~\ref{lem:KCCK}\ref{KC:OI+S-SM-notin-Eiio}.
Consequently, $(\clSminOICO) C \nsubseteq \clSmin \cap \clEiio = \clSminOX \supseteq \clSminOICO$.

\ref{prop:LR:SminOICO:SminOICO-L}
Assume first that $C \subseteq \clMU$.
Since $\clSminOICO = (\clOICO) \cap \clSminOX$ and $\clMU = \clMo \cap \clU$, Lemma~\ref{lem:intersection} and Propositions~\ref{prop:LR:OICO}\ref{prop:LR:OICO:OICO-L} and \ref{prop:LR:SminOX}\ref{prop:LR:SminOX:SminOX-L} imply that $(\clSminOICO) C \subseteq \clSminOICO$.
Conversely, if $C \nsubseteq \clMU$, then $C (\clSminOICO) \nsubseteq \clSminOICO$ by Lemma~\ref{lem:KCCK}\ref{CK:MU-TcUCO-notin-SminOICO}.
\end{proof}

\begin{proposition}
\label{prop:LR:SminOI}
Let $C$ be a clone.
\begin{enumerate}[label=\upshape{(\roman*)}]
\item\label{prop:LR:SminOI:SminOI-R}
$\clSminOI C \subseteq \clSminOI$ if and only if $C \subseteq \clSc$.
\item\label{prop:LR:SminOI:SminOI-L}
$C \clSminOI \subseteq \clSminOI$ if and only if $C \subseteq \clTcU$.
\end{enumerate}
\end{proposition}

\begin{proof}
\ref{prop:LR:SminOI:SminOI-R}
Assume first that $C \subseteq \clSc$.
Since $\clSminOI = \clOI \cap \clSmin$ and $\clSc = \clOI \cap \clS$, Lemma~\ref{lem:intersection} and Propositions~\ref{prop:LR:clones} and \ref{prop:LR:Smin}\ref{prop:LR:Smin:Smin-R} imply that $\clSminOI C \subseteq \clSminOI$.

Conversely, assume that $C \nsubseteq \clSc$.
Then $C \nsubseteq \clS$ or $C \nsubseteq \clOI$.
If $C \nsubseteq \clS$, then $\clSminOI C \nsubseteq \clSmin$ by Lemma~\ref{lem:KCCK}\ref{KC:S-Sc-notin-Smin}.
Otherwise $C \nsubseteq \clOI$ and $C \subseteq \clS$; in this case we have $\clSminOI C \nsubseteq \clEiio$ by Lemma~\ref{lem:KCCK}\ref{KC:OI+S-SM-notin-Eiio}.
Consequently, $\clSminOI C \nsubseteq \clSmin \cap \clEiio = \clSminOX \supseteq \clSminOI$.

\ref{prop:LR:SminOI:SminOI-L}
Assume first that $C \subseteq \clTcU$.
Since $\clSminOI = \clOI \cap \clSmin$ and $\clTcU = \clOI \cap \clU$, Lemma~\ref{lem:intersection} and Propositions~\ref{prop:LR:clones} and \ref{prop:LR:Smin}\ref{prop:LR:Smin:Smin-L} imply that $C \clSminOI \subseteq \clSminOI$.

Conversely, assume that $C \nsubseteq \clTcU$.
Then $C \nsubseteq \clOI$ or $C \nsubseteq \clU$.
If $C \nsubseteq \clOI$, then $C \clSminOI \nsubseteq \clOI$ by Lemma~\ref{lem:KCCK}\ref{CK:OI-SM-notin-OI}.
If $C \nsubseteq \clU$, then $C \clSminOI \nsubseteq \clSmin$ by Lemma~\ref{lem:KCCK}\ref{CK:U-McU-notin-Smin}.
Consequently, $C \clSminOI \nsubseteq \clOI \cap \clSmin = \clSminOI$.
\end{proof}

\begin{proposition}
\label{prop:LR:SminOO}
Let $C$ be a clone.
\begin{enumerate}[label=\upshape{(\roman*)}]
\item\label{prop:LR:SminOO:SminOO-R}
$\clSminOO C \subseteq \clSminOO$ if and only if $C \subseteq \clSc$.
\item\label{prop:LR:SminOO:SminOO-L}
$C \clSminOO \subseteq \clSminOO$ if and only if $C \subseteq \clU$.
\end{enumerate}
\end{proposition}

\begin{proof}
\ref{prop:LR:SminOO:SminOO-R}
Assume first that $C \subseteq \clSc$.
Since $\clSminOO = \clSminOX \cap \clSminXO$ and $\clSc = \clSc \cap \clSc$, Lemma~\ref{lem:intersection} and Proposition~\ref{prop:LR:SminOX}\ref{prop:LR:SminOX:SminOX-R}, together with Lemma~\ref{lem:neg-dual}, imply that $\clSminOO C \subseteq \clSminOO$.

Conversely, assume that $C \nsubseteq \clSc$.
Then $C \nsubseteq \clOI$ or $C \nsubseteq \clS$.
If $C \nsubseteq \clOI$, then $\clSminOO C \nsubseteq \clEq$ by Lemma~\ref{lem:KCCK}\ref{KC:OI-UWneg-notin-Eq}.
Otherwise we have $C \nsubseteq \clS$ and $C \subseteq \clOI$; in this case we have $\clSminOO C \nsubseteq \clSmin$ by Lemma~\ref{lem:KCCK}\ref{KC:S+OI-SminOO-notin-Smin}.
Consequently, $\clSminOO C \nsubseteq \clEq \cap \clSmin = \clSminOO$.

\ref{prop:LR:SminOO:SminOO-L}
Assume first that $C \subseteq \clU$.
Since $\clSminOO = \clSminOX \cap \clSminXO$ and $\clU = \clU \cap \clU$, Lemma~\ref{lem:intersection} and Proposition~\ref{prop:LR:SminOX}\ref{prop:LR:SminOX:SminOX-L}, together with Lemma~\ref{lem:neg-dual}, imply that $C \clSminOO \subseteq \clSminOO$.
Conversely, if $C \nsubseteq \clU$, then $C \clSminOO \nsubseteq \clSmin \supseteq \clSminOO$ by Lemma~\ref{lem:KCCK}\ref{CK:U-UWneg-notin-Smin}.
\end{proof}

\begin{proposition}
\label{prop:LR:TcUCO}
Let $C$ be a clone.
\begin{enumerate}[label=\upshape{(\roman*)}]
\item\label{prop:LR:TcUCO:TcUCO-R}
$(\clTcUCO) C \subseteq \clTcUCO$ if and only if $C \subseteq \clTcU$.
\item\label{prop:LR:TcUCO:TcUCO-L}
$C (\clTcUCO) \subseteq \clTcUCO$ if and only if $C \subseteq \clMU$.
\end{enumerate}
\end{proposition}

\begin{proof}
\ref{prop:LR:TcUCO:TcUCO-R}
Assume first that $C \subseteq \clTcU$.
Since $\clTcUCO = \clU \cap (\clOICO)$ and $\clTcU = \clU \cap \clOI$, Lemma~\ref{lem:intersection} and Propositions~\ref{prop:LR:clones} and \ref{prop:LR:OICO}\ref{prop:LR:OICO:OICO-R} imply that $(\clTcUCO) C \subseteq \clTcUCO$.

Conversely, assume that $C \nsubseteq \clTcU$.
Then $C \nsubseteq \clOI$ or $C \nsubseteq \clU$.
If $C \nsubseteq \clOI$, then $(\clTcUCO) C \nsubseteq \clNeq$ by Lemma~\ref{lem:KCCK}\ref{KC:OI-McU-notin-Neq}.
If $C \nsubseteq \clU$, then $(\clTcUCO) C \nsubseteq \clU$ by Lemma~\ref{lem:KCCK}\ref{KC:U-SM-notin-U}.
Consequently, $(\clTcUCO) C \nsubseteq \clNeq \cap \clU = \clTcU$.

\ref{prop:LR:TcUCO:TcUCO-L}
Assume first that $C \subseteq \clMU$.
Since $\clTcUCO = \clU \cap (\clOICO)$ and $\clMU = \clU \cap \clMo$, Lemma~\ref{lem:intersection} and Propositions~\ref{prop:LR:clones} and \ref{prop:LR:OICO}\ref{prop:LR:OICO:OICO-L} imply that $C (\clTcUCO) \subseteq \clTcUCO$.
Conversely, if $C \nsubseteq \clMU$, then $C (\clTcUCO) \nsubseteq \clSminOICO \supseteq \clTcUCO$ by Lemma~\ref{lem:KCCK}\ref{CK:MU-TcUCO-notin-SminOICO}.
\end{proof}

\begin{proposition}
\label{prop:LR:UOO}
Let $C$ be a clone.
\begin{enumerate}[label=\upshape{(\roman*)}]
\item\label{prop:LR:UOO:UOO-R}
$\clUOO C \subseteq \clUOO$ if and only if $C \subseteq \clTcU$.
\item\label{prop:LR:UOO:UOO-L}
$C \clUOO \subseteq \clUOO$ if and only if $C \subseteq \clU$.
\end{enumerate}
\end{proposition}

\begin{proof}
\ref{prop:LR:UOO:UOO-R}
Assume first that $C \subseteq \clTcU$.
Since $\clUOO = \clU \cap \clOO$ and $\clTcU = \clU \cap \clOI$, Lemma~\ref{lem:intersection} and Propositions~\ref{prop:LR:clones} and \ref{prop:LR:OO}\ref{prop:LR:OO:OO-R} imply that $\clUOO C \subseteq \clUOO$.

Conversely, assume that $C \nsubseteq \clTcU$.
Then $C \nsubseteq \clOI$ or $C \nsubseteq \clU$.
If $C \nsubseteq \clOI$, then $\clUOO C \nsubseteq \clEq$ by Lemma~\ref{lem:KCCK}\ref{KC:OI-UWneg-notin-Eq}.
If $C \nsubseteq \clU$, then $\clUOO C \nsubseteq \clU$ by Lemma~\ref{lem:KCCK}\ref{KC:U-UWneg-notin-U}.
Consequently, $\clUOO C \nsubseteq \clEq \cap \clU = \clTcU$.

\ref{prop:LR:UOO:UOO-L}
Assume first that $C \subseteq \clU$.
Since $\clUOO = \clU \cap \clOO$ and $\clU = \clU \cap \clOX$, Lemma~\ref{lem:intersection} and Propositions~\ref{prop:LR:clones} and \ref{prop:LR:OO}\ref{prop:LR:OO:OO-L} imply that $C \clUOO \subseteq \clUOO$.
Conversely, if $C \nsubseteq \clU$, then $C \clUOO \nsubseteq \clSmin \supseteq \clUOO$ by Lemma~\ref{lem:KCCK}\ref{CK:U-UWneg-notin-Smin}.
\end{proof}

\begin{proposition}
\label{prop:LR:UWneg}
Let $C$ be a clone.
\begin{enumerate}[label=\upshape{(\roman*)}]
\item\label{prop:LR:UWneg:UWneg-R}
$(\clUWneg) C \subseteq \clUWneg$ if and only if $C \subseteq \clSM$.
\item\label{prop:LR:UWneg:UWneg-L}
$C (\clUWneg) \subseteq \clUWneg$ if and only if $C \subseteq \clU$.
\end{enumerate}
\end{proposition}

\begin{proof}
\ref{prop:LR:UWneg:UWneg-R}
Assume first that $C \subseteq \clSM$.
Since $\clUWneg = \clU \cap \clWneg$ and $\clSM = \clU \cap \clW$, Lemma~\ref{lem:intersection} and Proposition~\ref{prop:LR:clones} imply that $(\clUWneg) C \subseteq \clUWneg$.

Conversely, assume that $C \nsubseteq \clSM$.
Then $C \nsubseteq \clS$ or $C \nsubseteq \clM$.
If $C \nsubseteq \clS$, then $(\clUWneg) C \nsubseteq \clUWneg$ by Lemma~\ref{lem:KCCK}\ref{KC:S-UWneg-notin-UWneg}.
Otherwise we have $C \nsubseteq \clM$ and $C \subseteq \clS$; in this case we have $(\clUWneg) C \nsubseteq \clU$ by Lemma~\ref{lem:KCCK}\ref{KC:M+S-UWneg-notin-U}.
Consequently, $(\clUWneg) C \nsubseteq (\clUWneg) \cap \clU = \clUWneg$.

\ref{prop:LR:UWneg:UWneg-L}
Assume first that $C \subseteq \clU$.
Since $\clUWneg = \clU \cap \clWneg$ and $\clU = \clU \cap \clU$, Lemma~\ref{lem:intersection} and Proposition~\ref{prop:LR:clones} imply that $C (\clUWneg) \subseteq \clUWneg$.
Conversely, if $C \nsubseteq \clU$, then $C (\clUWneg) \nsubseteq \clSmin \supseteq \clUWneg$ by Lemma~\ref{lem:KCCK}\ref{CK:U-UWneg-notin-Smin}.
\end{proof}

\begin{proposition}
\label{prop:LR:R}
Let $C$ be a clone.
\begin{enumerate}[label=\upshape{(\roman*)}]
\item\label{prop:LR:R:R-R}
$\clRefl C \subseteq \clRefl$ if and only if $C \subseteq \clS$.
\item\label{prop:LR:R:R-L}
$C \clRefl \subseteq \clRefl$ for every clone $C$.
\end{enumerate}
\end{proposition}

\begin{proof}
\ref{prop:LR:R:R-R}
Assume first that $C \subseteq \clS$.
Let $f \in \clRefl^{(n)}$ and $g_1, \dots, g_n \in C^{(m)}$.
It holds that, for all $\vect{a} \in \{0,1\}^m$,
\[
\begin{split}
\lhs
f(g_1, \dots, g_n)(\vect{a})
= f(g_1(\vect{a}), \dots, g_n(\vect{a}))
= f(\overline{g_1(\overline{\vect{a}})}, \dots, \overline{g_n(\overline{\vect{a}})})
\\ &
= f(g_1(\overline{\vect{a}}), \dots, g_n(\overline{\vect{a}}))
= f(g_1, \dots, g_n)(\overline{\vect{a}}),
\end{split}
\]
where the second equality holds because $g_i \in \clS$ and the third equality holds because $f \in \clRefl$.
Thus $f(g_1, \dots, g_n) \in \clRefl$.
Therefore $\clRefl C \subseteq \clRefl$.
Conversely, if $C \nsubseteq \clS$, then $\clRefl C \nsubseteq \clRefl$ by Lemma~\ref{lem:KCCK}\ref{KC:S-ROO-notin-R}.

\ref{prop:LR:R:R-L}
For any $f \in \clAll^{(n)}$ and $g_1, \dots, g_n \in \clRefl^{(m)}$ it holds that
\[
f(g_1, \dots, g_n)(\vect{a}) = f(g_1(\vect{a}), \dots, g_n(\vect{a})) = f(g_1(\overline{\vect{a}}), \dots, g_n(\overline{\vect{a}})) = f(g_1, \dots, g_n)(\overline{\vect{a}}),
\]
for all $\vect{a} \in \{0,1\}^m$.
Therefore $C \clRefl \subseteq \clRefl$ for any clone $C$.
\end{proof}

\begin{proposition}
\label{prop:LR:ROOC}
Let $C$ be a clone.
\begin{enumerate}[label=\upshape{(\roman*)}]
\item\label{prop:LR:ROOC:ROOC-R}
$(\clReflOOC) C \subseteq \clReflOOC$ if and only if $C \subseteq \clSc$.
\item\label{prop:LR:ROOC:ROOC-L}
$C (\clReflOOC) \subseteq \clReflOOC$ if and only if $C \subseteq \clM$.
\end{enumerate}
\end{proposition}

\begin{proof}
\ref{prop:LR:ROOC:ROOC-R}
Assume first that $C \subseteq \clSc$.
Since $\clReflOOC = (\clOOC) \cap \clRefl$ and $\clSc = \clOI \cap \clS$, Lemma~\ref{lem:intersection} and Propositions~\ref{prop:LR:OOC}\ref{prop:LR:OOC:OOC-R} and \ref{prop:LR:R}\ref{prop:LR:R:R-R} imply that $(\clReflOOC) C \subseteq \clReflOOC$.

Conversely, assume that $C \nsubseteq \clSc$.
Then $C \nsubseteq \clOI$ or $C \nsubseteq \clS$.
If $C \nsubseteq \clOI$, then $(\clReflOOC) C \nsubseteq \clOOC$ by Lemma~\ref{lem:KCCK}\ref{KC:OI-ROO-notin-OOC}.
If $C \nsubseteq \clS$, then $(\clReflOOC) C \nsubseteq \clRefl$ by Lemma~\ref{lem:KCCK}\ref{KC:S-ROO-notin-R}.
Consequently, $(\clReflOOC) C \nsubseteq (\clOOC) \cap \clRefl = \clReflOOC$.

\ref{prop:LR:ROOC:ROOC-L}
Assume first that $C \subseteq \clM$.
Since $\clReflOOC = (\clOOC) \cap \clRefl$ and $\clM = \clM \cap \clAll$, Lemma~\ref{lem:intersection} and Propositions~\ref{prop:LR:OOC}\ref{prop:LR:OOC:OOC-L} and \ref{prop:LR:R}\ref{prop:LR:R:R-L} imply that $C (\clReflOOC) \subseteq \clReflOOC$.
Conversely, if $C \nsubseteq \clM$, then $C (\clReflOOC) \nsubseteq \clOXC \supseteq \clReflOOC$ by Lemma~\ref{lem:KCCK}\ref{CK:M-ROOC-notin-OXC}.
\end{proof}

\begin{proposition}
\label{prop:LR:ROO}
Let $C$ be a clone.
\begin{enumerate}[label=\upshape{(\roman*)}]
\item\label{prop:LR:ROO:ROO-R}
$\clReflOO C \subseteq \clReflOO$ if and only if $C \subseteq \clSc$.
\item\label{prop:LR:ROO:ROO-L}
$C \clReflOO \subseteq \clReflOO$ if and only if $C \subseteq \clOX$.
\end{enumerate}
\end{proposition}

\begin{proof}
\ref{prop:LR:ROO:ROO-R}
Assume first that $C \subseteq \clSc$.
Since $\clReflOO = \clOO \cap \clRefl$ and $\clSc = \clOI \cap \clS$, Lemma~\ref{lem:intersection} and Propositions~\ref{prop:LR:OO}\ref{prop:LR:OO:OO-R} and \ref{prop:LR:R}\ref{prop:LR:R:R-R} imply that $\clReflOO C \subseteq \clReflOO$.

Conversely, assume that $C \nsubseteq \clSc$.
Then $C \nsubseteq \clOI$ or $C \nsubseteq \clS$.
If $C \nsubseteq \clOI$, then $\clReflOO C \nsubseteq \clOOC$ by Lemma~\ref{lem:KCCK}\ref{KC:OI-ROO-notin-OOC}.
If $C \nsubseteq \clS$, then $\clReflOO C \nsubseteq \clRefl$ by Lemma~\ref{lem:KCCK}\ref{KC:S-ROO-notin-R}.
Consequently, $\clReflOO C \nsubseteq (\clOOC) \cap \clRefl = \clReflOOC \supseteq \clReflOO$.

\ref{prop:LR:ROO:ROO-L}
Assume first that $C \subseteq \clOX$.
Since $\clReflOO = \clOO \cap \clRefl$ and $\clOX = \clOX \cap \clAll$, Lemma~\ref{lem:intersection} and Propositions~\ref{prop:LR:OO}\ref{prop:LR:OO:OO-L} and \ref{prop:LR:R}\ref{prop:LR:R:R-L} imply that $C \clReflOO \subseteq \clReflOO$.
Conversely, if $C \nsubseteq \clOX$, then $C \clReflOO \nsubseteq \clEiii \supseteq \clReflOO$ by Lemma~\ref{lem:KCCK}\ref{CK:OX-ROO-notin-Eiii}.
\end{proof}

\begin{proposition}
\label{prop:LR:const}
Let $C$ be a clone.
\begin{enumerate}[label=\upshape{(\roman*)}]
\item\label{prop:LR:const:C-R}
$\clVak C \subseteq \clVak$ for every clone $C$.
\item\label{prop:LR:const:C-L}
$C \clVak \subseteq \clVak$ for every clone $C$.
\item\label{prop:LR:const:C0-R}
$\clVako C \subseteq \clVako$ for every clone $C$.
\item\label{prop:LR:const:C0-L}
$C \clVako \subseteq \clVako$ if and only if $C \subseteq \clOX$.
\item\label{prop:LR:const:C1-R}
$\clVaki C \subseteq \clVaki$ for every clone $C$.
\item\label{prop:LR:const:C1-L}
$C \clVaki \subseteq \clVaki$ if and only if $C \subseteq \clXI$.
\end{enumerate}
\end{proposition}

\begin{proof}
Straightforward verification.
\end{proof}

\begin{proposition}
\label{prop:LR:empty}
$\clEmpty C \subseteq \clEmpty$ and $C \clEmpty \subseteq \clEmpty$  for every clone $C$.
\end{proposition}

\begin{proof}
Trivial.
\end{proof}

\begin{proof}[Proof of Theorem~\ref{thm:C1C2-stability}]
The theorem puts together Lemma~\ref{lem:neg-dual}
and
Propositions
\ref{prop:LR:clones},
\ref{prop:LR:Ei},
\ref{prop:LR:EqNeq},
\ref{prop:LR:OXC},
\ref{prop:LR:OOC},
\ref{prop:LR:OO},
\ref{prop:LR:OIC},
\ref{prop:LR:OICO},
\ref{prop:LR:Smin},
\ref{prop:LR:SminNeq},
\ref{prop:LR:SminOX},
\ref{prop:LR:SminOICO},
\ref{prop:LR:SminOI},
\ref{prop:LR:SminOO},
\ref{prop:LR:TcUCO},
\ref{prop:LR:UOO},
\ref{prop:LR:UWneg},
\ref{prop:LR:R},
\ref{prop:LR:ROOC},
\ref{prop:LR:ROO},
\ref{prop:LR:const},
\ref{prop:LR:empty}.
\end{proof}

With the help of Post's lattice and by reading off from Table~\ref{table:SM-stable}, we can determine for any pair $(C_1,C_2)$ of clones which $(\clIc,\clSM)$\hyp{}stable classes are $(C_1,C_2)$\hyp{}stable.
If $\clSM \subseteq C_2$, then any $(C_1,C_2)$\hyp{}stable class is $(\clIc,\clSM)$\hyp{}stable by Lemma~\ref{lem:stable-impl-stable}.
Therefore, in the case when $\clSM \subseteq C_2$, the $(C_1,C_2)$\hyp{}stable classes occur among the $(\clIc,\clSM)$\hyp{}stable ones and they can be easily picked out from Table~\ref{table:SM-stable}.
In particular, we have an explicit description of $(\clIc,C)$\hyp{}stable classes (``clonoids'' of Aichinger and Mayr~\cite{AicMay}) and $(C,C)$\hyp{}stable classes for $\clSM \subseteq C$.

\begin{corollary}
\leavevmode
\begin{enumerate}[label=\upshape{(\alph*)}]
\item
The $(\clIc,\clMcU)$\hyp{}stable classes are
$\clAll$, $\clEiio$, $\clEioi$, $\clEiii$, $\clEq$, $\clOXC$, $\clIXC$, $\clXOC$, $\clXIC$, $\clOX$, $\clIX$, $\clXO$, $\clXI$, $\clOOC$, $\clIIC$, $\clOO$, $\clII$, $\clOIC$, $\clIOC$, $\clOICO$, $\clIOCI$, $\clOICI$, $\clIOCO$, $\clOI$, $\clIO$, $\clSmin$, $\clSminOX$, $\clSminXO$, $\clSminOICO$, $\clSminIOCO$, $\clSminOI$, $\clSminIO$, $\clSminOO$, $\clM$, $\clMneg$, $\clMo$, $\clMoneg$, $\clMi$, $\clMineg$, $\clMc$, $\clMcneg$, $\clU$, $\clTcUCO$, $\clTcU$, $\clMU$, $\clMcU$, $\clWneg$, $\clTcWnegCO$, $\clTcWneg$, $\clMWneg$, $\clMcWneg$, $\clUOO$, $\clWnegOO$, $\clUWneg$, $\clRefl$, $\clReflOOC$, $\clReflIIC$, $\clReflOO$, $\clReflII$, $\clVak$, $\clVako$, $\clVaki$, $\clEmpty$.

\item
The $(\clIc,\clMU)$\hyp{}stable classes are
$\clAll$, $\clEiio$, $\clEioi$, $\clEiii$, $\clEq$, $\clOXC$, $\clIXC$, $\clXOC$, $\clXIC$, $\clOX$, $\clXO$, $\clOOC$, $\clIIC$, $\clOO$, $\clOIC$, $\clIOC$, $\clOICO$, $\clIOCO$, $\clSmin$, $\clSminOX$, $\clSminXO$, $\clSminOICO$, $\clSminIOCO$, $\clSminOO$, $\clM$, $\clMneg$, $\clMo$, $\clMineg$, $\clU$, $\clTcUCO$, $\clMU$, $\clWneg$, $\clTcWnegCO$, $\clMWneg$, $\clUOO$, $\clWnegOO$, $\clUWneg$, $\clRefl$, $\clReflOOC$, $\clReflIIC$, $\clReflOO$, $\clVak$, $\clVako$, $\clEmpty$.

\item
The $(\clIc,\clTcU)$\hyp{}stable classes are
$\clAll$, $\clEiii$, $\clEq$, $\clOX$, $\clIX$, $\clXO$, $\clXI$, $\clOO$, $\clII$, $\clOI$, $\clIO$, $\clSmin$, $\clSminOX$, $\clSminXO$, $\clSminOI$, $\clSminIO$, $\clSminOO$, $\clU$, $\clTcU$, $\clWneg$, $\clTcWneg$, $\clUOO$, $\clWnegOO$, $\clUWneg$, $\clRefl$, $\clReflOO$, $\clReflII$, $\clVak$, $\clVako$, $\clVaki$, $\clEmpty$.

\item
The $(\clIc,\clU)$\hyp{}stable classes are
$\clAll$, $\clEiii$, $\clEq$, $\clOX$, $\clXO$, $\clOO$, $\clSmin$, $\clSminOX$, $\clSminXO$, $\clSminOO$, $\clU$, $\clWneg$, $\clUOO$, $\clWnegOO$, $\clUWneg$, $\clRefl$, $\clReflOO$, $\clVak$, $\clVako$, $\clEmpty$.

\item
The $(\clIc,\clMcW)$\hyp{}stable classes are
$\clAll$, $\clEiio$, $\clEioi$, $\clEioo$, $\clEq$, $\clOXC$, $\clIXC$, $\clXOC$, $\clXIC$, $\clOX$, $\clIX$, $\clXO$, $\clXI$, $\clOOC$, $\clIIC$, $\clOO$, $\clII$, $\clOIC$, $\clIOC$, $\clOICO$, $\clIOCI$, $\clOICI$, $\clIOCO$, $\clOI$, $\clIO$, $\clSmaj$, $\clSmajIX$, $\clSmajXI$, $\clSmajIOCI$, $\clSmajOICI$, $\clSmajIO$, $\clSmajOI$, $\clSmajII$, $\clM$, $\clMneg$, $\clMo$, $\clMoneg$, $\clMi$, $\clMineg$, $\clMc$, $\clMcneg$, $\clUneg$, $\clTcUnegCI$, $\clTcUneg$, $\clMUneg$, $\clMcUneg$, $\clW$, $\clTcWCI$, $\clTcW$, $\clMW$, $\clMcW$, $\clUnegII$, $\clWII$, $\clWUneg$, $\clRefl$, $\clReflOOC$, $\clReflIIC$, $\clReflOO$, $\clReflII$, $\clVak$, $\clVako$, $\clVaki$, $\clEmpty$.

\item
The $(\clIc,\clMW)$\hyp{}stable classes are
$\clAll$, $\clEiio$, $\clEioi$, $\clEioo$, $\clEq$, $\clOXC$, $\clIXC$, $\clXOC$, $\clXIC$, $\clIX$, $\clXI$, $\clOOC$, $\clIIC$, $\clII$, $\clOIC$, $\clIOC$, $\clIOCI$, $\clOICI$, $\clSmaj$, $\clSmajIX$, $\clSmajXI$, $\clSmajIOCI$, $\clSmajOICI$, $\clSmajII$, $\clM$, $\clMneg$, $\clMoneg$, $\clMi$, $\clUneg$, $\clTcUnegCI$, $\clMUneg$, $\clW$, $\clTcWCI$, $\clMW$, $\clUnegII$, $\clWII$, $\clWUneg$, $\clRefl$, $\clReflOOC$, $\clReflIIC$, $\clReflII$, $\clVak$, $\clVaki$, $\clEmpty$.

\item
The $(\clIc,\clTcW)$\hyp{}stable classes are
$\clAll$, $\clEioo$, $\clEq$, $\clOX$, $\clIX$, $\clXO$, $\clXI$, $\clOO$, $\clII$, $\clOI$, $\clIO$, $\clSmaj$, $\clSmajIX$, $\clSmajXI$, $\clSmajIO$, $\clSmajOI$, $\clSmajII$, $\clUneg$, $\clTcUneg$, $\clW$, $\clTcW$, $\clUnegII$, $\clWII$, $\clWUneg$, $\clRefl$, $\clReflOO$, $\clReflII$, $\clVak$, $\clVako$, $\clVaki$, $\clEmpty$.

\item
The $(\clIc,\clW)$\hyp{}stable classes are
$\clAll$, $\clEioo$, $\clEq$, $\clIX$, $\clXI$, $\clII$, $\clSmaj$, $\clSmajIX$, $\clSmajXI$, $\clSmajII$, $\clUneg$, $\clW$, $\clUnegII$, $\clWII$, $\clWUneg$, $\clRefl$, $\clReflII$, $\clVak$, $\clVaki$, $\clEmpty$.

\item
The $(\clIc,\clSc)$\hyp{}stable classes are
$\clAll$, $\clEq$, $\clNeq$, $\clOX$, $\clIX$, $\clXO$, $\clXI$, $\clOO$, $\clII$, $\clOI$, $\clIO$, $\clS$, $\clSc$, $\clScneg$, $\clRefl$, $\clReflOO$, $\clReflII$, $\clVak$, $\clVako$, $\clVaki$, $\clEmpty$.

\item
The $(\clIc,\clS)$\hyp{}stable classes are
$\clAll$, $\clEq$, $\clNeq$, $\clS$, $\clRefl$, $\clVak$, $\clEmpty$.

\item
The $(\clIc,\clMc)$\hyp{}stable classes are
$\clAll$, $\clEioi$, $\clEiio$, $\clEq$, $\clOXC$, $\clIXC$, $\clXOC$, $\clXIC$, $\clOX$, $\clIX$, $\clXO$, $\clXI$, $\clOOC$, $\clIIC$, $\clOO$, $\clII$, $\clOIC$, $\clIOC$, $\clOICO$, $\clIOCI$, $\clOICI$, $\clIOCO$, $\clOI$, $\clIO$, $\clM$, $\clMneg$, $\clMo$, $\clMoneg$, $\clMi$, $\clMineg$, $\clMc$, $\clMcneg$, $\clRefl$, $\clReflOOC$, $\clReflIIC$, $\clReflOO$, $\clReflII$, $\clVak$, $\clVako$, $\clVaki$, $\clEmpty$.

\item
The $(\clIc,\clMo)$\hyp{}stable classes are
$\clAll$, $\clEioi$, $\clEiio$, $\clEq$, $\clOXC$, $\clIXC$, $\clXOC$, $\clXIC$, $\clOX$, $\clXO$, $\clOOC$, $\clIIC$, $\clOO$, $\clOIC$, $\clIOC$, $\clOICO$, $\clIOCO$, $\clM$, $\clMneg$, $\clMo$, $\clMineg$, $\clRefl$, $\clReflOOC$, $\clReflIIC$, $\clReflOO$, $\clVak$, $\clVako$, $\clEmpty$.

\item
The $(\clIc,\clMi)$\hyp{}stable classes are
$\clAll$, $\clEioi$, $\clEiio$, $\clEq$, $\clOXC$, $\clIXC$, $\clXOC$, $\clXIC$, $\clIX$, $\clXI$, $\clOOC$, $\clIIC$, $\clII$, $\clOIC$, $\clIOC$, $\clIOCI$, $\clOICI$, $\clM$, $\clMneg$, $\clMoneg$, $\clMi$, $\clRefl$, $\clReflOOC$, $\clReflIIC$, $\clReflII$, $\clVak$, $\clVaki$, $\clEmpty$.

\item
The $(\clIc,\clM)$\hyp{}stable classes are
$\clAll$, $\clEioi$, $\clEiio$, $\clEq$, $\clOXC$, $\clIXC$, $\clXOC$, $\clXIC$, $\clOOC$, $\clIIC$, $\clOIC$, $\clIOC$, $\clM$, $\clMneg$, $\clRefl$, $\clReflOOC$, $\clReflIIC$, $\clVak$, $\clEmpty$.

\item
The $(\clIc,\clOI)$\hyp{}stable classes are
$\clAll$, $\clEq$, $\clOX$, $\clIX$, $\clXO$, $\clXI$, $\clOO$, $\clII$, $\clOI$, $\clIO$, $\clRefl$, $\clReflOO$, $\clReflII$, $\clVak$, $\clVako$, $\clVaki$, $\clEmpty$.

\item
The $(\clIc,\clOX)$\hyp{}stable classes are
$\clAll$, $\clEq$, $\clOX$, $\clXO$, $\clOO$, $\clRefl$, $\clReflOO$, $\clVak$, $\clVako$, $\clEmpty$.

\item
The $(\clIc,\clXI)$\hyp{}stable classes are
$\clAll$, $\clEq$, $\clIX$, $\clXI$, $\clII$, $\clRefl$, $\clReflII$, $\clVak$, $\clVaki$, $\clEmpty$.

\item
The $(\clIc,\clAll)$\hyp{}stable classes are
$\clAll$, $\clEq$, $\clRefl$, $\clVak$, $\clEmpty$.
\end{enumerate}
\end{corollary}

\begin{corollary}
\leavevmode
\begin{enumerate}[label=\upshape{(\alph*)}]
\item
The $\clSM$\hyp{}stable classes are
precisely the 93 $(\clIc,\clSM)$\hyp{}stable classes.

\item
The $\clMcU$\hyp{}stable classes are
$\clAll$, $\clEiio$, $\clEioi$, $\clEiii$, $\clEq$, $\clOXC$, $\clIXC$, $\clXOC$, $\clXIC$, $\clOX$, $\clIX$, $\clXO$, $\clXI$, $\clOOC$, $\clIIC$, $\clOO$, $\clII$, $\clOIC$, $\clIOC$, $\clOICO$, $\clIOCI$, $\clOICI$, $\clIOCO$, $\clOI$, $\clIO$, $\clM$, $\clMneg$, $\clMo$, $\clMoneg$, $\clMi$, $\clMineg$, $\clMc$, $\clMcneg$, $\clU$, $\clTcUCO$, $\clTcU$, $\clMU$, $\clMcU$, $\clUOO$, $\clVak$, $\clVako$, $\clVaki$, $\clEmpty$.

\item
The $\clMU$\hyp{}stable classes are
$\clAll$, $\clOXC$, $\clIXC$, $\clOX$, $\clM$, $\clMneg$, $\clMo$, $\clU$, $\clMU$, $\clVak$, $\clVako$, $\clEmpty$.

\item
The $\clTcU$\hyp{}stable classes are
$\clAll$, $\clEiii$, $\clEq$, $\clOX$, $\clIX$, $\clXO$, $\clXI$, $\clOO$, $\clII$, $\clOI$, $\clIO$, $\clU$, $\clTcU$, $\clUOO$, $\clVak$, $\clVako$, $\clVaki$, $\clEmpty$.

\item
The $\clU$\hyp{}stable classes are
$\clAll$, $\clOX$, $\clU$, $\clVak$, $\clVako$, $\clEmpty$.

\item
The $\clMcW$\hyp{}stable classes are
$\clAll$, $\clEiio$, $\clEioi$, $\clEioo$, $\clEq$, $\clOXC$, $\clIXC$, $\clXOC$, $\clXIC$, $\clOX$, $\clIX$, $\clXO$, $\clXI$, $\clOOC$, $\clIIC$, $\clOO$, $\clII$, $\clOIC$, $\clIOC$, $\clOICO$, $\clIOCI$, $\clOICI$, $\clIOCO$, $\clOI$, $\clIO$, $\clM$, $\clMneg$, $\clMo$, $\clMoneg$, $\clMi$, $\clMineg$, $\clMc$, $\clMcneg$, $\clW$, $\clTcWCI$, $\clTcW$, $\clMW$, $\clMcW$, $\clWII$, $\clVak$, $\clVako$, $\clVaki$, $\clEmpty$.

\item
The $\clMW$\hyp{}stable classes are
$\clAll$, $\clXOC$, $\clXIC$, $\clXI$, $\clM$, $\clMneg$, $\clMi$, $\clW$, $\clMW$, $\clVak$, $\clVaki$, $\clEmpty$.

\item
The $\clTcW$\hyp{}stable classes are
$\clAll$, $\clEioo$, $\clEq$, $\clOX$, $\clIX$, $\clXO$, $\clXI$, $\clOO$, $\clII$, $\clOI$, $\clIO$, $\clW$, $\clTcW$, $\clWII$, $\clVak$, $\clVako$, $\clVaki$, $\clEmpty$.

\item
The $\clW$\hyp{}stable classes are
$\clAll$, $\clXI$, $\clW$, $\clVak$, $\clVaki$, $\clEmpty$.

\item
The $\clSc$\hyp{}stable classes are
$\clAll$, $\clEq$, $\clNeq$, $\clOX$, $\clIX$, $\clXO$, $\clXI$, $\clOO$, $\clII$, $\clOI$, $\clIO$, $\clS$, $\clSc$, $\clScneg$, $\clRefl$, $\clReflOO$, $\clReflII$, $\clVak$, $\clVako$, $\clVaki$, $\clEmpty$.

\item
The $\clS$\hyp{}stable classes are
$\clAll$, $\clS$, $\clRefl$, $\clVak$, $\clEmpty$.

\item
The $\clMc$\hyp{}stable classes are
$\clAll$, $\clEioi$, $\clEiio$, $\clEq$, $\clOXC$, $\clIXC$, $\clXOC$, $\clXIC$, $\clOX$, $\clIX$, $\clXO$, $\clXI$, $\clOOC$, $\clIIC$, $\clOO$, $\clII$, $\clOIC$, $\clIOC$, $\clOICO$, $\clIOCI$, $\clOICI$, $\clIOCO$, $\clOI$, $\clIO$, $\clM$, $\clMneg$, $\clMo$, $\clMoneg$, $\clMi$, $\clMineg$, $\clMc$, $\clMcneg$, $\clVak$, $\clVako$, $\clVaki$, $\clEmpty$.

\item
The $\clMo$\hyp{}stable classes are
$\clAll$, $\clOXC$, $\clIXC$, $\clOX$, $\clM$, $\clMneg$, $\clMo$, $\clVak$, $\clVako$, $\clEmpty$.

\item
The $\clMi$\hyp{}stable classes are
$\clAll$, $\clXOC$, $\clXIC$, $\clXI$, $\clM$, $\clMneg$, $\clMi$, $\clVak$, $\clVaki$, $\clEmpty$.

\item
The $\clM$\hyp{}stable classes are
$\clAll$, $\clM$, $\clMneg$, $\clVak$, $\clEmpty$.

\item
The $\clOI$\hyp{}stable classes are
$\clAll$, $\clEq$, $\clOX$, $\clIX$, $\clXO$, $\clXI$, $\clOO$, $\clII$, $\clOI$, $\clIO$, $\clVak$, $\clVako$, $\clVaki$, $\clEmpty$.

\item
The $\clOX$\hyp{}stable classes are
$\clAll$, $\clOX$, $\clVak$, $\clVako$, $\clEmpty$.

\item
The $\clXI$\hyp{}stable classes are
$\clAll$, $\clXI$, $\clVak$, $\clVaki$, $\clEmpty$.

\item
The $\clAll$\hyp{}stable classes are
$\clAll$, $\clVak$, $\clEmpty$.
\end{enumerate}
\end{corollary}

It is noteworthy that $(\clIc,\clSM)$\hyp{}stability is equivalent to $\clSM$\hyp{}stability.
This is explained by the following proposition, which also provides other equivalences for clones containing $\mu$.

\begin{proposition}
\label{prop:simplify}
\leavevmode
\begin{enumerate}[label=\upshape{(\roman*)}]
\item\label{mu-in-out}
For any $f \in \clAll$, we have $f \ast \mu = \mu(f_{\sigma_1}, f_{\sigma_2}, f_{\sigma_3})$, where, for $i \in \nset{3}$, $\sigma_i \colon \nset{n} \to \nset{n+2}$, $1 \mapsto i$, $j \mapsto j + 2$ for $2 \leq j \leq n$.

\item\label{C1C2-C1C2'}
Let $G \subseteq \clAll$, let $C_1 := \clonegen{G \cup \{\mu\}}$, $C'_1 := \clonegen{G}$, and let $C_2$ be a clone containing $\mu$.
Then a class $F \subseteq \clAll$ is $(C_1,C_2)$\hyp{}stable if and only if it is $(C'_1,C_2)$\hyp{}stable.

\item\label{SM-simplify}
The following are equivalent for a class $F \subseteq \clAll$.
\begin{enumerate}[label=\upshape{(\alph*)}]
\item $F$ is $\clSM$\hyp{}stable.
\item $F$ is $(\clIc,\clSM)$\hyp{}stable.
\item $F$ is minor\hyp{}closed and $\mu(f,g,h) \in F$ whenever $f, g, h \in F$.
\end{enumerate}
\end{enumerate}
\end{proposition}

\begin{proof}
\ref{mu-in-out}
It is easy to verify that
\[
\begin{split}
\lhs
f(\mu(a_1, a_2, a_3), a_4, \dots, a_{n+2})
\\ &
= \mu(f(a_1, a_4, \dots, a_{n+2}), f(a_2, a_4, \dots, a_{n+2}), f(a_3, a_4, \dots, a_{n+2}))
\end{split}
\]
holds for all $a_1, \dots, a_{n+2} \in \{0,1\}$.
For, if $a_1 = a_2 = a_3$, then both left and right sides are equal to $f(a_1, a_4, \dots, a_{n+2})$.
If $a_1$, $a_2$, and $a_3$ are not all equal, then two of them must be equal while the third is distinct from the other two.
Assume, without loss of generality, that $a_1 = a_2 \neq a_3$.
Then $\mu(a_1, a_2, a_3) = a_1$, so the left side equals $f(a_1, a_4, \dots, a_{n+2})$.
Since $f(a_1, a_4, a_5, \dots, a_{n+2}) = f(a_2, a_4, a_5, \dots, a_{n+2})$, we see that also the right side equals $f(a_1, a_4, \dots, a_{n+2})$.

\ref{C1C2-C1C2'}
Since $C'_1 \subseteq C_1$, stability under right composition with $C_1$ implies stability under right composition with $C'_1$.
Assume now that $F$ is $(C'_1,C_2)$\hyp{}stable.
By Lemma~\ref{lem:right-stab-gen}, $F$ is minor\hyp{}closed and $f \ast g \in F$ whenever $f \in F$ and $g \in G$.
Moreover, $f \ast \mu = \mu(f_{\sigma_1},f_{\sigma_2},f_{\sigma_3})$, where $f_{\sigma_1}$, $f_{\sigma_2}$, $f_{\sigma_3}$ are the minors of $f$ specified in part~\ref{mu-in-out}.
Since $F$ is minor\hyp{}closed, we have $f_{\sigma_1}, f_{\sigma_2}, f_{\sigma_3} \in F$.
By our assumption, $\mu \in C_2$, and since $F$ is stable under left composition with $C_2$, it follows that $\mu(f_{\sigma_1},f_{\sigma_2},f_{\sigma_3}) \in F$.
It follows from Lemma~\ref{lem:right-stab-gen} that $F$ is stable under right composition with $C_1$.

\ref{SM-simplify}
This is a consequence of part~\ref{C1C2-C1C2'} and Lemma~\ref{lem:left-stab-gen}.
\end{proof}


\section{Concluding remarks}
\label{sec:concluding}

The results of this paper cover only a part of the clones $C$ of Boolean functions for which there are only a finite number of $(\clIc,C)$\hyp{}stable classes
(see Theorem~\ref{thm:Sparks}\ref{thm:Sparks:finite}).
The description of $(\clIc,C)$\hyp{}stable classes for the other clones $C$ containing a near\hyp{}unanimity operation (i.e., the clones $\clUk{k}$, $\clTcUk{k}$, $\clMUk{k}$, $\clMcUk{k}$, $\clWk{k}$, $\clTcWk{k}$, $\clMWk{k}$, $\clMcWk{k}$ for $k \geq 3$)
remains a topic for further research.


\section*{Acknowledgments}

The author would like to thank Miguel Couceiro and Sebastian Kreinecker for inspiring discussions.

This work is funded by National Funds through the FCT -- Funda\c{c}\~ao para a Ci\^encia e a Tecnologia, I.P., under the scope of the project UIDB/00297/2020 (Center for Mathematics and Applications).


\end{document}